\documentclass[12pt]{amsart}
              
\usepackage[marginratio=1:1,height=660pt,width=480pt,tmargin=70pt]{geometry}

\usepackage{amssymb,latexsym,amsmath,amsthm,amscd,tensor}
\usepackage{mathrsfs}
\usepackage{graphicx}
 \usepackage{mathrsfs}
  
\usepackage{fancyhdr}

\usepackage{etoolbox}

\appto\appendix{\addtocontents{toc}{\protect\setcounter{tocdepth}{1}}}
\appto\listoffigures{\addtocontents{lof}{\protect\setcounter{tocdepth}{1}}}
\appto\listoftables{\addtocontents{lot}{\protect\setcounter{tocdepth}{1}}}

\newtheorem{theorem}{Theorem}[section]

\newtheorem{proposition}[theorem]{Proposition}
\newtheorem{corollary}[theorem]{Corollary}

\theoremstyle{definition}
\newtheorem{definition}[theorem]{Definition}

\theoremstyle{remark}
\newtheorem{remark}[theorem]{Remark}
\usepackage{booktabs}

\usepackage[small,nohug]{diagrams}        

\diagramstyle[labelstyle=\scriptstyle]

\usepackage{color} 
\usepackage{xcolor}
 \definecolor{carmine}{rgb}{0.59, 0.0, 0.09}
\definecolor{mediumpersianblue}{rgb}{0.0, 0.4, 0.65}
\definecolor{persianplum}{rgb}{0.44, 0.11, 0.11}
\usepackage{amsmath}
\usepackage{amsfonts}
\usepackage{amssymb}

\usepackage[colorlinks=true,
            linkcolor=persianplum, 
            urlcolor=olive,
            citecolor=mediumpersianblue, 
            backref=page]{hyperref}

\makeatletter
\renewcommand*\env@matrix[1][*\c@MaxMatrixCols c]{%
  \hskip -\arraycolsep
  \let\@ifnextchar\new@ifnextchar
  \array{#1}}
\makeatother
\newcommand{\be}{\begin{equation}}

\newcommand{\ee}{\end{equation}}

\newcommand{\al}{\alpha}
\newcommand{\bal}{\bar{\alpha}{}}

\newcommand{\bfv}{\mathbf{v}}

\newcommand{\dz}{\wedge}

\newcommand{\ba}{\begin{array}}

\newcommand{\ea}{\end{array}}

\newcommand{\beq}{\begin{eqnarray}}

\newcommand{\eeq}{\end{eqnarray}}

\newtheorem{lm}{lemma}

\newtheorem{thee}{theorem}

\newtheorem{proo}{proposition}

\newtheorem{co}{corollary}

\newtheorem{rem}{remark}

\newtheorem{deff}{definition}

\newcommand{\bd}{\begin{deff}}

\newcommand{\ed}{\end{deff}}

\newcommand{\bl}{\begin{lm}}

\newcommand{\el}{\end{lm}}

\newcommand{\bp}{\begin{proo}}

\newcommand{\ep}{\end{proo}}

\newcommand{\bt}{\begin{thee}}

\newcommand{\et}{\end{thee}}

\newcommand{\bc}{\begin{co}}

\newcommand{\ec}{\end{co}}

\newcommand{\brm}{\begin{rem}}

\newcommand{\erm}{\end{rem}}

\newcommand{\der}{{\rm d}}

\newcommand{\RR}{\mathbb{R}}

\newcommand{\PP}{\mathbb{P}}
\newcommand{\exd}{\mathrm{d}}

\newcommand{\cG}{\mathcal{G}}

\newcommand{\cF}{\mathcal{F}}
\newcommand{\cU}{\mathcal{U}}
\newcommand{\cE}{\mathcal{E}}

\newcommand{\cN}{\mathcal{N}}
\newcommand{\cC}{{\cN^*_+}}

\newcommand{\cO}{\mathcal{O}}
\newcommand{\cQ}{\mathcal{Q}}

\newcommand{\cD}{\mathscr{D}}
\newcommand{\cH}{\mathscr{H}}

\newcommand{\cA}{\mathcal{A}}
\newcommand{\cB}{\mathcal{B}}

\newcommand{\cZ}{\mathcal{Z}}

\newcommand{\rT}{\mathrm{T}}

\newcommand{\HH}{\mathbf{H}}

\newcommand{\rG}{\mathbf{G}}
\newcommand{\rP}{\mathbf{P}}
\newcommand{\fg}{\mathfrak{g}}

\newcommand{\bW}{\mathbf{W}}
\newcommand{\bG}{\mathbf{G}}

\newcommand{\w}{{\,{\wedge}\;}}

\usepackage{amssymb}
\usepackage{amscd}

\newcommand{\Rho}{{\mbox{\sf P}}}

\newcommand{\bma}{\begin{pmatrix}}
\newcommand{\ema}{\end{pmatrix}}

\newcommand\m[1]{\begin{pmatrix}#1\end{pmatrix}}

\newcommand{\rpl}                         
{\mbox{$
\begin{picture}(12.7,8)(-.5,-1)
\put(0,0.2){$+$}
\put(4.2,2.8){\oval(8,8)[r]}
\end{picture}$}}

\newcommand{\lpl}                         
{\mbox{$
\begin{picture}(12.7,8)(-.5,-1)
\put(2,0.2){$+$}
\put(6.2,2.8){\oval(8,8)[l]}
\end{picture}$}}

\newcommand{\id}{\operatorname{id}}
\newcommand{\Ker}{\operatorname{Ker}}

\newcommand{\Span}{\mathrm{Span}}

\renewcommand{\cO}{\mathcal{N}}

\newcommand{\sL}{\mathscr{L}}

\usepackage{float}

\makeatletter
\renewcommand*{\p@section}{\S\,}
\renewcommand*{\p@subsection}{\S\,}
\renewcommand*{\p@subsubsection}{\S\,}
 
\makeatother

\DeclareFontFamily{U}{MnSymbolC}{}
\DeclareSymbolFont{MnSyC}{U}{MnSymbolC}{m}{n}
\DeclareFontShape{U}{MnSymbolC}{m}{n}{
    <-6>  MnSymbolC5
   <6-7>  MnSymbolC6
   <7-8>  MnSymbolC7
   <8-9>  MnSymbolC8
   <9-10> MnSymbolC9
  <10-12> MnSymbolC10
  <12->   MnSymbolC12}{}
\DeclareMathSymbol{\im}{\mathbin}{MnSyC}{'270}

\newcommand{\bbS}{\mathbb{S}}
\newcommand{\bbR}{\mathbb{R}}

\newcommand{\slg}{\mathbf{SL}}
\newcommand{\glg}{\mathbf{GL}}

\newcommand{\sla}{\mathfrak{sl}}
\newcommand{\gla}{\mathfrak{gl}}

\newcommand{\SOtt}{\mathbf{SO}_{2,2}}
\newcommand{\sott}{\mathfrak{so}_{2,2}}

\numberwithin{equation}{section}
  
\newcommand{\ttf}{(2,3,5)}
\newcommand{\W}{{Weyl}}

\begin{document}

\title{para-K\"ahler-Einstein  4-manifolds and \\
non-integrable twistor distributions }

\author{Gil Bor, Omid Makhmali${}^*$, and Pawe{\l} Nurowski}

 \address{\newline Gil Bor\\\newline
 Centro de Intevetigaci\'on en Matem\'aticas (CIMAT), Guanajuato, Mexico\\\newline
   \textit{Email address: }{\href{mailto:gil@cimat.mx}{\texttt{gil@cimat.mx}}}\\\newline\newline
   Omid Makhmali\\\newline
Center for Theoretical Physics, PAS, Al. Lotnik\'ow 32\slash46, 02-668 Warszawa, Poland \\\newline
    \textit{Email address: }{\href{mailto:omakhmali@cft.edu.pl}{\texttt{omakhmali@cft.edu.pl}}}\\\newline\newline
Pawe{\l} Nurowski\\\newline
Center for Theoretical Physics, PAS, Al. Lotnik\'ow 32\slash46, 02-668 Warszawa, Poland \\\newline
\textit{Email address: }{\href{mailto:nurowski@fuw.edu.pl}{\texttt{nurowski@fuw.edu.pl}}}
 }

\date{\today}
 

\maketitle
\footnotetext{ Data sharing not applicable to this article as no datasets were generated or analysed during the current study.}

\begin{abstract}
  We study  the local geometry of 4-manifolds equipped with a \emph{para-K\"ahler-Einstein} (pKE)  metric,  a special type of  split-signature pseudo-Riemannian metric, and  their  associated  \emph{twistor distribution}, a  rank 2 distribution  on the   5-dimensional  total space of the circle bundle of self-dual  null 2-planes. For pKE metrics with non-zero scalar curvature this twistor  distribution has exactly two integral leaves and   is `maximally non-integrable' on their complement, a so-called \ttf-distribution.   Our main result  establishes a  simple  correspondence between   the anti-self-dual Weyl tensor of a pKE metric with non-zero scalar curvature  and the Cartan quartic of the associated  twistor distribution. This will be followed by a  discussion of  this correspondence for  general split-signature metrics which is shown to be much more involved.  We  use Cartan's method of equivalence to produce a large number of explicit examples of pKE metrics with non-zero scalar curvature whose anti-self-dual Weyl tensor have  special real  Petrov  type. In the case of real Petrov type $D,$ we obtain a complete local classification. Combined with the main result, this produces twistor distributions whose Cartan quartic has the same algebraic type as the  Petrov type of the constructed pKE metrics. In a similar manner, one can obtain  twistor distributions with Cartan quartic of arbitrary algebraic type. As a byproduct of our pKE examples  we naturally obtain para-Sasaki-Einstein metrics in five dimensions. Furthermore,  we study various Cartan geometries naturally associated to certain classes of pKE 4-dimensional metrics. We observe that in some geometrically distinguished cases the corresponding \emph{Cartan connections} satisfy the Yang-Mills equations. We then provide explicit examples of such Yang-Mills Cartan connections.   
\end{abstract}

\setcounter{tocdepth}{2} 
\tableofcontents 

\section{Introduction and main results}

The main purpose of this article is to give a detailed treatment of para-K\"ahler-Einstein structures in dimension 4 and examine their relation to maximally non-integrable distributions of rank 2 in dimension 5, referred  to as (2,3,5)-distributions. 

Our initial motivation for this article is twofold. Firstly, it is an extension of  the observation made in  \cite{BN-rolling} where, inspired by the rolling problem of Riemannian surfaces \cite{BH-G2}, a notion of \emph{projective rolling} was defined which gives rise to (2,3,5)-distributions. Consequently, it was observed that the (2,3,5)-distributions whose algebra of infinitesimal symmetries is maximal i.e.  the split real form of $\mathfrak g_2,$ can be obtained from such construction with a direct link to  the homogeneous \emph{para-K\"ahler-Einstein} metric on $\mathbf{SL}_3(\mathbb{R})/\mathbf{GL}_2(\mathbb{R})$ referred to as the \emph{dancing metric.} 
We point out that para-Hermitian structures and their variations  naturally appear in various geometric settings since they were first defined in \cite{Libermann}.  The reader may consult \cite{Survey} for a survey. 

Our second motivation comes from the twistorial construction of rank 2 distributions from  conformal structures of split signature in dimension four, referred to as \emph{twistor distributions}, described in \cite{AN-G2}.   In open subsets where the self-dual Weyl curvature is non-zero such distributions, which are naturally induced on the $\mathbb S^1$-bundle of self-dual null planes,  are  (2,3,5).  We will discuss that, in general, the fundamental invariant of twistor distributions, referred to as the \emph{Cartan quartic}, depend on the fourth jet of the components of the Weyl curvature of  the conformal structure. This poses  a basic question: whether the Cartan quartic of twistor distributions   can have any root type? In this article we answer this question affirmatively. As a by-product of our construction one naturally obtains explicit examples of \emph{para-Sasaki-Einstein} metrics.   

Before proceeding further, a brief definition of the geometric structures appearing in this article is in order. As will be defined in  \ref{sec:rudim-4-dimens}, an \emph{almost para-complex} structure on a manifold, $M,$ is defined as an  endomorphism $K:\mathrm{T}M\to \mathrm{T}M$ satisfying $K^2=\mathrm{Id}_{\mathrm{T}M}$ whose $\pm 1$-eigenspaces have rank 2.  As a result, unlike almost-complex structures, the eigenspaces of $K$ split each tangent space of $M$ into two transversal distributions. The integrability of these distributions induces a \emph{para-complex} structure on $M$ which results in two transversal foliations of  $M.$ Similarly, an \emph{almost-para-Hermitian}, \emph{para-Hermitian} and \emph{para-K\"ahler}  structure can be defined in terms of $K$ and a pseudo-Riemannian metric of split signature satisfying certain compatibility condition, as explained in  \ref{sec:almost-para-herm} and \ref{sec:pk-structure-an}.

The first objective of this article, presented in  \ref{sec:almost-para-herm-2}, is to give a unifying treatment of almost para-Hermitian and para-K\"ahler structures in dimension four via Cartan's method of equivalence and analyze the curvature decompositions in each case. To our knowledge such presentation has been absent in the literature. 

Our main topic of interest, treated in  \ref{sec:pke-4-manifolds}, is para-K\"ahler-Einstein (pKE) metrics,  defined as para-K\"ahler structures for which the metric is Einstein i.e. its trace-free Ricci tensor is zero. Because of our interest in non-integrable twistor distributions, we restrict our considerations to only those pKE metrics with  \emph{non-zero scalar curvature}. In \ref{sec:pk-structures-as} it is shown that such metrics define a Cartan  geometry of type $(\mathbf{SL}_3(\mathbb{R}),\mathbf{GL}_2(\mathbb{R})).$   In  \ref{sec:cart-reduct-symm} we investigate  five classes of pKE metrics for which the \emph{root type}, or the \emph{Petrov type}, of the anti-self-dual Weyl tensor $\W^-$ (see \ref{sec:cart-penr-petr}), is \emph{non-generic} and \emph{real}. Let us mention briefly that by root type, or the Petrov type,  of $\W^-$ we refer to the multiplicity  pattern of the roots of the  4th order polynomial obtained from the  representation of $\W^-$ as a binary quartic at each point. 
For these five classes of pKE metrics we carry out the Cartan reduction procedure case by case. This enables us to find all homogeneous models of pKE metrics in dimension 4. Moreover, we find explicit examples of pKE metrics with non-zero scalar curvature in every \emph{special real} root type of $\W^-$. In particular, Theorem \ref{thm:typeII-exa-1} gives explicit examples of pKE structures of Petrov type $II$, Theorem \ref{typeiii} gives examples for  Petrov type $III$, and Theorem \ref{typeiv} gives examples of Petrov type $N$. The pKE structures of  \emph{real} Petrov type $D$ are described in Theorem \ref{thm:petrov-types-D-LocNormForm}. This Petrov type is particularly interesting, since we found \emph{all possible} such pKE metrics with non-zero scalar curvature.  The method we have used to derive these examples is known as Cartan's  reduction method which is followed by integrating the reduced structure equations.   Moreover, Cartan's reduction method combined with   Cartan-K\"ahler analysis  is used to obtain the local generality of pKE metrics for each Petrov type in the real analytic case, as presented in Table \ref{tab:CK}. We point out that   \ref{sec:cart-reduct-symm} serves as an example of how effective Cartan's method of equivalence is in studying geometric structures.

In \ref{sec:2-3-5-dist} we discuss the relationship between pKE metrics and (2,3,5)-distributions.  More precisely, in \ref{sec:primer-2-3-5} we  briefly review the basic facts about (2,3,5)-distributions,  including the  Cartan connection, the Cartan quartic and associated conformal structure of split signature. 
In \ref{sec:2-3-5-dist-asd}, we recall   the fact that for any pKE structure, and more generally, any indefinite conformal structure in dimension four,  a rank 2 distribution, referred to as a \emph{twistor distribution}, is naturally induced on its  5-dimensional space of null self-dual planes, $\cO_+,$ and that another rank 2 twistor distribution is induced on the space of null anti-self-dual planes, $\cO_-;$ this  was originally observed in \cite{AN-G2}.    Moreover, the distribution induced on $\cO_+$ or $\cO_-$ is a (2,3,5) distribution at a point if and only if the self-dual Weyl tensor, $Weyl^+,$  or the anti-self-dual Weyl tensor, $Weyl^-,$ is non-zero at that point, respectively. After necessary coframe adaptations, which are performed at the beginning of \ref{sec:2-3-5-dist-asd}, we state our prime result on the surprising proportionality of the two quartics, the Cartan quartic of the (2,3,5) twistor distribution on $\cO_+$, and the anti-self-dual Weyl quartic of the pKE metric, as in the following theorem.
\newtheorem*{rmk4}{\bf Theorem \ref{cor:MainResult}}
\begin{rmk4}
\textit{Given a pKE metric for which the scalar curvature is non-zero,  the  Cartan quartic for the  non-integrable twistor distribution on $\cO_+$ is a non-zero multiple of the  quartic representation of the anti-self-dual Weyl curvature $Weyl^-$.} 
\textit{In particular,  the Cartan quartic of  the twistor distribution and the anti-self-dual Weyl curvature of the metric have the same root type.}
\end{rmk4} 
In Remark \ref{rmk:natural-identification} we explain the natural identification that underlies this theorem. 
As a result of  Theorem \ref{cor:MainResult}  and our examples of pKE metrics of non-generic Petrov types in \ref{sec:pke-4-manifolds}, we obtain a large class of explicit examples of  \emph{twistor distributions} for each special real algebraic type of the Cartan quartic.  
 Additionally, one obtains that the associated $(3,2)$ signature conformal structure on an open subset of $\cO_+$ has an Einstein representative whose  conformal holonomy is a subgroup of $\mathbf{SL}_3(\mathbb{R})\subset \mathbf{G}_2^*\subset\mathbf{SO}_{4,3};$ this was also observed in \cite{SW-G2}.  
 The main purpose of \ref{sec:2-3-5-ASD} is to show why the coincidence of the root types of the quartics, explained in Theorem  \ref{cor:MainResult}, is remarkable. This is done by obtaining the  Cartan quartic  for the twistor distribution on $\cO_-$. It is shown that the coefficients of the Cartan quartic on $\cO_-$ depend on the 4th jet of $Weyl^-$ and there is no obvious relation between the algebraic types of these quartics.

 Finally, we also mention that, starting in \ref{sec:pk-structures-as},  a number of \emph{Cartan geometries} are introduced which are naturally associated with pKE metrics in dimension 4. Since these Cartan geometries live on \emph{principal bundles} over 4-dimensional manifolds with split-signature \emph{conformal metrics}, one can study the \emph{vacuum Yang-Mills equations} for the corresponding \emph{Cartan} connections. As far as we know few papers are concerned with such studies. Here an honorable exception is a paper by S. Merkulov \cite{Merkulov_1984}, who established in 1984 that the vacuum Yang-Mills equations for the Cartan normal conformal connection of a 4-dimensional conformal structure $(M,[g])$, are equivalent to the vanishing of its \emph{Bach tensor}. Also in this vein is the work \cite{Nurowski_2003}, where in particular, the vacuum Yang-Mills equations for Cartan connections associated with 3-dimensional parabolic geometries of type $(\mathbf{SL}_3(\mathbb{R}),\mathbf{P}_{12}),$ and $(\mathbf{SU}_{2,1},\mathbf{P}_{12})$ were considered, in which $\mathbf{P}_{12}$ denotes the Borel subgroup. It turns out that Cartan geometries of various type appear in the process of Cartan reduction performed on a given Cartan geometry. It is also clear that  reduction of a Cartan geometry results in a principal bundle over the same base. Therefore, if  the vacuum Yang-Mills equations can be defined  for a certain Cartan geometry, they can also be defined for all   the Cartan geometries obtained from the reduction procedure. We take this approach in \ref{sec:pk-structures-as} and in the subsequent sections, where we reduce the initial $(\mathbf{SL}_3(\mathbb{R}),\mathbf{GL}_2(\mathbb{R}))$-type Cartan geometry to Cartan geometries of various types depending on the Petrov types. In particular, we find in Proposition \ref{so222}, as a consequence of Theorem \ref{so22}, that \emph{all} pKE 4-dimensional structures for which the Einstein constant is equal to -3 satisfy vacuum Yang-Mills equations for the $\sla_3(\bbR)$-valued Cartan connection of the associated $(\mathbf{SL}_3(\mathbb{R}),\mathbf{GL}_2(\mathbb{R}))$-type Cartan geometry. Similarly, in Theorem \ref{thm:typeII-exa-1} we give \emph{examples} of pKE metrics satisfying the vacuum Yang-Mills equations for the Cartan connection $\cB$ of a Cartan geometry of type $(\SOtt,\mathbf{T}^2)$, which can be obtained for pKE metrics of  \emph{real special} Petrov type.

The EDS calculations mentioned in the text are carried out using the \texttt{Cartan} package in Maple written by Jeanne Clelland and an exterior differential package for Mathematica written by Sotirios Bonanos.

\addtocontents{toc}{\setcounter{tocdepth}{1}}
 
\subsection*{Conventions}
In this article we will be working in the real smooth category.  Since our results are of local nature, the manifolds
can be taken to be the maximal open sets over which the assumptions made in each statement is valid. 

The manifold $M$ is always 4 dimensional equipped with a metric $g$ of split  signature. The 1-forms $\theta^1,\theta^2,\theta^3,\theta^4$ represent a coframe on $M$ with respect to which  $g=2\theta^1\theta^3+2\theta^2\theta^4$ where it is understood that the terms such as $\theta^1\theta^3$ represent symmetric tensor product of the 1-forms $\theta^1$ and $\theta^3$ i.e. 
\[\theta^a\theta^b=\textstyle{\frac 12(\theta^a\otimes\theta^b +\theta^b\otimes\theta^a)}\]   
 
Given an $n$-dimensional  manifold, $N,$ equipped with a coframe $\{\beta^1,\dots,\beta^n\},$ the corresponding set of frame will be expressed as $\{\frac{\partial}{\partial\beta^1},\dots,\frac{\partial}{\partial\beta^n}\}$ i.e. $\textstyle{\frac{\partial}{\partial\beta^b}}\im\beta^a=\delta^a_{~b}.$ Given a function $F:N\to\RR,$ the so-called coframe derivatives of $F,$ denoted by  $F_i:N\to \RR,$ are defined as
 \begin{equation}
   \label{eq:coframDer}
F_i=\textstyle{\frac{\partial}{\partial\beta^i}}\im\exd F.
 \end{equation}
 When a set of 1-forms on $N$ is introduced as $I=\{\gamma^1,\dots,\gamma^k\},$ it represents  the ideal that is algebraically generated by the 1-forms $\gamma^1,\dots,\gamma^k\in \rT^* N,$ and is called a \emph{Pfaffian system}. A Pfaffian system is called integrable if it satisfies the Frobenius condition, $\exd I\subset I,$ where $\exd$ is the exterior derivative. The integral manifolds of an integrable Pfaffian system are called its \emph{leaves}. Locally, around a generic point $x\in N,$ the leaves of an integrable Pfaffian system induce a smooth foliation which enables one to consider the quotient space of its leaves, referred to as the \emph{leaf space} of $I.$ Since our treatment is local, we can always work in  sufficiently small neighborhoods which allows one to define the leaf space of an integrable Pfaffian system.

\subsection*{Acknowledgments}
This project benefited from the Simons Semester \emph{Symmetry and Geometric Structures} at IMPAN in 2018. GB was supported by Conacyt grant A1-S-45886. OM would like to thank Kael Dixon and Niky Kamran  for their interest and helpful conversations. OM acknowledges the partial support of  the grant GA19-06357S from  the Czech Science Foundation, GA\v CR. The research leading to these results has also received funding from the Norwegian Financial Mechanism 2014-2021 with project registration number 2019/34/H/ST1/00636.

\addtocontents{toc}{\setcounter{tocdepth}{2}} 
 
\section{Almost para-Hermitian and para-K\"ahler structures}
\label{sec:almost-para-herm-2}

The goal of this section is to fix notation, give necessary definitions and recall some facts that will be needed in subsequent sections. 
 More precisely, in  \ref{sec:rudim-4-dimens}  we recall some basic facts about pseudo-Riemannian
metrics in dimension four. The notion of almost para-Hermitian structure and the decomposition of its curvature into irreducible components with respect to the action of its
structure group  is defined in  \ref{sec:almost-para-herm-3}. Furthermore, in  \ref{sec:almost-para-herm-3} we define the so-called \emph{Petrov type} of the Weyl curvature.
 In \ref{sec:pk-metrics} we define para-K\"ahler structures in terms of  additional integrability and compatibility conditions  imposed on an almost para-Hermitian structure. We derive their structure  equations, curvature decomposition and give a local coordinate expression in terms of a  potential function. We end the section by giving examples of pKE structures in terms of  potential functions which, as will be shown in  \ref{sec:cart-reduct-symm}, correspond to the only homogeneous models with non-zero scalar curvature.

\subsection{Rudiments of indefinite  pseudo-Riemannian metrics in dimension 4}
\label{sec:rudim-4-dimens}
In this section we briefly recall the decomposition of the space of 2-forms into self-dual and anti-self-dual 2-forms using the Hodge star operator. As a result,  two 5-dimensional circle bundles of self-dual null planes and anti-self-dual null planes is obtained for any indefinite pseudo-Riemannian metric in dimension four. Subsequently, we recall the structure equations of split signature metrics and their curvature decomposition.

\subsubsection{The Hodge star operator.}
\label{sec:selfd-antis-forms}
From now on let $(M,g)$ be a 4-dimensional real oriented manifold equipped with a split signature metric $g$. Locally, we can
always find a real coframe $(\theta^1, \theta^2, \theta^3, \theta^4)=(\al^1,\al^2,\bal^1,\bal^2)$ in which the metric takes the form: 
\begin{equation}
  \label{metric}
 g = g_{ab}\theta^a\theta^b = 2\theta^1\theta^3+2\theta^2 \theta^4=2\al^1\bal^1+2\al^2\bal^2.
 \end{equation}
We denote by $(\textstyle{\frac{\partial}{\partial\theta^1},\frac{\partial}{\partial\theta^2},\frac{\partial}{\partial\theta^3},\frac{\partial}{\partial\theta^4}})=(\textstyle{\frac{\partial}{\partial\alpha^1},\frac{\partial}{\partial\alpha^2},\frac{\partial}{\partial\bar\alpha^1},\frac{\partial}{\partial\bar\alpha^2}})$  the dual frame. 
The coframe 
$(\theta^1, \theta^2, \theta^3, \theta^4)$ is \emph{null} i.e. the metric 
$g$ has constant coefficients $g_{ab}$, with $ g_{13} = g_{31}  = g_{24} = g_{42}=1$ as the only non-vanishing ones,  which implies \emph{nullity}, $g(\frac{\partial}{\partial\theta^a},\frac{\partial}{\partial\theta^b})=0$ for $a=1,\dots,4$. Note that in this paper the coframe/frame on $M$ will be denoted by two notations $(\theta)$ or $(\al,\bal)$, and by $(\frac{\partial}{\partial\theta})$ or  $(\frac{\partial}{\partial\alpha},\frac{\partial}{\partial\bar\alpha})$, respectively. The reason for this redundancy will be made clear in the next sections, when sometimes one, and sometimes the other notation will be more convenient.

We assume  that the coframe  $(\al^1, \al^2, \bal^1, \bal^2)$ is positively oriented, therefore, the volume form $\mathbf{vol}_g$ on $ M$ can be expressed as
$$\mathbf{vol}_g =\al^1\wedge\al^2\wedge \bal^1\wedge \bal^2.$$
This enables one to define the Hodge star operator i.e. a linear map $*: \Lambda^2 \mathrm{T}^*M  \to
\Lambda^2 \mathrm{T}^* M $ such that
$$*\omega(X,Y)\mathbf{vol}_g=\omega\dz X^\flat\dz Y^\flat$$
Here $X^\flat$ is a 1-form  associated to a vector field $X$ such that $Y\im X^\flat=g(X,Y)$ for all vector fields $Y$ on $M.$ 

It is easy to see that 
$$*^2=\id_{\Lambda^2 \mathrm{T}^*M},$$
and therefore the operator $*$  splits $\Lambda^2\mathrm{T}^* M$  into the direct sum
$$\Lambda^2\mathrm{T}^*M=\Lambda^2_+\oplus\Lambda^2_-$$
of its eigenspaces $\Lambda^2_\pm$, corresponding to its respective $\pm 1$  eigenvalues.
In what follows we will frequently use the basis $(\sigma^1_\pm,\sigma^2_\pm,\sigma^3_\pm)$  of the 
$\pm 1$-eignespaces
of $*$, expressed in terms of the null coframe $(\al^1,\al^2,\bal^1,\bal^2)$  as
\begin{subequations}\label{asd-sd-null}
  \begin{align}
 \sigma^1_+&=\al^1\dz\al^2,\quad\sigma^2_+=\bal^1\dz\bal^2,\quad\sigma^3_+=\al^1\dz\bal^1+\al^2\dz\bal^2 \quad \in\Lambda^2_+\label{sip}\\
\sigma^1_-&=\al^1\dz\bal^2,\quad\sigma^2_-=\bal^1\dz\al^2, \quad\sigma^3_-=\al^1\dz\bal^1-\al^2\dz\bal^2\quad \in\Lambda^2_-\label{sim}
  \end{align}
\end{subequations}
As a result, in terms of   $(\textstyle{\frac{\partial}{\partial\alpha^1},\frac{\partial}{\partial\alpha^2},\frac{\partial}{\partial\bar\alpha^1},\frac{\partial}{\partial\bar\alpha^2}}),$ a basis for the self-dual and anti-self-dual bivectors in $\Lambda^2\mathrm{T M}=\Lambda^{2,+}\oplus\Lambda^{2,-}$ can be expressed as
\[
  \begin{aligned}    
    &\textstyle{\sigma^+_1=\frac{\partial}{\partial\alpha^1}\dz \frac{\partial}{\partial\alpha^2},\quad \sigma^+_2=\frac{\partial}{\partial\bar\alpha^1}\dz \frac{\partial}{\partial\bar\alpha^2},\quad \sigma^+_3= \frac{\partial}{\partial\alpha^1}\dz \frac{\partial}{\partial\bar\alpha^1}+\frac{\partial}{\partial\alpha^2}\dz \frac{\partial}{\partial\bar\alpha^2}  \quad \in\Lambda^{2,+},}\\
  &\textstyle{  \sigma^-_1=\frac{\partial}{\partial\alpha^1}\dz \frac{\partial}{\partial\bar\alpha^2},\quad \sigma^-_2=\frac{\partial}{\partial\alpha^2}\dz \frac{\partial}{\partial\bar\alpha^1},\quad \sigma^-_3=\frac{\partial}{\partial\alpha^1}\dz \frac{\partial}{\partial\bar\alpha^1}-\frac{\partial}{\partial\alpha^2}\dz \frac{\partial}{\partial\bar\alpha^2}  \quad \in\Lambda^{2,-}.}
  \end{aligned}
\]
At each point $p$ of $M$ the bivector $\frac{\partial}{\partial\alpha^1}\dz \frac{\partial}{\partial\alpha^2}$  defines a \emph{null} plane $N_+=\Span\{\frac{\partial}{\partial\alpha^1},\frac{\partial}{\partial\alpha^2}\}$. Recall that a plane $N_+$ is  null if $g(X,X)=0$  for all  $X\in N_+$. At every point $p\in M$ we then have the space $\cN_p$ of all null planes at $p$.

Let us consider a pair $(a,N_p)$ where  $a\in \SOtt$ and $N_p\in\cN_p.$ The group $\SOtt$ acts naturally on the space $\cN_p$ via:
\[(a,N_p)\,\to\, a\cdot N_p=\{aX\,|\, X\in N_p \}.\]
This action decomposes $\cN_p$ into {\em two} orbits
$$\cN_p=\cO_{p+}\sqcup\cO_{p-}.$$ 
Each of these orbits is diffeomorphic to a circle $\cO_{p\pm}\cong\bbS^1$. Take $N_{p+}\in \cO_{p+}$ and assume $N_{p+}=\Span\{\frac{\partial}{\partial\alpha^1}|_{p},\frac{\partial}{\partial\alpha^2}|_{p}\}.$ Note that its defining bivector $\frac{\partial}{\partial\alpha^1}\w\frac{\partial}{\partial\alpha^2}|_{p}$ is \emph{self-dual}. Since the action of $\SOtt$ does not change self-duality of null planes, the orbit $\cO_{p+}$ is called  the \emph{space of self-dual null planes} at $p.$ Consequently the orbit $\cO_{p-}$ is comprised of null planes defined by  anti-self-dual bivectors and  is therefore called the \emph{space of anti-self-dual null planes} at $p.$

More explicitly, at $p\in M$ the  set of self-dual  and anti-self-dual   null planes can be  parametrized, respectively, by $\lambda,\mu\in \RR\cup\{\infty\}$ in the following way
\begin{equation}
  \label{eq:sd-asd-parametrized}
 \cO_{p+}= \mathrm{Ker}\{\theta^1+\mu \theta^4,\theta^2-\mu\theta^3\},\quad \cO_{p-}= \mathrm{Ker}\{\theta^2-\lambda \theta^1,\theta^3+\lambda\theta^4\}.
  \end{equation}
This parametrization will be used in \ref{sec:2-3-5-dist}. The bundles $\cN_{\pm}:=\underset{p}\bigcup \cN_{p\pm}$ equipped with the projections 
\begin{equation}
  \label{eq:nupm} 
  \nu_+\colon\cN_+\to M,\qquad \nu_-\colon\cN_-\to M,
\end{equation}
where $(\nu_\pm)^{-1}(p)=\cN_{p_\pm},$ at  $p\in M,$  are referred to by various names, including circle \emph{twistor} bundles (\cite{AN-G2}) or bundles of real $\alpha$-planes and $\beta$-planes \cite{AG-Book}. In what follows we will frequently refer to the \emph{circle} bundles $\cN_+$ and $\cN_-$ as the  bundle of self-dual and anti-self-dual null planes, respectively.

  \subsubsection{Structure equations.}
\label{sec:structure-equations}
 The null coframe $(\theta^1,\theta^2,\theta^3,\theta^4)$ uniquely defines the
Levi-Civita connection 1-forms $\Gamma^a_{~b}$ via the first structure equations: 
 $$\begin{aligned}
&\der\theta^a+\Gamma^a_{~b}\dz\theta^b=0&\mathrm{(torsionfreeness)},\\
&g_{ac}\Gamma^c_{~b}+g_{bc}\Gamma^c_{~a}=0&\mathrm{(metricity)}.
\end{aligned}
$$
As a result, the Riemann curvature of the metric $g$, given by  the $\sott$-valued 2-form $\tfrac12R^a_{~bcd}\theta^c\dz\theta^d,$ is defined via the second structure equations
\begin{equation}
  \label{c2se}
\der\Gamma^a_{~b}+\Gamma^a_{~c}\dz\Gamma^c_{~b}=\tfrac12R^a_{~bcd}\theta^c\dz\theta^d.
\end{equation}
Via the action of $\SOtt,$ the Riemann curvature decomposes into the   components known as the traceless Ricci tensor, Weyl curvature and the scalar curvature.  The Ricci tensor is defined as $R_{ab}=R^c_{~acb}$  and the scalar curvature is $R=R_{ab}g^{ab}$, where $g^{ab}g_{bc}=\delta^a_{~c}.$ The trace-free part of the Ricci tensor is defined as $\overset{\circ}R_{ab}=R_{ab}-\tfrac14 Rg_{ab}$. 
Defining  the  Schouten tensor  as
$$\Rho_{ab}=\tfrac12 R_{ab}-\tfrac{1}{12}Rg_{ab},$$
the Weyl tensor is expressed as
\begin{equation}
  \label{eq:Weylabcd}
  C^a_{~bcd}=R^a_{~bcd}+g_{ad}\Rho_{cb}-g_{ac}\Rho_{db}+g_{bc}\Rho_{da}-g_{bd}\Rho_{ca}.
\end{equation}
 
Solving the metricity condition for  the first structure equations, it follows that  the connection 1-forms  $\Gamma^a_{~b}$ can be expressed as
\begin{equation}
  \label{LC}
\Gamma^a_{~b}=\bma[cc|cc]
\Gamma^1_{~1}&\Gamma^1_{~2}&0&\Gamma^1_{~4}\\
\Gamma^2_{~1}&\Gamma^2_{~2}&-\Gamma^1_{~4}&0\\
\cmidrule(lr){1-4}
0&-\Gamma^4_{~1}&-\Gamma^1_{~1}&-\Gamma^2_{~1}\\
\Gamma^4_{~1}&0&-\Gamma^1_{~2}&-\Gamma^2_{~2}
\ema.
\end{equation}
Consequently, the  torsion-free condition yields
\begin{equation}
  \label{ca1}
\begin{aligned}
&\der\al^1=-\Gamma^1_{~1}\dz\al^1-\Gamma^1_{~2}\dz\al^2-\Gamma^1_{~4}\dz\bal^2\\
&\der\al^2=-\Gamma^2_{~1}\dz\al^1-\Gamma^2_{~2}\dz\al^2+\Gamma^1_{~4}\dz\bal^1\\
&\der\bal^1=\Gamma^4_{~1}\dz\al^2+\Gamma^1_{~1}\dz\bal^1+\Gamma^2_{~1}\dz\bal^2\\
&\der\bal^2=-\Gamma^4_{~1}\dz\al^1+\Gamma^1_{~2}\dz\bal^1+\Gamma^2_{~2}\dz\bal^2.
\end{aligned}
\end{equation}
Passing to the second structure equations, one notes that due to the symmetries of the Riemann tensor, $R_{abcd}=R_{[ab][cd]}=R_{cdab}$, setting 
$R^{ab}{}_{cd}=g^{ae}g^{bf}R_{efcd}$,  one obtains a linear map  given by  
$$Riemann:\Lambda^2\mathrm{T}^*M\to\Lambda^2\mathrm{T}^*M,\quad\quad 
Riemann(\theta^a\dz\theta^b)=\tfrac12 R^{ab}{}_{cd}\theta^c\dz\theta^d.$$ 
Since 
$\Lambda^2\mathrm{T}^*M=\Lambda^2_+\oplus\Lambda^2_-$, the matrix form of this map can be expressed as
\begin{equation}
  \label{riedec}
Riemann=\bma[c||c] Weyl^++\tfrac{1}{12}R \id_{\Lambda^2_+}&\overset{\circ}{Ricci}\\\cmidrule(lr){1-2}\morecmidrules\cmidrule(lr){1-2}
\overset{\circ}{Ricci}{}^* &Weyl^-+\tfrac{1}{12}R \id_{\Lambda^2_-}\ema.
\end{equation}
Here $Weyl^+$ and $Weyl^-$ are traceless $3\times 3$ matrices, and $\overset{\circ}{Ricci}{}^*$ is a $3\times 3$ matrix related to the $3\times 3$ matrix of trace-free Ricci tensor, $\overset{\circ}{Ricci},$ via $\overset{\circ}{Ricci}{}^*=(H\overset{\circ}{Ricci}H^{-1})^T$, where $H=\tiny{\bma 0&0&-1\\0&2&0\\-1&0&0\ema}$.  
The matrices $HWeyl^+, HWeyl^-$ are symmetric and their components will be denoted by $(\Psi_0{}',\Psi_1{}',\Psi_2{}',\Psi_3{}',\Psi_4{}')$ and  $(\Psi_0,\Psi_1,\Psi_2,\Psi_3,\Psi_4)$ respectively. 
 Moreover, let us denote the the 9-components of $\overset{\circ}{Ricci}$ by $(\Rho_{11},\Rho_{12},\Rho_{22},\Rho_{14}, \Rho_{13}-\Rho_{24},\Rho_{23},\Rho_{33},\Rho_{34},\Rho_{44}).$ It follows that the scalar curvature can be written as $R=12(\Rho_{13}+\Rho_{24})$,  as given in  \eqref{eq:RiemmanCurv1}. Since the Ricci tensor $R_{ab}$ and the Schouten tensor $\Rho_{ab}$ are linearly related,   we will be using the Schouten tensor in the  sequel. As a result, the second structure equations \eqref{c2se} read
 \begin{equation}
   \label{ca21}
\begin{aligned}
\tfrac12\der&(\Gamma^1_{~1} +  \Gamma^2_{~2})+\Gamma^1_{~4}\dz\Gamma^4_{~1}=
\\ 
& 
 -\Psi_3' \sigma^1_+ - 
 \Psi_1' \sigma^2_+-\tfrac12 (2\Psi_2' - \Rho_{13}-\Rho_{24})
 \sigma^3_++\Rho_{14} \sigma^1_-  -\Rho_{23} \sigma^2_- + 
 \tfrac12(\Rho_{13}-\Rho_{24})\sigma^3_-\\
\der&\Gamma^4_{~1}+\Gamma^4_{~1}\dz(\Gamma^1_{~1}+\Gamma^2_{~2})=\\ 
&-\Psi_4' \sigma^1_+  - (\Psi_2' + \Rho_{13} + \Rho_{24})\sigma^2_+-\Psi_3'\sigma^3_+ - \Rho_{11} \sigma^1_-- \Rho_{22} \sigma^2_-+
\Rho_{12}\sigma^3_-\\
\der&\Gamma^1_{~4} +(\Gamma^1_{~1}+\Gamma^2_{~2})\dz \Gamma^1_{~4}=\\ 
& (\Psi_2' + \Rho_{13} + 
    \Rho_{24}) \sigma^1_++ 
 \Psi_0' \sigma^2_++\Psi_1'\sigma^3_+
 + \Rho_{44} \sigma^1_- +\Rho_{33} \sigma^2_- + 
\Rho_{34}\sigma^3_-, \end{aligned}
 \end{equation}
with analogous equations for the `unprimed' objects:
\begin{equation}
\label{ca22}
\begin{aligned}
\tfrac12\der&(\Gamma^1_{~1} -  \Gamma^2_{~2})+\Gamma^1_{~2}\dz\Gamma^2_{~1}=
\\ 
& 
 \Psi_1 \sigma^1_-+\Psi_3 \sigma^2_- -\tfrac12 (2\Psi_2 - \Rho_{13}-\Rho_{24})
 \sigma^3_- +\Rho_{12} \sigma^1_+ -\Rho_{34} \sigma^2_+ + 
 \tfrac12(\Rho_{13}-\Rho_{24})\sigma^3_+\\
\der&\Gamma^2_{~1} +\Gamma^2_{~1}\dz(\Gamma^1_{~1}-\Gamma^2_{~2})=\\ 
&- 
 \Psi_0 \sigma^1_-- (\Psi_2 + \Rho_{13} + \Rho_{24}) \sigma^2_-+\Psi_1\sigma^3_- 
-\Rho_{11} \sigma^1_+  - \Rho_{44} \sigma^2_+ + 
\Rho_{14}\sigma^3_+\\
\der&\Gamma^1_{~2}+(\Gamma^1_{~1}-\Gamma^2_{~2})\dz\Gamma^1_{~2}=\\ 
& (\Psi_2 + \Rho_{13} + \Rho_{24})
  \sigma^1_-+\Psi_4 \sigma^2_-  -\Psi_3\sigma^3_- + \Rho_{22} \sigma^1_+ + \Rho_{33} \sigma^2_++
\Rho_{23}\sigma^3_+.
\end{aligned}
\end{equation}
Here we used the respective basis $(\sigma^1_\pm,\sigma^2_\pm,\sigma^3_\pm)$ of  $\Lambda^2_\pm$, as defined in \eqref{asd-sd-null}.

\begin{remark}\label{rmk:2approaches}
  The usual way of employing the system of equations \eqref{ca1}, \eqref{ca21}-\eqref{ca22}, is to think about $(\theta^1,\theta^2,\theta^3,\theta^4)$ as a given coframe on $M$, and to use the equations \eqref{ca1}, \eqref{ca21}-\eqref{ca22} to uniquely determine the Levi-Civita connection forms $\Gamma^i_{~j}$, and consequently the curvature $R^a{}_{bcd}$ of $g$, in terms of this chosen coframe. Alternatively, in the language of  $G$-structures, one observes that $\theta^i$'s are ambiguous up to an action of $\SOtt$ since they were chosen so that \eqref{metric} is satisfied. One says that  $\SOtt$ is the \emph{structure group} of the pseudo-Riemannian structure. As a result, one can define a principal  $\SOtt$-bundle $\pi\colon\cF\to M$, as the bundle of all null coframes with respect to which \eqref{metric} holds. In this language the $\theta^i$'s give rise to a  lifted null coframe at each point of $\cF$ and the $\Gamma^i_j$'s mimic the Maurer-Cartan forms of $\sott$; they are uniquely defined  on $\cF$ as a result of the torsion-free condition. Moreover, these 1-forms, together with $\theta^i$s, form a basis of 1-forms at every point of $\cF$. Hence, one obtains a unique coframe at each point of $\cF$, consisting of 1-forms $(\theta^i,\Gamma^i{}_j)$, which is transformed equivariantly in each fiber of $\cF$ and satisfy the equations \eqref{ca1}, \eqref{ca21}-\eqref{ca22}  everywhere on $\cF$. We refer to \cite{Gardner,Olver}  for an overview of this  exterior differential system (EDS) viewpoint. 
\end{remark}

\subsection{Almost para-Hermitian metrics}
\label{sec:almost-para-herm-3}
In this section we define almost para-complex structures and almost para-Hermitian metrics. We obtain the structure equations and curvature decomposition. Using the curvature decomposition,  we recall the well-known Petrov classification of such structures.
\subsubsection{Definitions}
\label{sec:almost-para-herm}
An \emph{almost para-Hermitian structure} $(M,g,K)$ on a 4-dimensional manifold $M$ with a metric $g$ of signature $(+,+,-,-)$ is defined in terms of an endomorphism 
$$K:\mathrm{T}M\to\mathrm{T}M,$$ 
such that 
\[\hspace{4.5cm}
 K^2=\id_{\mathrm{T}M},\qquad \qquad  (K~\mathrm{paracomplex}),\]
whose $\pm 1$-eigenvalues have rank 2 and, additionally, satisfies the compatibility condition
\[
 g(KX,KY)=-g(X,Y),\qquad \forall X,Y\in \mathrm{T}M,\qquad (K~\mathrm{metric~compatible}).\]
An almost para-Hermitian structure $(M,g,K)$ distinguishes a pair of rank 2 distributions $\cH$ and $\bar{\cH}$ on $M$ defined as the $\pm 1$-eigenspaces of $K$ i.e. 
\begin{equation}
  \label{eq:H_barH_null_distr}
  {\cH}=(K+\id_{\mathrm{T}M})\mathrm{T}M,\quad\mathrm{and}\quad \bar{\cH}=(K-\id_{\mathrm{T}M})\mathrm{T}M.
\end{equation}
It follows that  $$\mathrm{T}M={\cH}\oplus\bar{\cH}.$$ Moreover $\cH$ and $\bar{\cH}$ are \emph{null} with respect to $g$ and must belong to the same  orbit in the space ${\cN}={\cO}_+\sqcup {\cO}_-$ of all null planes.

An almost para-Hermitian structure $(M,g,K)$ is called \emph{half}-para-Hermitian if precisely one of $\cH$ or $\bar{\cH}$ are integrable i.e. either $ [{\cH},{\cH}]\subset{\cH}$ or $  [\bar{\cH},\bar{\cH}]\subset\bar{\cH}.$ If $\cH$ and $\bar{\cH}$ are \emph{both} integrable i.e. 
$$[{\cH},{\cH}]\subset{\cH},\quad\mathrm{and}\quad  [\bar{\cH},\bar{\cH}]\subset\bar{\cH},
$$
then the almost-para-Hermitian structure $(M,g,K)$ is called \emph{para-Hermitian}.

An almost para-Hermitian structure $(M,g,K)$ defines a \emph{para-K\"ahler} 2-form
\begin{equation}
  \label{eq:2form1}
\rho(X,Y):=g(KX,Y).
\end{equation}
The fact that $\rho$ is skew symmetric, $\rho(X,Y)=-\rho(Y,X)$,  follows from the algebraic properties of $K$.  

An almost para-Hermitian structure $(M,g,K)$ is called \emph{almost para-K\"ahler} if and only if  the 2-form $\rho$ is closed i.e.
\begin{equation*}
\der\rho=0.
 \end{equation*}
 An almost para-Hermitian structure $(M,g,K)$ is  para-K\"ahler if it is para-Hermitian and almost para-K\"ahler i.e. $\cH$ and $\bar\cH$ are integrable and $\rho$ is closed. 
\subsubsection{Almost para-Hermitian structure in an adapted frame.} 
\label{sec:almost-para-herm-1}
A coframe $(\al^1,\al^2,\bal^1,\bal^2)$
 on a 4-dimensional manifold $M$ is \emph{adapted} to an almost para-Hermitian structure $(M,g,K)$ if and only if 
\begin{equation}
  \label{ac2}
\begin{aligned}
&g=2\al^1\bal^1+2\al^2\bal^2\\
&K=\textstyle{\al^1\otimes\frac{\partial}{\partial\alpha^1} +\al^2\otimes\frac{\partial}{\partial\alpha^2}  -\bal^1\otimes\frac{\partial}{\partial\bar\alpha^1}-\bal^2\otimes\frac{\partial}{\partial\bar\alpha^2}}.
\end{aligned}
\end{equation}
It follows that in such adapted coframes
\begin{equation}
\label{ac3}
\rho=\al^1\dz\bal^1+\al^2\dz\bal^2,
\end{equation}

  At every point of a 4-dimensional almost para-Hermitian manifold $(M,g,K)$ the stabilizer $\mathbf{H}\subset \mathbf{GL}_4(\bbR)$ of the pair $(g,K)$, i.e.
  $$\mathbf{H}=\{U\in\glg_4(\bbR)\,:\,g(UX,UY)=g(X,Y)\,\,\&\,\,K(UX)=UK(X)\},$$
  satisfies
  $$\mathbf{H}\cong \glg_2(\bbR)\subset\SOtt.$$
  Expressing  $\mathbf{H}$ in the coframe  $(\theta^1,$ $\theta^2,\theta^3,\theta^4)=(\al^1,\al^2,\bal^1,\bal^2)$ as in \eqref{ac2}, provides the 4-dimensional reducible representation $$T:\mathbf{H}\to\glg_4(\bbR)$$ of $\mathbf{H}$ given by
  \begin{equation} \label{ga2}
    T(U)=\bma
A&0\\
0&(A^{T})^{-1}
\ema\quad\mathrm{with}\quad A=
\begin{pmatrix}
  a_{11} & a_{12}\\
  a_{21} & a_{22}
\end{pmatrix}
\in \glg_2(\bbR)
.\end{equation}

As a result the geometry arising from the pair $(g,K)$ reduces the structure group of $M$ from $\glg_4(\RR)$ to $\glg_2(\bbR)$ via representation $T.$ 
 The $\glg_2(\bbR)$ irreducible decomposition of  $\bbR^4$ as a $\glg_2(\bbR)$-module is $\bbR^4=\bbR^2\oplus(\bbR^2)^*$. It reflects the splitting of $\mathrm{T}M$, into $\mathrm{T}M=\cH\oplus\bar{\cH}$.

\begin{proposition}\label{opr1}
  Every almost-para-Hermitian structure $(M,g,K)$ on a 4-dimensional manifold locally admits an adapted coframe. If $(\theta^a)=(\al^1,\al^2,\bal^1,$ $\bal^2)$
is a coframe adapted to $(M,g,K)$ then the most general adapted coframe is given by
\begin{equation}\label{ga1}
\tilde{\theta}^a=T(U)^a_{~b}\theta^b,
\end{equation}
where the  $4\times 4$ matrices $T(U)=(T(U)^a{}_b)$ are as in  \eqref{ga2}.
\end{proposition}

\subsubsection{$\glg_2(\bbR)$ invariant curvature decomposition}
\label{sec:glg2-bbr-invariant}
Any coframe adapted to $(M,g,K)$ is in particular a null coframe, as in \eqref{metric}. Thus to analyze the properties of $(M,g,K)$ we can use the structure equations \eqref{ca1}, \eqref{ca21}-\eqref{ca22}.  
 The stabilizer $\mathbf{H}\cong\glg_2(\bbR)$ of the pair $(g,K)$ is therefore the \emph{structure group} of the almost para-Hermitian structure $(M,g,K)$. It acts, via the representation $T$, on any adapted coframe $(\theta^a)$ as in \eqref{ga1}. The induced transformation of the Levi-Civita connection  \eqref{LC} and its curvature is given by
\begin{subequations}
  \label{troa}
\begin{align}
&\theta^a\to \tilde{\theta}{}^a=T(U)^a_{~b}\theta^b,   \label{troa1}\\
&\Gamma^a_{~b}\to \tilde{\Gamma}{}^a_{~b}=T(U)^a_{~c}\Gamma^c_{~d} T(U)^{-1d}{}_b-dT(U)^a_{~c} T(U)^{-1c}{}_b,   \label{troa2}\\
&R^a_{~bcd}\to\tilde{R}{}^a_{~bcd}=T(U)^a_{~e}R^e_{~fgh}T(U)^{-1f}{}_{b}T(U)^{-1g}{}_{c}T(U)^{-1h}{}_{d}.  \label{troa3}
\end{align}
\end{subequations}
The  transformations \eqref{troa3} gives  the action of $\glg_2(\bbR)$ on the 20-dimensional vector space  of the \emph{curvature tensors} $R^a_{~bcd}$. Using this action one can decompose the curvature tensor into its indecomposable components. First we define 10 vector spaces defined  in terms of the curvature components 
$(\Rho_{ab},\Psi_\mu,\Psi_\mu')$, $a,b=1,2,3,4$, $\mu=0,1,2,3,4$ as 
\begin{equation}
  \label{eq:Decomp1}
\begin{aligned}
Ric^3_{1}=&\{\Pi_{AB}~\mathrm{s.t.}~\Pi_{AB}=\bma \Rho_{11}&\Rho_{12}\\\Rho_{12}&\Rho_{22}\ema\}\\
Ric^3_{2}=&\{\bar{\Pi}_{AB}~\mathrm{s.t.}~\bar{\Pi}_{AB}=\bma \Rho_{33}&\Rho_{34}\\\Rho_{34}&\Rho_{44}\ema\}\\
Ric^3_{3}=&\{P^A_{~B}~\mathrm{s.t.}~P^A_{~B}=\bma \Rho_{13}-\Rho_{24}&2\Rho_{23}\\2\Rho_{14}&-\Rho_{13}+\Rho_{24}\ema\}\\
Scal^1=&\{\Rho_{13}+\Rho_{24}\}\\
Weyl^5_{1}=&\{W_{ABCD}=W_{(ABCD)}~\mathrm{s.t}\\&W_{0000}=\Psi_0,~W_{0001}=\Psi_1,~W_{0011}=\Psi_2,~W_{0111}=\Psi_3,~W_{1111}=\Psi_4,\}\\
Weyl^1_{1}=&\{\Psi_0'\},\quad Weyl^1_{2}=\{\Psi_1'\},
\quad Weyl^1_{3}= \{\Psi_2'\},\quad Weyl^1_{4}=\{\Psi_3'\},\quad Weyl^1_{5}=\{\Psi_4'\}.
\end{aligned}
\end{equation}
Here the (spinorial) indices $A,B,C,D=0,1$, and the equations $W_{ABCD}=W_{(ABCD)}$ mean that $W_{ABCD}$ is totally symmetric in indices $A,B,C,D$. 
The notation for the spaces $Ric^i_j$ and $Weyl^i_j$ is such that the upper index indicates the dimension of each space, and the lower index enumerates spaces of the same dimension. In particular the spaces $Ric^3_1$ and $Ric^3_2$  have dimensions 3 as spaces of symmetric $2\times2$ matrices, $Ric^3_3$ has dimension 3 as the space of traceless $2\times2$ matrices, and $Weyl^5_1$ has dimension 5 as the space of  symmetric tensors of degree 4 in dimension 2.

\begin{proposition}\label{riedech}
The $\glg_2(\bbR)\subset\SOtt$ invariant decomposition of the 20-dimensional curvature space, $Riemann^{20},$ of an almost para-Hermitian structure $(M,g,K)$ in dimension 4 is
\begin{equation}\label{decu}
\begin{aligned}
Riemann^{20}=&\underbrace{Ric^3_1\oplus Ric^3_2\oplus Ric^3_3}_{\text{traceless Ricci}}\oplus\\
&\underbrace{Scal^1}_{\text{Ricci scalar}} \oplus\\&\underbrace{ Weyl^1_1\oplus Weyl^1_2\oplus Weyl^1_3\oplus Weyl^1_4\oplus Weyl^1_5}_{\text{self-dual Weyl}}\oplus\\&\underbrace{Weyl^5_1.}_{\text{anti-self-dual Weyl}}
\end{aligned}  
\end{equation}
\end{proposition}

\begin{proof}
This decompositions can be obtained similar to the decomposition of $Riemann^{20}$ into $Weyl^\pm$, $\overset{\circ}{Ricci}$ and $R$  via the $\SOtt$ invariant decomposition of  $\Lambda^2\mathrm{T}^*M$.
In this case
one decomposes $\Lambda^2\mathrm{T}^*M$ using the $\glg_2(\bbR)$ invariant decomposition of the tangent space 
$$\mathrm{T}M={\cH}\oplus\bar{\cH},$$
which is possible for any almost para-Hermitian manifold $(M,g,K)$. 
The associated decomposition of the cotangent bundle $\Lambda^1\mathrm{T}^*M$ is given by 
$$\Lambda^1\mathrm{T}^* M=\Lambda^{(1,0)}\oplus\Lambda^{(0,1)},$$
where
\[
  \begin{aligned}    \Lambda^{(1,0)}&=\{~\omega\in\Lambda^1\mathrm{T}^*M~|~\bar{X}\im\omega=0,~\forall\bar{X}\in\bar{\cH}~\}\\
    \Lambda^{(0,1)}&=\{~\bar{\omega}\in\Lambda^1\mathrm{T}^*M~|~X\im\bar{\omega}=0,~\forall X\in {\cH}~\}.
  \end{aligned}\]
 As a result, $\Lambda^2\mathrm{T}^*M$ is decomposed into 
$$\Lambda^2\mathrm{T}^*M=\Lambda^{(2,0)}\oplus\Lambda^{(1,1)}\oplus\Lambda^{(0,2)}.$$
It turns out that $\Lambda^{(2,0)}$ and $\Lambda^{(0,2)}$ are 1-dimensional, and $\Lambda^{(1,1)}$ has dimension 4. Choosing an adapted coframe $(\al^1,\al^2,\bal^1,\bal^2)$ we can write a basis for these spaces in terms of   the self-dual and anti-self-dual 2-forms,  $(\sigma^1_\pm,\sigma^2_\pm,\sigma^3_\pm),$  in \eqref{asd-sd-null}, as follows.
\[
  \begin{aligned}
    \Lambda^{(2,0)}\cap\Lambda^2_+&=\Span\{\sigma^1_+\}=\Lambda^{(2,0)},\\ \Lambda^{(0,2)}\cap\Lambda^2_+&=\Span\{\sigma^2_+\}=\Lambda^{(0,2)}\\ 
    \Lambda^{(1,1)}\cap\Lambda^2_+&=\Span\{\sigma^3_+\},\\
        \Lambda^{(1,1)}\cap\Lambda^2_-&=\Span\{\sigma^1_-,\sigma^2_-,\sigma^3_-\}=\Lambda^2_-,\\
    \end{aligned}
\] 
This gives a natural decomposition of $\Lambda^2_+$ into 1-dimensional $\glg_2(\bbR)$ invariant subspaces 
  $$\Lambda^2_+=\Span\{\sigma^1_+\}\oplus\Span\{\sigma^3_+\}\oplus\Span\{\sigma^2_+\} .$$ 
  Using this we can further decompose the map $Riemann$ from \eqref{riedec} as 
  \begin{equation}
    \label{eq:RiemmanCurv1}
   \begin{aligned}
 Riemann=&\tfrac{1}{12}R \id_{6\times 6}+
 \bma[c|c|c||ccc]
\Psi_2'&-2\Psi_3'&\Psi_4'&\Rho_{22}&2\Rho_{12}&\Rho_{11}\\
\cmidrule(lr){1-6}
\Psi_1'&-2\Psi_2'&\Psi_3'&\Rho_{23}&\Rho_{13}-\Rho_{24}&-\Rho_{14}\\
\cmidrule(lr){1-6}
\Psi_0'&-2\Psi_1'&\Psi_2'&\Rho_{33}&-2\Rho_{34}&\Rho_{44}\\
\cmidrule(lr){1-6}\morecmidrules\cmidrule(lr){1-6}
\Rho_{44}&2\Rho_{14}&\Rho_{11} &\Psi_2&2\Psi_1&\Psi_0\\
\Rho_{34}&\Rho_{13}-\Rho_{24}&-\Rho_{12} &-\Psi_3&-2\Psi_2&-\Psi_1\\
\Rho_{33}&-2\Rho_{23}&\Rho_{22} &\Psi_4&2\Psi_3&\Psi_2
 \ema.\end{aligned}
  \end{equation}
Comparing this with the decomposition \eqref{riedec} one obtains
\begin{itemize}
\item $Weyl^+$ gets decomposed into five  1-dimensional $\glg_2(\bbR)$ invariant subspaces denoted by $Weyl^1_1,\dots,Weyl^1_5$ which correspond to the components $\Psi'_0,\dots,\Psi'_4$ in \eqref{eq:RiemmanCurv1} respectively. 
\item $\overset{\circ}{Ricci}$ is decomposed into three invariant subspaces,  $Ric^3{}_1, Ric^3{}_2, Ric^3{}_3$ which correspond to the rows $(\Rho_{22},\Rho_{12},\Rho_{11})$, $(\Rho_{23}, \Rho_{13}-\Rho_{24},\Rho_{14})$, and $(\Rho_{33},$ $\Rho_{34},\Rho_{44})$ in \eqref{eq:RiemmanCurv1}  respectively. \item $Weyl^-$  remains indecomposable with its 5-dimensional representation $Weyl^5_1$ whose components are $(\Psi_0,\Psi_1,\Psi_2,\Psi_3,\Psi_4)$.
\item The Ricci scalar $R=12(\Rho_{13}+\Rho_{24})$ is proportional to the trace of $Riemann$ and gives the 1-dimensional invariant subspace $Scal^1$.  
\end{itemize} 
As a result one obtains  the decompositions \eqref{decu}.
\end{proof}

It is straightforward to find  the explicit action of the $\glg_2(\bbR)$ group on the indecomposable components of the curvature in \eqref{eq:Decomp1}.
\begin{proposition}\label{orbits}
 The curvature components $Ric^3_1$, $Ric^3_2$, $Ric^3_3$ and $Weyl^5_1$ are `tensorial' with respect to the action of $\glg_2(\bbR)$ i.e. for $U\in\glg_2(\bbR)\subset \SOtt$ given by \eqref{ga2}
 if an adapted coframe $(\theta^a)$ is transformed by
$$\theta^a\to\tilde{\theta}{}^a=T(U)^a_{~b}\theta^b,$$ 
then the transformation law for the curvature components $\Pi_{AB}$, $\bar{\Pi}_{AB}$ and $P^A_{~B}$ in \eqref{eq:Decomp1} is
\begin{subequations}\label{trrs}
  \begin{align}
&\Pi_{AB}\to\tilde{\Pi}{}_{AB}=\Pi_{CD}A^{tC}{}_A A^{tD}{}_B,\label{trrs1}\\
&\bar{\Pi}_{AB}\to\tilde{\bar{\Pi}}{}_{AB}=\bar{\Pi}_{CD}A^{-1C}{}_A A^{-1D}{}_B,\label{trrs2}\\
&P^A_{~B}\to\tilde{P}^A_{~B}=A^A_{~C}P^C_{~D}A^{-1D}{}_B,\label{trrs3}\\
&W_{ABCD}\to\tilde{W}{}_{ABCD}=W_{EFGH}A^{-1E}{}_AA^{-1F}{}_BA^{-1G}{}_CA^{-1H}{}_D.\label{trrs4}
\end{align}\end{subequations}
The curvature scalars $\Psi_0'$, $\Psi_1'$, $\Psi_2'$, $\Psi_3'$, $\Psi_4'$, $\Rho_{14}+\Rho_{23}$, are  \emph{weighted} scalars and transform to
\begin{equation}
  \label{eq:SD-Weyl-Trans}
  \begin{aligned} 
\Psi_0'&\to\tilde{\Psi}{}_0'=(\det A)^2~\Psi_0',\quad&\Psi_4'&\to\tilde{\Psi}{}_4'=(\det A)^{-2}~\Psi_4',\\
\Psi_1'&\to\tilde{\Psi}{}_1'=(\det A)~\Psi_1',\quad&\Psi_3'&\to\tilde{\Psi}{}_3'=(\det A)^{-1}~\Psi_3',\\
\Psi_2'&\to\tilde{\Psi}{}_2'=\Psi_2',\quad  &\Rho_{13}+\Rho_{24}&\to\tilde{\Rho}{}_{13}+\tilde{\Rho}{}_{24}=\Rho_{13}+\Rho_{24}.
\end{aligned}
\end{equation}
\end{proposition}
\begin{corollary}
 Every almost para-Hermitian structure $(M,g,K)$ in dimension 4  possesses   two scalar invariants which are the scalar curvature of the metric $g,$ given by $R=12(\Rho_{12}+\Rho_{34})$, and $\Psi_2',$ arising from  the self-dual Weyl tensor of the metric.  Moreover  the \emph{vanishing} of \emph{each} of the $\glg_2(\bbR)$ densities, $\Psi_0',\dots,\Psi_4'$, as well as  each of the $\glg_2(\bbR)$ tensors, $\Pi_{AB}$, $\bar{\Pi}_{AB}$ and $P^A_{~B},$ is an  invariant property of almost para-Hermitian structures.
\end{corollary}
\subsubsection{Cartan-Penrose-Petrov classification of the Weyl tensor}
\label{sec:cart-penr-petr}
One of the basic \emph{pointwise invariants} of 4-dimensional metrics of split signature  is the so-called Petrov type of its self-dual and anti-self-dual Weyl curvatures. To define it, note that the transformation law \eqref{trrs4} shows the action of  the structure group $\glg_2(\bbR)$ on the anti-self-dual Weyl tensors, $Weyl^5_1,$ as an 5-dimensional representation. This representation is isomorphic with the standard representation of $\glg_2(\RR)$ on $\mathrm{Sym}^4(\RR^2)^*$ i.e. the degree 4 homogeneous polynomials in two variables. Using \eqref{eq:Decomp1}, the quartic polynomial is given by
\begin{equation*}
\begin{aligned}
W(\xi)&=W_{ABCD}\xi^A\xi^B\xi^C\xi^D\\
&=\Psi_4(\xi^1)^4+4\Psi_3(\xi^1)^3(\xi^0)+6\Psi_2(\xi^1)^2(\xi^0)^2+4\Psi_1(\xi^1)(\xi^0)^3+\Psi_0(\xi^0)^4,
\end{aligned}
\end{equation*}
where $\xi=(\xi^0,\xi^1).$
It turns out that  $\xi$ can serve as a  homogeneous coordinate for the circle bundle of anti-self-dual planes, $\cN_-.$ More precisely,  using the Weyl curvature \eqref{eq:Weylabcd},    define $C_{abcd}=g_{ad}C^d_{~bcd},$ which can be used to define the multilinear map 
\[\bW:=C_{abcd}(\theta^a\w\theta^b)\circ(\theta^c\w\theta^d):\mathrm{Sym}^2(\Lambda^2\rT M)\rightarrow \cC^\infty(M).\] Restricting to anti-self-dual null planes, $\cN_-\subset \Lambda^2\rT M$, as   in  \eqref{eq:sd-asd-parametrized}, one can define the   quartic polynomial 
\begin{equation}
  \label{eq:quartic_ASD_Weyl}
  \begin{aligned}
W(\lambda)&\textstyle{=\bW(\frac{\partial}{\partial\theta^1}+\lambda\frac{\partial}{\partial\theta^2},\frac{\partial}{\partial\theta^4}-\lambda\frac{\partial}{\partial\theta^3},\frac{\partial}{\partial\theta^1}+\lambda\frac{\partial}{\partial\theta^2},\frac{\partial}{\partial\theta^4}-\lambda\frac{\partial}{\partial\theta^3})}\\
&=\Psi_4\lambda^4+4\Psi_3\lambda^3+6\Psi_2\lambda^2+4\Psi_1\lambda+\Psi_0
\end{aligned}
\end{equation}
where  
\[\Psi_4=C_{2323},\quad \Psi_3=C_{1323},\quad \Psi_2=-C_{1423},\quad \Psi_1=C_{2414},\quad \Psi_0=C_{1414},\]
which establishes the relation  $\lambda=\frac{\xi^1}{\xi^0}$ between the parameters.  The quartic  \eqref{eq:quartic_ASD_Weyl} is a representation of the anti-self-dual Weyl curvature of the metric $g.$ 

Similarly, restricting to self-dual null planes, $\cN_+\subset \Lambda^2\rT M$, and using the affine parametrization \eqref{eq:sd-asd-parametrized}, one can define the   quartic polynomial 
\begin{equation} 
  \label{eq:quartic_SD_Weyl}
  \begin{aligned}
W'(\mu)&\textstyle{=\bW(\frac{\partial}{\partial\theta^4}-\mu\frac{\partial}{\partial\theta^1}, \frac{\partial}{\partial\theta^3}+\mu\frac{\partial}{\partial\theta^2}, \frac{\partial}{\partial\theta^4}-\mu\frac{\partial}{\partial\theta^1}, \frac{\partial}{\partial\theta^3}+\mu\frac{\partial}{\partial\theta^2})}\\
&=\Psi'_4\mu^4+4\Psi'_3\mu^3+6\Psi'_2\mu^2+4\Psi'_1\mu+\Psi'_0
\end{aligned}
\end{equation}
where 
\[\Psi'_4=C_{1212},\quad \Psi'_3=C_{1213},\quad \Psi'_2=C_{1234},\quad \Psi'_1=C_{1334},\quad \Psi'_0=C_{3434},\]
The quartic $W'(\mu)$ is a representation of the self-dual Weyl curvature of $g$ whose coefficients transform according to   \eqref{eq:SD-Weyl-Trans}.

The Petrov type at each point is the root type of the quartics $W(\lambda)$ and $W'(\mu)$ at that point, since   multiplicity pattern of the  roots  is invariant under the induced action of the structure group $\SOtt$.   Note that since the coefficients of the quartics are real and transform under the action of   $\glg_2(\RR)$, the root type is closed under complex conjugation. As a result,  there are  10  root types for each of the quartics $W(\lambda)$ and $W'(\mu).$ Following the tradition in General Relativity, where the metric has Lorentzian signature, root types are grouped into the six  Petrov types, denoted by   $G$, $II$, $III$, $N$, $D$ and $O$. In the case of metrics of split signature, due to different reality conditions, one obtains a finer classification of Petrov types given by   
\begin{enumerate}
\item type $G^r$: 4 real simple  roots.  
\item type $G^c$: 2 real simple  roots and 2  complex conjugate roots.
\item type $G^{cc}$: 2 pairs of complex conjugate  roots..  
\item type $II^r$: 1 double real root, 2 simple real roots.
\item type $II^c$: 1 double real root, 2  complex conjugate roots.
\item type $III$: 1 triple real root and 1 simple real root.
\item type $D^r$: 2 double real roots.
\item type $D^c$: 2 double complex conjugate roots.
\item type $N$: 1 quadruple real root.
\item type $O$:   when all the coefficients of the quartic are zero.
\end{enumerate}
The letter $G$ stands for \emph{general type} since, generically, the Petrov type of a quartic is  $G$.
 If the quartic is non-zero, then the 9  root types  are listed in Figure ~\ref{fig:Petrov_real},  which shows  the self-conjugate pattern of roots in each type. The horizontal line represents the real line and  conjugation of roots is given by reflection with respect to the horizontal line.
\begin{figure}[h]
\centering
\includegraphics[width=.8\textwidth]{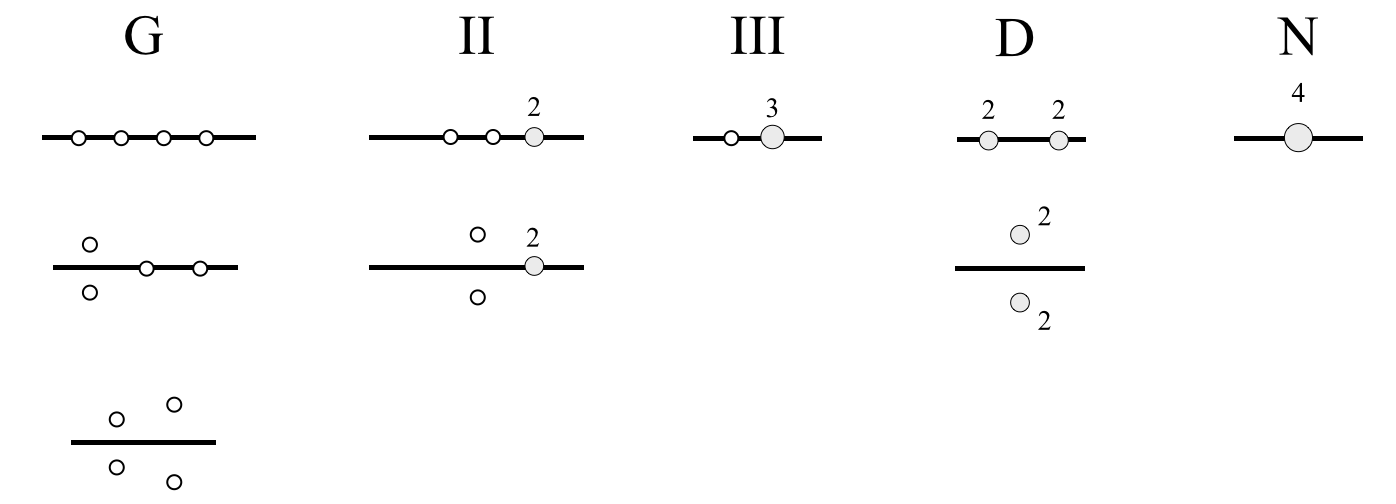}
\caption{Root types of a non-zero quartic with real coefficients.}
\label{fig:Petrov_real}
\end{figure} 

In what follows the root types $G^r, II^r, III, D^r, N, O$ will be referred to as the \emph{special real Petrov types}.  It is clear from  our discussion that since the parameters   $\lambda$ and $\mu$ in the quartics $W(\lambda)$ and $W'(\mu),$ parametrize $\cN_-$ and $\cN_+,$ respectively, a choice of real root for these quartics determine a choice of an anti-self-dual and self-dual null plane. This enables one to consider  anti-self-dual or self-dual null planes that correspond to a real root of the quartics $W(\lambda)$ or $W'(\mu).$  

\subsection{Para-K\"ahler (pK) metrics} 
\label{sec:pk-metrics}
In this section the para-K\"ahler condition is  used to reduce the structure equations of an almost para-Hermitian structure. After deriving their structure equations and curvature decomposition, we show that para-K\"ahler structures can be described in terms of a potential function, using which we give two examples of para-K\"ahler-Einstein metrics. These two examples turn out to be homogeneous as will be explained in \ref{sec:cart-reduct-symm}.

\subsubsection{PK structures in an adapted coframe.}
\label{sec:pk-structure-an}
A special feature of every almost para-Hermitian geometry $(M,g,K)$ is that in addition to  the (weighted) tensorial invariants arising from the curvature of the Levi-Civita connection, it  has invariants of lower order referred to as the \emph{intrinsic torsion}. These are defined in terms of the (Grey-Harvella type) decomposition of the covariant derivative of the 2-form $\rho$ \eqref{eq:2form1}  with respect to $\glg_2(\bbR)$. Two of these (relative) invariants are of particular interest in our setting. We will  express them in terms of the Levi-Civita connection 1-forms $\Gamma^a{}_b.$ 
Using \eqref{troa2}, one obtains that the transformation of the connection 1-forms $\Gamma^1_{~4}$ and $\Gamma^4_{~1}$ does not involve the  inhomogeneous terms   $\der (T(U))T(U)^{-1}$, which leads to the following proposition.
\begin{proposition}
Under the gauge transformation  \eqref{troa} of  adapted coframes  $(\theta^1,\theta^2,\theta^3,\theta^4)$ for an almost para-Hermitian structure, 
the  connection 1-forms $\Gamma^1_{~4}$ and $\Gamma^4_{~1}$  transform as
$$\begin{aligned}
&\Gamma^1_{~4}\to\tilde{\Gamma}^1_{~4}=(\det A)~\Gamma^1_{~4},\qquad &\Gamma^4_{~1}\to\tilde{\Gamma}^4_{~1}=(\det A)^{-1}~\Gamma^4_{~1}
\end{aligned}$$
where $A\in\glg_2(\bbR)$. As a result,  the vanishing of each of the connection 1-forms $\Gamma^1_{~4}$ and $\Gamma^4_{~1}$ is an 
invariant property of an almost para-Hermitian structure.
\end{proposition}
We have the following proposition.
\begin{proposition}\label{pkprop}
An almost para-Hermitian structure $(M,g,K)$  is para-K\"ahler if and only if $$\Gamma^1_{~4}=0\quad\mathrm{and}\quad \Gamma^4_{~1}=0,$$ in one (and therefore any) adapted coframe. As a result, the Levi-Civita connection form of $g$ is reduced to
$$\Gamma^a_{~b}=\bma \Gamma&0\\0&-\Gamma^T\ema,\quad\mathrm{with}\quad \Gamma\in\gla_2(\bbR)\otimes\Lambda^1\mathrm{T}^*M\equiv{\bf End}(\bbR^2)\otimes\Lambda^1\mathrm{T}^*M,$$
in any adapted coframe.
\end{proposition}
\begin{proof}
By Frobenius theorem the integrability of $\cH $ and $\bar{\cH}$ in an adapted coframe is equivalent to  $\der\al^1\dz\al^1\dz\al^2=0=\der\al^2\dz\al^1\dz\al^2$ and $\der\bal^1\dz\bal^1\dz\bal^2=0=\der\bal^2\dz\bal^1\dz\bal^2$ respectively. Using the  first structure equations (\ref{ca1}), it follows that  the simultaneous  integrability of $\cH$ and $\bar{\cH}$  implies 
\begin{equation}
\label{pk1}
\begin{aligned}
&\der\Gamma^4_{~1}\dz\al^1\dz\bal^1\dz\bal^2=0,\quad\quad\der\Gamma^4_{~1}\dz\al^2\dz\bal^1\dz\bal^2=0,\quad\\
 &\der\Gamma^1_{~4}\dz\bal^1\dz\al^1\dz\al^2=0,\quad\quad \der\Gamma^1_{~4}\dz\bal^2\dz\al^1\dz\al^2=0.\end{aligned}
 \end{equation}
 On the other hand, the almost K\"ahler condition $\der\rho=0$, when written in an adapted coframe reads 
\[\der(\al^1\dz\al^2+\bal^1\dz\bal^2)=0.\] 
Using the first structure equations \eqref{ca1} it follows that this condition is equivalent to 
\begin{equation}\label{pk2}
 \Gamma^1_{~4}\dz\bal^1\dz\bal^2+\Gamma^4_{~1}\dz\al^1\dz\al^2=0.
\end{equation}
It follows from  \eqref{pk1} and \eqref{pk2} that $\Gamma^4{}_1=0$ and $\Gamma^1{}_4=0$, as claimed. 
\end{proof}
Proposition \ref{pkprop} leads to the following ``para'' analogue of the well-known fact that the holonomy of Riemannian 4-manifolds which are  K\"ahler is a subgroup of $\mathbf{U}_2$. 
\begin{corollary}\label{cor:pk-structures-an}
For any 4-dimensional para-K\"ahler  structure $(M,g,K)$ the pseudo-Riemannian holonomy of the metric $g$ is reduced  from $\SOtt$ to $\glg_2(\bbR)$ via the representation $T$ in \eqref{ga1}. This holonomy  reduction is equivalent to  the property that $K$ is parallel with respect to the Levi-Civita connection $\nabla$ of $g.$ 
\end{corollary} 
\begin{remark}\label{rmk:pk-G-str}
The corollary above is a consequence of the so-called \emph{holonomy principle} in pseudo-Riemannian geometry which  establishes a one to one correspondence between the space of parallel sections of tensor bundles   and the invariant vectors in $\rT_xM$ under the action of the holonomy group $\mathrm{Hol}_x$ at each point $x\in M$.  We refer to \cite{BI-holonomy} for further discussion of the holonomy group of pseudo-Riemannian metrics of split signature.  
 
Let us also  point out that  in the spirit of  Remark \ref{rmk:2approaches},  Proposition \ref{pkprop} implies that   
  the bundle of adapted null frames for para-K\"ahler structures is a  principal
$\mathbf{GL}_2(\mathbb{R})$-bundle $\cF^8\to M$ obtained from reducing  the $\SOtt$-bundle $\cF\to M,$  with the property that the reduced first order structure equations given by \eqref{ca1}, \eqref{ca21}-\eqref{ca22} have no intrinsic torsion. 
\end{remark}
\subsubsection{Curvature of pK geometry and pK-Einstein (pKE) condition}
\label{sec:curv-pk-geom}
In this section we discuss  the curvature of para-K\"ahler structures.
\begin{proposition}
The curvature $Riemann^{20}$ of every 4-dimensional para-K\"ahler structure $(M,g,K)$ can be decomposed as
$$
Riemann^{20}=Ric^3_2\oplus\Big( Scal^1 \cong Weyl^1_3\Big)\oplus Weyl^5_1,
$$
which, compared to \eqref{eq:Decomp1}, means $Ric^3_1=Ric^3_3=Weyl^1_1=Weyl^1_2=Weyl^1_4=Weyl^1_5=0$. 

More explicitly, the curvature operator $Riemann$  in \eqref{eq:RiemmanCurv1}, expressed in  term of the basis of 2-forms $(\sigma^i_\pm)$  in  \eqref{asd-sd-null}, is given by
 $$\begin{aligned}
 Riemann=&-\Psi_2' \id_{6\times 6}+
 \bma[c|c|c||ccc]
\Psi_2'&0&0&&0&\\
\cmidrule(lr){1-6}
0&-2\Psi_2'&0&\Rho_{23}&\Rho_{13}-\Rho_{24}&-\Rho_{14}\\
\cmidrule(lr){1-6}
0&0&\Psi_2'&&0&\\
\cmidrule(lr){1-6}\morecmidrules\cmidrule(lr){1-6}
&2\Rho_{14}& &\Psi_2&2\Psi_1&\Psi_0\\
0&\Rho_{13}-\Rho_{24}&0&-\Psi_3&-2\Psi_2&-\Psi_1\\
&-2\Rho_{23}& &\Psi_4&2\Psi_3&\Psi_2
 \ema.\end{aligned}
$$
\end{proposition}
\begin{proof}
The proof follows from  a simple substitution of $\Gamma^4{}_1=0$ and $\Gamma^1{}_4=0$ into the last two of the second structure equations (\ref{ca21}). 
\end{proof}
\begin{remark}
Note that the curvature conditions 
$$\Rho_{11}=\Rho_{12}=\Rho_{22}=\Rho_{33}=\Rho_{34}=\Rho_{44}=\Psi_0'=\Psi_1'=\Psi_3'=\Psi_4'=\Psi_2'+\tfrac{1}{12}R=0,$$
implied by the para-K\"ahler condition $\Gamma^4{}_1=\Gamma^1{}_4=0$, when inserted to the second structure equations (\ref{ca21})-(\ref{ca22}), give that  
the entire curvature $\tfrac12R^a{}_{bcd}\theta^c\dz\theta^d$ of the para-K\"ahler structure is a $\gla_2(\bbR)$-valued 2-form  
i.e. $\tfrac12R^a{}_{bcd}\theta^c\dz\theta^d\in \Lambda^{(1,1)}.$
\end{remark}

Recall that a 4-dimensional pseudo-Riemannian manifold $(M,g)$ where  $g$ has split signature is called \emph{Einstein} if and only if  its traceless Ricci curvature vanishes, i.e. $\overset{\circ}{Ricci}=0$ in  (\ref{riedec}). Therefore, one obtains the following.
\begin{corollary}
The curvature of a 4-dimensional para-K\"ahler-Einstein structure decomposes to
$$
Riemann^{20}=\Big( Scal^1 \cong Weyl^1_3\Big)\oplus Weyl^5_1.
$$
When written in an adapted coframe it reads

\begin{equation}
  \label{eq:pKE-curv}
  \begin{aligned}
 Riemann=
 \bma[c||c]
 \begin{matrix}[c|c|c]
 0&0&0\\\cmidrule(lr){1-3}0&-3\Psi_2'&0\\\cmidrule(lr){1-3}0&0&0\end{matrix}&0\\
 \cmidrule(lr){1-2}\morecmidrules\cmidrule(lr){1-2}
 0& \begin{matrix}
 \Psi_2-\Psi_2'&2\Psi_1&\Psi_0\\-\Psi_3&-2\Psi_2-\Psi_2'&-\Psi_1\\\Psi_4&2\Psi_3&\Psi_2-\Psi_2'\end{matrix}
 \ema.\end{aligned}
\end{equation}
\end{corollary}
\begin{remark}\label{rmk:SD-Weyl-Curv}
It follows from \ref{sec:curv-pk-geom} that for para-K\"ahler-Einstein manifolds the two constant curvature components are related by $R=-12\Psi_2'.$
 From now on, we restrict ourselves to para-K\"ahler-Einstein 4-manifolds with  non-vanishing $Weyl^+$ i.e.  \emph{we always assume} 
$$\Psi_2'=\mathrm{const}\neq 0.$$
 Moreover, following  the discussion in \ref{sec:cart-penr-petr} on the Petrov type of the anti-self-dual Weyl curvature, $Weyl^-=Weyl^5_1,$ one obtains that the Petrov type of the quartic representation of $Weyl^+,$ as the self-dual Weyl curvature of the metric $g,$ is $D$ if $\Psi'_2\neq 0,$ and $O$ if $\Psi_2'=0.$
\end{remark}

\subsubsection{Para-K\"ahler structure in a coordinate system}
\label{sec:pk-struct-coord}
One of the  features of K\"ahler metrics is that they  can be locally expressed in terms of a function, called the \emph{K\"ahler potential}. An analogous feature for the para-K\"ahler structures in  4 dimensions is described in the following two propositions.
\begin{proposition}\label{pkc}
  Let $\cU$ be an open set of $\bbR^4$, and let $(a,b,x,y)$ be Cartesian coordinates in $\cU$. Consider a real-valued sufficiently differentiable function $V=V(a,b,x,y)$ on $\cU$ such that
  $$\det \bma V_{ax}& V_{ay}\\V_{bx}&V_{by}\ema\neq 0\quad\mathrm{in}\quad {\cU}.$$
  Define
  $$\begin{aligned}
    g=&~2\der a~(V_{ax}\der x +V_{ay}\der y)+2\der b~(V_{bx}\der x+V_{by}\der y),\\
    K=&~\partial_a\otimes\der a+\partial_b\otimes\der b-\partial_x\otimes\der x-\partial_y\otimes\der y,\\
    \rho=&~\der a\dz(V_{ax}\der x +V_{ay}\der y)+\der b\dz(V_{bx}\der x+V_{by}\der y).\end{aligned}$$
  Then the pair $(g,K)$ defines a para-K\"ahler structure on $\cU$ with $\rho(\cdot,\cdot)=g(K(\cdot),\cdot)$.

  The para-K\"ahler structure $({\cU},g,K)$ is Einstein  i.e. $Ric(g)=\Lambda g,$
  if and only if the potential function $V$ satisfies
  \begin{equation}
    \label{ec}
  \det \bma V_{ax}& V_{ay}\\V_{bx}&V_{by}\ema=c_1 c_2~ \mathrm{e}^{-\Lambda V}
  \end{equation}
  for a real number $\Lambda$ and real-valued functions $c_1=c_1(a,b)$, $c_2=c_2(x,y).$
\end{proposition}
\begin{proof} In the adapted coframe
\[\begin{aligned}
    \al^1&=\der a,\qquad    &\al^2&=\der b\\
    \bal^1&=V_{ax}\der x+V_{ay}\der y,\qquad   &\bal^2&=V_{bx}\der x+V_{by}\der y,
  \end{aligned}\]
   the  1-forms $\Gamma^a_{~b}$, constituting the $\gla_2(\bbR)$ part of the Levi-Civita  connection, read
  \[\begin{aligned}    \Gamma^1{}_1=&\frac{V_{aay}V_{bx}-V_{aax}V_{by}}{V_{ay}V_{bx}-V_{ax}V_{by}}\al^1+\frac{V_{aby}V_{bx}-V_{abx}V_{by}}{V_{ay}V_{bx}-V_{ax}V_{by}}\al^2\\
    \Gamma^1{}_2=&\frac{V_{aby}V_{bx}-V_{abx}V_{by}}{V_{ay}V_{bx}-V_{ax}V_{by}}\al^1+\frac{V_{bby}V_{bx}-V_{bbx}V_{by}}{V_{ay}V_{bx}-V_{ax}V_{by}}\al^2\\
    \Gamma^2{}_1=&\frac{V_{aay}V_{ax}-V_{aax}V_{ay}}{-V_{ay}V_{bx}+V_{ax}V_{by}}\al^1+\frac{V_{aby}V_{ax}-V_{abx}V_{ay}}{-V_{ay}V_{bx}+V_{ax}V_{by}}\al^2\\
     \Gamma^2{}_2=&\frac{V_{aby}V_{ax}-V_{abx}V_{ay}}{-V_{ay}V_{bx}+V_{ax}V_{by}}\al^1+\frac{V_{bby}V_{ax}-V_{bbx}V_{ay}}{-V_{ay}V_{bx}+V_{ax}V_{by}}\al^2.
  \end{aligned}\]
  It is straightforward to check that $\der\rho=0$, and  $\Gamma^1{}_4=\Gamma^4{}_1=0$, as it should be for the Levi-Civita connection in an adapted coframe of a para-K\"ahler structure.

  For the calculation of the Ricci tensor it is more convenient to work in the coordinate frame $(\der a,\der b,\der x, \der y)$ rather than in the adapted frame $(\al^1,\al^2,\bal^1,\bal^2)$. Thus, we need to  display the Levi-Civita connection 1-forms in the coordinate frame as well. Let us use the following notation for the coordinates
  $$x^A=(a,b), \quad x^{\dot{A}}=(x,y),\quad A=1,2,\quad \dot{A}=1,2.$$
  It follows that
  $$\Gamma^A{}_{\dot{A}}=\Gamma^{\dot{A}}{}_A=0,$$
  and the connection 1-forms $\Gamma^A{}_B$ and $\Gamma^{\dot{A}}{}_{\dot{B}}$ are given by \eqref{eq:Conn1Form-LocCoord} in the Appendix.

  Using the expressions for the Levi-Civita connection in \eqref{eq:Conn1Form-LocCoord} the curvature can be calculated easily and  the Ricci tensor satisfies
  $$R_{AB}=R_{\dot{A}\dot{B}}=0,$$
  and
  \[R_{A\dot{A}}=R_{\dot{A}A}=-\frac{\partial^2}{\partial_{x^A}\partial_{x^{\dot{A}}}}\log\big(V_{ax}V_{by}-V_{ay}V_{bx}\big).\]
  Since in the $(A,\dot{A})$ notation the metric $g$ reads as $$g=\sum_{A,\dot{A}=1,2}\frac{\partial^2V}{\partial_{x^A}\partial_{x^{\dot{A}}}}(\der x^A\otimes\der x^{\dot{A}}+\der x^{\dot{A}}\otimes\der x^A),$$  the Einstein equations are $$-\frac{\partial^2}{\partial_{x^A}\partial_{x^{\dot{A}}}}\log\big(V_{ax}V_{by}-V_{ay}V_{bx}\big)=\Lambda ~\frac{\partial^2V}{\partial_{x^A}\partial_{x^{\dot{A}}}},$$
 which after integration give
  $$V_{ax}V_{by}-V_{ay}V_{bx}=c_1 c_2~\mathrm{e}^{-\Lambda V}.$$
\end{proof}
There is a converse to this proposition:
\begin{proposition}
  Every para-K\"ahler structure $(M,g,K)$ in dimension four is locally expressible in terms of a para-K\"ahler potential function $V$ as in Proposition \ref{pkc}.
\end{proposition}
\begin{proof}
The integrability of the distributions  $\cH=\mathrm{Ker}\{\al^1,\al^2\}$ and $\bar{\cH}=\mathrm{Ker}\{\bal^1,\bal^2\}$ in a para-K\"ahler structure implies that in some neighbourhood ${\cU}\subset M$ there exists a coordinate system $(a,b,x,y)$ such that
\[\begin{matrix}[ll]
  \al^1=A_1\der a+B_1\der b,\quad &\al^2=A_2\der a+B_2\der b\\
  \bal^1=P_1\der x+Q_1\der y,\quad &\bal^2=P_2\der x+Q_2\der y,
\end{matrix}\]
for some  functions $A_i,B_i,P_i,Q_i$ defined in $\cU$. Since the coframe $(\al^1,\al^2,\bal^1,\bal^2)$ is defined up to the $\glg_2(\bbR)$ action  \eqref{ga1}, we can use this transformation to bring the coframe into the form
$$\begin{matrix}[ll]
  \al^1=\der a,\quad &\al^2=\der b\\
  \bal^1=P\der x+Q\der y,\quad&\bal^2=R\der x+S\der y,
\end{matrix}$$
with new functions $P,Q,R,S$ on $\cU$ such that $PS-QR\neq 0$. In this new adapted frame we have $\der\al^1=0$, $\der\al^2=0$. Inserting this into \eqref{ca1} with $\Gamma^1_4=\Gamma^4_1=0,$ we get
\[\Gamma^1{}_1\dz\al^1+\Gamma^1{}_2\dz\al^2=0\quad\& \quad\Gamma^2{}_1\dz\al^1+\Gamma^2{}_2\dz\al^2=0,\]
which implies
\[\begin{aligned}
  \Gamma^1{}_1&=a_1\al^1+a_2\al^2,\qquad   \Gamma^1{}_2=a_3\al^1+a_4\al^2\\
  \Gamma^2{}_1&=a_5\al^1+a_6\al^2,\qquad  \Gamma^2{}_2=a_7\al^1+a_8\al^2,
  \end{aligned}\]
for some unknown functions $a_1,a_2,\dots,a_8$ on $\cU$. Inserting this back into the last two  of the structure equations \eqref{ca1} gives
\[\begin{aligned}
  0=&(Q_x-P_y)\der x\dz\der y+\\&\big(P_b-a_2P-a_6R\big)\der b\dz\der x+\big(Q_b-a_2Q-a_6S\big)\der b\dz \der y+\\&\big(P_a-a_1P-a_5R\big)\der a\dz\der x+\big(Q_a-a_1Q-a_5S\big)\der a\dz\der y\\
  0=&(S_x-R_y)\der x\dz\der y+\\&\big(R_b-a_8R-a_4P\big)\der b\dz\der x+\big(S_b-a_8S-a_4Q\big)\der b\dz \der y+\\&\big(R_a-a_7R-a_3P\big)\der a\dz\der x+\big(S_a-a_7S-a_3Q\big)\der a\dz\der y.
\end{aligned}
\]
This in particular means that
\[(Q_x-P_y)=0\quad\mathrm{and}\quad (S_x-R_y)=0.\]
As a result, locally, there exist functions $U$ and $W$ on $\cU$ such that
\[Q=U_y,\quad P=U_x,\quad S=W_y,\quad R=W_x.\]
  Thus, one obtains
\[\begin{matrix}[ll]
  \al^1=\der a,\quad &\al^2=\der b\\
  \bal^1=U_x\der x+U_y\der y,\quad&\bal^2=W_x\der x+W_y\der y.
\end{matrix}\]
 Since the 2-form $\rho$ is given by
 \[\rho=\al^1\dz \bal^1+\al^2\dz\bal^2,\]
the  para-K\"ahler condition $\der\rho=0$ implies 
\[0=\der\rho=(W_a-U_b)_y\der a\dz\der b\dz \der y+(W_a-U_b)_x\der a\dz\der b\dz\der x.\]
 This means that $W_a-U_b=f(a,b)$ for some function $f$ of variables $a,b$ only. But since in the coframe $(\al^1,\al^2,\bal^1,\bal^2)$ functions $W$ and $U$ appear only in terms of their $x$ and $y$ derivatives, they can be chosen so that  $f(a,b)\equiv 0.$   Hence there exists a differentiable function $V=V(a,b,x,y)$ on $\cU$ such that
  \[W=V_b\quad\mathrm{and}\quad U=V_a.\]
  Thus, the adapted coframe can be expressed as
\[\begin{matrix}[ll]
  \al^1=\der a,\quad &\al^2=\der b,\\
  \bal^1=V_{ax}\der x+V_{ay}\der y,\quad&\bal^2=V_{bx}\der x+V_{by}\der y.
  \end{matrix}\]
Expressing $g$, $K$ and $\rho$ in terms of the adapted coframe as in  (\ref{ac2})-(\ref{ac3}) gives the Proposition.
\end{proof}
\subsubsection{Homogeneous models}
\label{sec:simple-examples}
Two particular solutions of the Einstein condition (\ref{ec}) are given by the potentials
  \begin{itemize}
  \item[i)] $V_1=-\frac{1}{\Psi_2'}\log(b+a x-y)$,
    \item[ii)] $V_2=-\frac{2}{3\Psi_2'}\log\big((1-\tfrac32 \Psi_2' ax)(1-\tfrac32\Psi_2'by)\big)$, 
  \end{itemize}
  where $\Psi_2'=\mathrm{const}\neq 0$.
  
  Both potentials are solutions of \eqref{ec} with $\Lambda=-3\Psi_2'$ and  $c_1c_2=\tfrac{1}{\Psi_2'{}^2}$ for $V_1$, and $c_1c_2=1$ for $V_2$. Let the para-K\"ahler structure $(g_i,K_i,\rho_i)$ correspond to $V_i$ where $i=1,2.$ Then, in an open set of $\bbR^4$ parametrized by $(a,b,x,y)$ one finds
\[ K_1=K_2=\partial_a\otimes\der a+\partial_b\otimes\der b-\partial_x\otimes\der x-\partial_y\otimes\der y.\]
 Straightforward computation gives
\[\begin{aligned}
    g_1=&\textstyle{\frac{2\der a \big((y-b)\der x-x\der y\big)+2\der b\big(a\der x-\der y\big)}{\Psi_2'(b+a x-y)^2}},\qquad     \rho_1=&\textstyle{\frac{\der a \dz\big((y-b)\der x-x\der y\big)+\der b\dz \big(a\der x-\der y\big)}{\Psi_2'(b+a x-y)^2}}\end{aligned}
  \]
    and 
\[\begin{aligned}
    g_2=&\textstyle{\frac{2\der a \der x}{(1-\tfrac32\Psi_2'ax)^2}+\frac{2\der b\der y}{(1-\tfrac32\Psi_2'bx)^2}},\qquad     \rho_2=&\textstyle{\frac{\der a \dz\der x}{(1-\tfrac32\Psi_2'ax)^2} +\frac{\der b\dz \der y}{(1-\tfrac32\Psi_2'by)^2}}.\end{aligned}
   \] 
   The potential $V_1$ corresponds to the homogeneous para-K\"ahler-Einstein structure referred to as  \emph{the dancing metric} in \cite{BN-rolling} which is the unique homogeneous model that is self-dual, i.e. $Weyl^-=0,$ and not anti-self-dual for which $\Psi'_2=1.$ The potential $V_2$ corresponds to the only other homogeneous para-K\"ahler-Einstein structure. It has the property that the Petrov type of $Weyl^-$ is \emph{D.} A  derivation of these potential functions is outlined in \ref{sec:petrov-type-d} and \ref{sec:type-o}. Finding explicit examples of  pKE structures in terms of potential functions satisfying the PDE \eqref{ec} is not an easy task. In the next section we use an alternative technique to give more explicit examples of pKE structures.
   
\section{Para-K\"ahler-Einstein (pKE) metrics in dimension 4}  
\label{sec:pke-4-manifolds}
This section is the heart of the article, in which we describe pKE structures as Cartan geometries, give an in-depth study when the Petrov type is \emph{real} and \emph{special} and provide explicit examples. 
To be more specific, in \ref{sec:pk-structures-as} pKE structures are interpreted as Cartan geometries of type $(\mathbf{SL}_3(\RR),\mathbf{GL}_2(\RR))$. If the Einstein constant is $-3,$ then they satisfy the Yang-Mills equations for the associated $\sla_3(\bbR)$-valued \emph{Cartan} connection.
In \ref{sec:cart-reduct-symm}  we focus on pKE structures for which $Weyl^-$ has special real Petrov type and give examples of each type.  In particular, we find all homogeneous models, give a local normal form for all real Petrov type $D$ pKE metrics and present examples of real Petrov  type   \emph{II} that satisfy Yang-Mills equations. Moreover, we use Cartan-K\"ahler machinery to find the local generality of all Petrov types assuming analyticity. 
   
\subsection{Cartan geometries of type 
  $(\mathbf{SL}_3(\RR), \mathbf{GL}_2(\RR))$}
\label{sec:pk-structures-as}
In order to view pKE structures as a Cartan geometry, let us specialize the EDS \eqref{ca1}, \eqref{ca21}-\eqref{ca22} to the case of para-K\"ahler-Einstein metrics. We have
\begin{equation}\label{pk}
  \begin{aligned}
\der\al^1&=-\Gamma^1_{~1}\dz\al^1-\Gamma^1_{~2}\dz\al^2,\\
\der\al^2&=-\Gamma^2_{~1}\dz\al^1-\Gamma^2_{~2}\dz\al^2,\\
\der\bal^1&=\Gamma^1_{~1}\dz\bal^1+\Gamma^2_{~1}\dz\bal^2,\\
\der\bal^2&=\Gamma^1_{~2}\dz\bal^1+\Gamma^2_{~2}\dz\bal^2,\\
\exd\Gamma^1_{~1}&=-\Gamma^1_{~2}\w\Gamma^2_{~1}+(-2\Psi'_2-\Psi_2)\al^1\w\bal^1+\Psi_1\al^1\w\bal^2-\Psi_3\al^2\w\bal^1+ (\Psi_2-\Psi'_2)\al^2\w\bal^2\\
\exd\Gamma^1_{~2}&=-\Gamma^1_{~1}\w\Gamma^1_{~2} -\Gamma^1_{~2}\w\Gamma^2_{~2}-\Psi_3\al^1\w\bal^1+(\Psi_2-\Psi'_2)\al^1\w\bal^2-\Psi_4\al^2\w\bal^1+\Psi_3\al^2\w\bal^2\\
\exd\Gamma^2_{~1}&=\Gamma^1_{~1}\w\Gamma^2_{~1}+\Gamma^2_{~1}\w\Gamma^2_{~2}+\Psi_1\al^1\w\bal^1-\Psi_0\al^1\w\bal^2 +(\Psi_2-\Psi'_2)\al^2\w\bal^1 -\Psi_1\al^2\w\bal^2\\
\exd\Gamma^2_{~2}&=\Gamma^1_{~2}\w\Gamma^2_{1}+(\Psi_2-\Psi'_2)\al^1\w\bal^1-\Psi_1\al^1\w\bal^2+\Psi_3\al^2\w\bal^1 +(-2\Psi'_2-\Psi_2)\al^2\w\bal^2\\
\end{aligned}
\end{equation}
where we have used  $\Gamma^1_{~4}=\Gamma^4_{~1}=0$ from Proposition \ref{pkprop}

As discussed in Remark \ref{rmk:pk-G-str} and \ref{rmk:2approaches}, the EDS \eqref{pk} can be regarded as the structure equations for the coframe on the principal $\mathbf{GL}_2(\mathbb{R})$-bundle $\cF^8\to M$ which is the  8-dimensional \emph{bundle of adapted null frames} for para-K\"ahler-Einstein structures. In fact, one can show that para-K\"ahler-Einstein structures correspond to  Cartan geometries of type    $(\mathbf{SL}_3(\mathbb{R}),\mathbf{GL}_2(\mathbb{R}))$. First let us define a Cartan geometry \cite{Sharpe,CS-Parabolic}.
\begin{definition}\label{def:cartan-conn-def}
  A \emph{Cartan geometry} $(\cG,S,\psi),$ of type $(G,H)$ is a principal $H$-bundle $\cG\to S,$ equipped with a $\mathfrak{g}$-valued 1-form $\cA,$ which is a \emph{Cartan connection}, i.e., 
  \begin{enumerate}
    \item $\cA_u:T_u\cG\rightarrow \mathfrak{g}$ is linear isomorphism for all $u\in \cG.$
    \item $\cA$ is $H$-equivariant, i.e., $R_h^*\cA=\mathrm{Ad}(h^{-1})\circ \cA,$ where $R_h$ denotes the right action by $h\in H.$ 
    \item $\cA(X_v)=v,$  for every fundamental vector field $X_v$ of $\tau:\cG\rightarrow S, v\in \mathfrak h.$
  \end{enumerate}
  The curvature of the Cartan connection $\cA$ is given by $K_\cA=\exd\cA+ \cA\w\cA \in\Omega^2(\cG,\mathfrak{g})$ which is horizontal and defines the \emph{curvature function} $\kappa_\cA : \cG\to \bigwedge^2(\mathfrak{g}/\mathfrak{h})^*\otimes\mathfrak{g}.$  
\end{definition}
 Let us now consider the $\mathfrak{sl}_3(\mathbb{R})$-valued 1-form
\begin{equation}
  \label{eq:SL3-pke-conn-A}
\cA:=\bma[c|c] 
 \Gamma-\tfrac13\mathrm{Tr}(\Gamma) \id_{2\times 2}&\al\\
 \cmidrule(lr){1-2}
 \bal&-\tfrac13\mathrm{Tr}(\Gamma)\ema,  
\end{equation}
where
\begin{equation}
  \label{eq:Gammas-A-pke-conn}
   \Gamma=\bma
 \Gamma^1{}_1&\Gamma^1{}_2\\
 \Gamma^2{}_1&\Gamma^2{}_2\ema,\qquad \al=\bma \al^1\\\al^2\ema,\qquad \bar\al=(\bal^1,\bal^2),
\end{equation}
Using $\cA$, the structure equations \eqref{pk} can be expressed as
\begin{equation}
  \label{eq:curvature2form}
K_\cA=\der \cA+\cA\dz \cA,
\end{equation}
where, using the 2-forms $\sigma_+^i$'s and $\sigma_-^i$'s in \eqref{asd-sd-null}, one has
\begin{equation}
\label{curvsl3}
\begin{aligned}
K_\cA=&\bma[c|c]
\begin{matrix}\Psi_1&1+\Psi_2-\Psi_2'\\
-\Psi_0&-\Psi_1\end{matrix}&0\\
\cmidrule(lr){1-2}
0&0
\ema \sigma^1_-+\bma[c|c]
\begin{matrix}\Psi_3&\Psi_4\\
-1-\Psi_2+\Psi_2'&-\Psi_3\end{matrix}&0\\
\cmidrule(lr){1-2}
0&0
\ema\sigma^2_-+\\
&\bma[c|c]
\begin{matrix}\tfrac12(1-2\Psi_2-\Psi_2')&-\Psi_3\\
\Psi_1&\tfrac12(\Psi_2'+2\Psi_2-1)\end{matrix}&0\\
\cmidrule(lr){1-2}
0&0
\ema\sigma^3_-+\\&\tfrac12(1-\Psi_2')\bma[c|c] \id_{2\times 2}&0\\
\cmidrule(lr){1-2}
0&-2\ema\sigma^3_+.
\end{aligned}
\end{equation}
One can interpret the 1-form $\cA$ as an $\sla_3(\bbR)$-valued Cartan connection on the principal $\mathbf{GL}_2(\mathbb{R})$-bundle $\cF^8\to M$ of null frames adapted to a pKE structure.  As a result, one obtains the following theorem.  
\begin{theorem}\label{so22}
Every pKE 4-manifold  defines a Cartan geometry of type $(\slg_3(\bbR),\glg_2(\bbR))$ for which the structure  bundle
$$\glg_2(\bbR)\to \cF^8\to M$$
is the bundle of adapted null frames. The curvature $K_\cA$ vanishes, i.e. the Cartan geometry is flat, if and only if  the Einstein constant is  -3 and the anti-self-dual Weyl tensor vanishes, i.e.
\[\Psi'_2=1\qquad\mathrm{and}\qquad Weyl^-=0.\] The flat model, i.e. $K_\cA=0,$ is locally equivalent to the para-K\"ahler-Einstein structure induced by the   dancing metric discussed in  \ref{sec:simple-examples}. 
\end{theorem}
Recall that  the exterior derivative of \eqref{eq:curvature2form} gives  the Bianchi identity 
\begin{equation}
  \label{eq:BianchiIden}
  DK_\cA:=\der K_\cA+\cA\dz K_\cA-K_\cA\dz \cA=0.
\end{equation}
As a result of the theorem above one obtains the following. 
\begin{proposition}\label{so222}
A 4-dimensional para-K\"ahler-Einstein structure satisfies the Cartan connection Yang-Mills equations, $D*K_\cA=0$,  if and only if  $\Psi_2'=1.$
\end{proposition}
\begin{proof}
By \eqref{curvsl3} and the definition of self-dual and anti-self-dual null planes in  \eqref{asd-sd-null},  it follows that 
\[\Psi'_2=1\Longleftrightarrow *K_\cA=-K_\cA.\]
Since the curvature $K_\cA$ of a Cartan connection is always horizontal (see Definition \ref{def:cartan-conn-def} and \eqref{curvsl3}), one can apply the Hodge star to  $K_\cA$ defined on ${\cF}^8$.   
Applying  the Bianchi identity \eqref{eq:BianchiIden}, one obtains $D*K_\cA=-DK_\cA=0.$ Alternatively, by taking a section $s\colon M\to\cF^8,$ computing the curvature and applying the Hodge star    one can verify the claim.
 
Conversely, it is  a matter of straightforward computation to show that  $D*K_\cA=0$ combined with the Bianchi identities \eqref{eq:BianchiIden}, or equivalently equations \eqref{bi}, and the EDS \eqref{pk} imply $\Psi'_2=1.$
\end{proof}
\begin{remark}
    Note that pKE structures can also be associated to Cartan geometries of type $(\RR^4\rtimes \mathbf{GL}_2(\RR),\mathbf{GL}_2(\RR))$  whose flat model satisfies $\Psi'_2=0$ and $Weyl^-=0$.  This point of view is however not desirable for the purpose of this article, since we always assume $\Psi_2'\neq 0.$  
  \end{remark}

\subsection{Cartan reduction: homogeneous models, examples and local generality} 
\label{sec:cart-reduct-symm}
In this section we carry out the Cartan reduction procedure for pKE metrics whose anti-self-dual Weyl curvature has non-generic real Petrov type. Our reduction will not be exhaustive and will omit Petrov type $G$. We will  describe the reduction procedure for Petrov type $II$ and $D$ in detail, give a complete local normal form for type $D$ pKE metrics, and  use the same method to find examples of  Petrov types $III$ and $N$. The reduction allows us to find all homogeneous models and use the Cartan-K\"ahler theory to find the local generality of all  Petrov types assuming analyticity.
\subsubsection{Reduction for special real Petrov types}\label{sec:redii} 
Recall from \eqref{eq:quartic_ASD_Weyl}  the quartic 
\begin{equation}
  \label{eq:quartic2}
  W(\lambda)=\Psi_4\lambda^4+4\Psi_3\lambda^3+6\Psi_2\lambda^2+4\Psi_1\lambda+\Psi_0
\end{equation}
where $\lambda$ is the affine parameter for the anti-self-dual null planes in \eqref{eq:sd-asd-parametrized} for a given choice of adapted coframe. Taking a different  choice of coframe, $\tilde{\theta_1},\dots,\tilde\theta_4,$ by the action of the structure group, as in \eqref{ga1}, one obtains a quartic whose coefficients,$\tilde\Psi_0,\dots,\tilde\Psi_4,$ can be expressed in terms of $\Psi_i$'s and the elements $a_{11},a_{12},a_{21},a_{22}$ of the matrix $A\in\glg_2(\bbR).$ For instance, one obtains
\begin{equation}
  \label{eq:tilde-psi_0}
  \tilde\Psi_0=\textstyle{\frac{1}{\mathrm{det}(A)}(a^4_{21}\Psi_4-4a_{21}^3a_{22}\Psi_3+6a_{21}^2a_{22}^2\Psi_2 -4a_{21}a_{22}^3\Psi_1+a_{22}^4\Psi_0)}.
\end{equation}
\begin{remark}\label{rmk:infinit-action-psi0}
  Note that the infinitesimal form of  the   transformation law for $\Psi_0$ given in \eqref{eq:tilde-psi_0} is represented by the Bianchi identity for $\Psi_0$ in  \eqref{eq:bi2}. For a discussion on obtaining the group action from its infinitesimal see \cite{Gardner}.
\end{remark}
It is clear from \eqref{eq:tilde-psi_0} that if $W(\lambda_0)=0$ then with respect to the coframe obtained from the action of \eqref{ga2} where $a_{11}=a_{22}=1,a_{12}=0,$ and $a_{21}=-\lambda_0$, the root $\lambda_0$ would be translated to zero, i.e. $\tilde\Psi_0=0$ in this choice of coframe.

If $Weyl^-$ has a repeated root then by our discussion above, there is a coframe adaptation with respect to which the double root is translated to zero. 
If the  root has multiplicity $k\leq 4$, in the newly adapted coframe we have 
\begin{equation}
\label{eq:typeII-translated}
  \Psi_0=\dots=\Psi_{k-1}= 0,\quad \mathrm{and}\quad \Psi_k\neq 0.
\end{equation}
Using the group action on $\Psi_i$'s or, equivalently the Bianchi identities  \eqref{bi}, it follows that the bundle of adapted coframes that preserves the condition \eqref{eq:typeII-translated} gives rise to a 7-dimensional principal bundle $\cF^7\to M.$ More precisely, the new adapted coframes were determined by a choice of $a_{21}$ in \eqref{ga2} as a result of which the  structure group is reduced to $\HH_{(1)}\subset \SOtt$ defined as
\begin{equation}
  \label{eq:typeII-red-str-gp}
\HH_{(1)}=\left\{T(U)=
  \begin{pmatrix}
    A & 0\\
    0 & -A^T
  \end{pmatrix}\ \vline\ \ A=
  \begin{pmatrix}
  a_{11} & a_{12}\\
  0   & a_{22}
  \end{pmatrix}\in \mathbf{GL}_2(\mathbb{R})\right\}.
\end{equation}
Using the    gauge transformations \eqref{troa2} arising from the structure group $\HH_{(1)}$ for   adapted coframes with respect to which \eqref{eq:typeII-translated} holds one obtains that the transformation of the connection 1-form $\Gamma^2_{~1}$ does not involve the inhomogeneous terms $\exd(T(U))T(U)^{-1}$ and  therefore the following proposition holds. 
\begin{proposition}\label{typ2} 
Given a pKE metric, if the anti-self-dual Weyl curvature has special real Petrov type, i.e. it  has a repeated root whose multiplicity is at least 2, then the bundle of adapted coframes which preserves the condition 
\eqref{eq:typeII-translated} is given by a principal $\HH_{(1)}$-bundle $\cF^7\to M.$ Adapted coframes $(\alpha,\bar\alpha)$ arising as sections of $\cF^7$ satisfy \eqref{pk} where 
\begin{equation}
  \label{eq:Gamma21-reduced}
  \Gamma^2_{~1}=J_1\al^2+J_2\bal^1,
\end{equation}
for  some functions $J_1$ and $J_2$ on $M$. 
   \end{proposition}
   \begin{proof}
        The proof of the proposition  simply follows from the Bianchi identity for $\Psi_{k-1}$ in \eqref{bi}  by inserting \eqref{eq:typeII-translated}, which results in \eqref{eq:Gamma21-reduced}.
   \end{proof}
\begin{remark}
As a result of  the proof above, one can express the quantities $J_1$ and $J_2$ in \eqref{eq:Gamma21-reduced}  in terms of $\Psi_{ij}$'s in  \eqref{bi}. For instance, for $k=2$ in \eqref{eq:typeII-translated} one obtains 
\[J_1=\textstyle{\frac{-1}{3\Psi_2}\Psi_{21}},\qquad J_2=\textstyle{\frac{1}{3\Psi_2}\Psi_{24}}.\]
where $\Psi_{24}$ and $\Psi_{21}$ are referred to as the \emph{coframe derivatives} of $\Psi_2.$  However, we will not take this point of view in order to avoid fractional expressions.  
\end{remark}
A second coframe adaptation will result in the following proposition.
\begin{proposition}\label{rei} 
  Every 4-dimensional pKE structure whose anti-self-dual Weyl tensor is of special real Petrov type  defines  a Cartan geometry of type $(\SOtt,\mathbf{T}^2)$  where $\mathbf{T}^2\cong \bbS^1\times\bbS^1$ is the maximal torus in $ \SOtt.$ The structure equations are given by \eqref{pk} where 
\begin{equation}
  \label{eq:J16-typeIID}
  \begin{aligned}
\Gamma^2_{~1}=&J_1\al^2+J_2\bal^1,\\
  \Gamma^1_{~2}=&-J_3\al^1+J_6\al^2+J_5\bal^1+J_4\bal^2
\end{aligned}
\end{equation}
for some functions $(J_1,J_2,J_3,J_4,J_5,J_6)$ on $M.$ The $\sott$-valued Cartan connection can be represented as 
\begin{equation}
\label{cacoB}
\cB=\bma[c|c]\begin{matrix}\textstyle{\frac{1}{2}\Gamma^1_{~1}}&\sqrt{\tfrac32|\Psi_2'|}~\al^1\\ \sqrt{\tfrac32|\Psi_2'|}~\bal^1&-\textstyle{\frac{1}{2}\Gamma^1_{~1}}\end{matrix}&0\\
\cmidrule(lr){1-2}
0&\begin{matrix}\textstyle{\frac{1}{2}\Gamma^2_{~2}}&\sqrt{\tfrac32|\Psi_2'|}~\al^2\\\sqrt{\tfrac32|\Psi_2'|}~\bal^2&-\textstyle{\frac{1}{2}\Gamma^2_{~2}}
\end{matrix}
\ema.
\end{equation}
The flatness condition for the resulting Cartan geometry, $K_{\cB}=\der \cB+\cB\dz \cB=0, $    is equivalent to     $J_1= J_2= J_3= J_4= J_5= J_6=\Psi_4= 0$, and $\Psi_2= \Psi_2'$ and implies that $Weyl^-$ has Petrov type $D$.  
  \end{proposition}
\begin{proof}
 Assuming that the Petrov type is $II$ or $D$, the adaptation \eqref{eq:typeII-translated} reads
 \begin{equation}
   \label{eq:typeII-psi01-translated}
   \Psi_0=\Psi_{1}=0,\quad \mathrm{and}\quad \Psi_2\neq 0.
\end{equation}
If the conditions \eqref{eq:typeII-psi01-translated} are to be preserved,   by Proposition \ref{typ2}, one has \eqref{eq:Gamma21-reduced} and, using the Bianchi identities \eqref{bi},  one obtains the  differential relations
\begin{equation}
\label{pkIIa}
 \begin{aligned}
  \der J_1=& J_1\Gamma^1_{~1}-J_1^2\al^1+J_{12}\al^2+J_{13}\bal^1+J_1J_2\bal^2,\\
  \der J_2=&-J_2\Gamma^2_{~2}-J_1J_2\al^1+(J_{13}+\Psi_2-\Psi_2')\al^2+J_{23}\bal^1+J_2^2\bal^2,\\
    \der\Psi_2'=&0,\\
    \der\Psi_2=&-3J_1\Psi_2\al^1+(2J_1\Psi_3+\Psi_{31})\al^2+(2J_2\Psi_3-\Psi_{34})\bal^1+3J_2\Psi_2\bal^2,\\
    \der\Psi_3=&3\Psi_2\Gamma^1_{~2}-\Psi_3(\Gamma^1_{~1}-\Gamma^2_{~2})-\Psi_{31}\al^1-(J_1\Psi_4+\Psi_{41})\al^2+(J_2\Psi_4-\Psi_{44})\bal^1 -\Psi_{34}\bal^2,\\
    \der\Psi_4=&4\Psi_3\Gamma^1_{~2}-2\Psi_4(\Gamma^1_{~1}-\Gamma^2_{~2})+\Psi_{41}\al^1+\Psi_{42}\al^2+\Psi_{43}\bal^1+\Psi_{44}\bal^2
 \end{aligned}
  \end{equation}
for some functions $J_{12},J_{13},J_{23}$ on $M.$  These relations are obtained by inserting \eqref{eq:Gamma21-reduced}  in the structure equations \eqref{pk} and requiring that the exterior derivative of the right hand side of the equations are zero.

Let the quantities $J_i,\Psi_j,\Psi'_2$ and $\tilde J_i,\tilde\Psi_j,\tilde\Psi'_2$ be the quantities appearing in the structure equations for two choices of   such adapted coframe related via $\tilde\theta^a=T(U)^a{}_b\theta^b$. Using the structure group $\HH_{(1)}$ as in \eqref{eq:typeII-red-str-gp} and its induced gauge transformation \eqref{troa}, one obtains
 \begin{equation}
\label{pkIIb}
\begin{aligned}
\tilde J_1&= a_{11}^{-1}J_1,\qquad  &&\tilde J_2=a_{22} J_2,\\
\tilde\Psi_2'&=\Psi_2',\qquad   &&\tilde\Psi_2=\Psi_2,\\
\tilde\Psi_3&=\textstyle{\frac{1}{a_{22}}(-3a_{12}\Psi_2+a_{11}\Psi_3)},\qquad &&\tilde\Psi_4=\textstyle{\frac{1}{a_{22}^2} (6a_{12}^2\Psi_2-4a_{11}a_{12}\Psi_3+a_{11}^2\Psi_4)}.
    \end{aligned}
 \end{equation}
As was mentioned in Remark \ref{rmk:infinit-action-psi0}, the infinitesimal version of  the transformations above are given by the Bianchi identities \eqref{pkIIa}.

Consequently, using the transformation law for $\Psi_3$ given in \eqref{pkIIb}, we can further restrict to the bundle of adapted coframes with respect to which 
\begin{equation}
  \label{eq:tyepII-reduction-2}
\Psi_0=\Psi_1=\Psi_3=0,\qquad \Psi_2\neq 0.
\end{equation}
This can be seen explicitly from \eqref{pkIIb} by setting
\[a_{12}=\frac{\Psi_3}{3\Psi_2}a_{11}.\]
As a result, when the Petrov type of $Weyl^-$ is $II$ or $D,$ the bundle of adapted coframes preserving \eqref{eq:tyepII-reduction-2} gives rise to a principal $\HH_{(2)}$-bundle $\cF^6\to M$ where
\begin{equation}
  \label{eq:typeII-red-str-gp-2}
\HH_{(2)}:=\left\{T(U)=
  \begin{pmatrix}
    A & 0\\
    0 & -A^T
  \end{pmatrix}\ \vline\ \ A=
  \begin{pmatrix}
  a_{11} & 0\\
  0   & a_{22}
  \end{pmatrix}\in \mathbf{GL}_2(\mathbb{R})\right\}.
\end{equation}
Since for such coframes the gauge transformations \eqref{troa2} do not affect $\Gamma^1_{~2}$ and $\Gamma^2_{~1}$ by the inhomogeneous term $\exd(T(U))T(U)^{-1}$,  
one obtains the expressions \eqref{eq:J16-typeIID} for some functions $(J_1,J_2,J_3,J_4,J_5,J_6)$ satisfying the Bianchi identities \eqref{eq:Bianchies-TypeIID-1}.

    It follows that the set of coframes adapted to the condition \eqref{eq:tyepII-reduction-2} gives a principal $\mathbf{T}^2$-bundle which is  equipped by the Cartan connection $\cB.$ The flatness condition follows from a straightforward computation. 

For real  Petrov types $III$ and $N$ one  needs to find the appropriate reduction of the structure bundle for pKE metrics and  proceeds analogously to find the Cartan connection \eqref{cacoB}. If the Petrov type of $Weyl^-$ is $III,$ the desired principal $\mathbf{T}^2$-bundle is given by the set of adapted null coframes with respect to which
  \[\Psi_0=\Psi_1=\Psi_2=\Psi_4=0,\qquad \Psi_3\neq 0.\]
If the Petrov type is $N,$ by Proposition \ref{typ2} one can consider the bundle of adapted coframes with respect to which
  \[\Psi_3=\Psi_2=\Psi_1=\Psi_0=0,\qquad \Psi_4\neq 0.\]
  For such adapted  coframes the 1-forms $\Gamma^2_{~1}$ is reduced as in \eqref{eq:Gamma21-reduced}. Subsequently, one can use  identities \eqref{bi} to obtain
  \[
    \begin{aligned}
  \exd\Psi_{43} &=(2\Gamma^2_{~2}-3\Gamma^1_{~1})\Psi_{43}-5J_2\Psi_4\Gamma^1_2+(-3\Psi_2'\Psi_4-\Psi_{442}) \al^1\\
 &   +\Psi_{432}\al^2+\Psi_{433}\bal^1+\Psi_{434}\bal^2,
\end{aligned}\]
where $\Psi_{43}$ is the coframe derivative of $\Psi_4$ with respect to $\bal^1$. Therefore  the principal $\mathbf{T}^2$-bundle $\cF^6\to M$ is given by  adapted null coframes with respect to which one additionally has $\Psi_{43}=0.$ After these reductions  $\Gamma^1_2$ and $\Gamma^2_{~1}$ can be expressed as \eqref{eq:J16-typeIID} for some functions $J_1,\dots,J_6$ on $M.$ We will not express the Bianchi identities for $J_i$'s when the  Petrov type is $III$ or $N.$ It is a matter of straightforward computation to show that  $K_\cB=0$ for real Petrov types $III$ and $N$ implies $\Psi'_2=0$ which is not being considered in this article.  

  \end{proof}
\subsubsection{Petrov type $D$}
\label{sec:petrov-type-d}
PKE metrics of real Petrov type $D$ are particularly interesting because we find   an explicit local normal form for them as shown below. Suppose that  $Weyl^-$ has Petrov type $D$ i.e.  the quartic \eqref{eq:quartic2} has two real roots with double multiplicity.   
Although Proposition \ref{rei} describes such pKE metrics as  a Cartan  geometry of type $(\SOtt,\mathbf{T}^2)$,  here we carry out the reduction procedure further and describe all such pKE metrics in some normal coordinate system. First,  we have the following.
  \begin{proposition}\label{thm:typeD-AllCases-NormalForm}
Given a pKE metric whose  anti-self-dual Weyl curvature has real Petrov type $D$ everywhere, the Cartan connection $\cB$ on $\cF^6\to M$ satisfies the structure equations \eqref{pk} where 
\begin{equation}
  \label{eq:Gamma1221-reduced}
  \Gamma^1_{~2}=-J_3\theta^1+J_4\theta^4,\qquad \Gamma^2_{~1}=J_1\theta^2+J_2\theta^3.
\end{equation}
The EDS obtained from the reduced structure equations together with the differential relations among the functions $J_1,\dots,J_4,\Psi_2,\Psi'_2$ given by
\begin{equation}
  \label{eq:BianchiTypeD-1}
  \begin{aligned}
\der J_1 &= -J_1^2 \theta^1+J_1 J_2 \theta^4+\Gamma^1_{~1} J_1+2J_1J_3 \theta^2-J_{41} \theta^3\\
\der J_2 &= -\Gamma^2_{~2} J_2-J_1 J_2 \theta^1+(-J_{41}+\Psi_2-\Psi'_2) \theta^2+2J_2J_{4} \theta^3+J_2^2 \theta^4\\
\der J_3 &= \Gamma^2_{~2} J_3-2J_1J_3 \theta^1+J_3^2 \theta^2+J_3 J_4 \theta^3+(-J_{41}+\Psi_2-\Psi'_2) \theta^4\\
\der J_4 &= J_3 J_4 \theta^2+J_4^2 \theta^3-\Gamma^1_{~1} J_4+J_{41} \theta^1+2J_2J_{4} \theta^4\\
\der J_{41} &= (-2 J_1 J_{41}-2 J_1 \Psi'_2-J_1 \Psi_2+2J_1J_2 J_{3}) \theta^1+(2 J_3 J_{41}-2J_1J_{3} J_4) \theta^2\\
&\ \  +(-2J_2J_3 J_{4}+2 J_4 J_{41}+2 J_4 \Psi'_2+J_4 \Psi_2) \theta^3+(-2J_1J_2 J_{4}+2 J_2 J_{41}) \theta^4\\
\der \Psi_2 &= -3 J_1 \Psi_2 \theta^1+3 J_2 \Psi_2 \theta^4+3 J_3 \Psi_2 \theta^2+3 J_4 \Psi_2 \theta^3\\
\der\Psi_2'&=0\\
  \end{aligned}
\end{equation}
for some function $J_{41}$ is closed under the exterior derivative operator $\exd$. As a result, the local moduli space of    type $D$ pKE metrics is 5-dimensional. 
\end{proposition}
 \begin{proof}
   The differential identities \eqref{eq:BianchiTypeD-1} are obtained via straightforward computations discussed previously. The fact that the space of such pKE metrics is 5-dimensional follows from the Frobenius theorem applied to the resulting closed EDS. More precisely, define the 13-dimensional bundle $\cE^{13}\to \cF^6$ whose fibers are parametrized by $J=(J_1,J_2,J_3,J_4,J_{41},\Psi_2,\Psi'_2).$ Since the Pfaffian system  \eqref{eq:BianchiTypeD-1} is integrable, its leaf space is 7-dimensional parametrized by $J.$ As the infinitesimal group actions in \eqref{eq:BianchiTypeD-1} suggests, one obtains that the action of the structure group $\HH_{(2)}$ transforms the quantities $J_2$ and $J_4$ by
   \[\tilde J_2=a_{22}J_2,\qquad \tilde J_4=a_{11}J_4.\]
   To find the local generality one considers  generic pKE metrics of type $D.$ As a result, it can be  assumed that  $J_2,J_4\neq 0,$  after restricting to sufficiently small neighborhoods. Hence, setting $a_{11}=\frac{1}{J_4}$ and $a_{22}=\frac{1}{J_2},$ one can normalize  $J_2=J_4=1$ and obtain a canonical coframe at each point of such pKE metrics. Consequently, the remaining 5 parameters of $J$ can be used to distinguish every generic type $D$ pKE metric up to isomorphism. Therefore, the local generality of   type $D$ pKE metrics depends on 5 constants. 
\end{proof}
It turns out that the  pKE metrics of type \emph{D} locally belong to one of  four branches which can be characterized according to  the vanishing of the quantities $J_2$ and $J_4.$ Each branch can be locally expressed using normal coordinates given below.
  \begin{theorem}\label{thm:petrov-types-D-LocNormForm}
    In sufficiently small open sets, every pKE metric \[g=2\theta^1\theta^3+2\theta^2\theta^4,\] whose anti-self-dual Weyl curvature has real Petrov  type $D$,  belongs to one of the following four branches. The first branch is characterized by the property  that $J_2$ and $J_4$ are non-vanishing and is comprised of 5-parameter family of pKE metrics for which a choice of  normalized coframe can be expressed as
        \begin{equation}
      \label{eq:J24-nonvan}
      \begin{aligned}
        \theta^1 &=(y^4)^2\left(1-y^3\right)\exd y^1-y^4\exd y^2 ,\\
        \theta^2& =(y^4)^2\exd y^1-\theta^1,
        \\
         \theta^3 &=\textstyle{\left(-\Psi_2'(y^3)^3+\frac{k_1(y^3)^2}{y^4}-\frac{k_2y^3}{(y^4)^2}-\frac{k_3+\frac 12 k_4}{(y^4)^3}\right)\theta^1-\exd y^3-\frac{y^3}{y^4}\exd y^4}\\
         \theta^4&=\textstyle{\left(\Psi_2'(y^3-1)^3-\frac{k_1(y^3-1)^2}{y^4}+\frac{k_2(y^3-1)}{(y^4)^2}+\frac{k_3}{(y^4)^3}\right)\theta^2+\exd y^3+\frac{y^3-1}{y^4}\exd y^4}
         \end{aligned}
       \end{equation}
for some constants $k_1,\dots,k_4$ and $\Psi'_2.$ 
The second branch is characterized by the condition that $J_4=0$ and $J_2$ non-vanishing and is comprised of a 3-parameter family of pKE metrics for which a choice of  coframe can be expressed as 
\begin{equation}
  \label{eq:J2nonvan-J4van}
  \begin{gathered}
\theta^1= \textstyle{k_2(y^1)^2y^4\exd y^3+\frac{y^1}{y^4}\exd y^4+\exd y^1},\qquad \theta^2=\textstyle{y^1\exd y^3+\frac{1}{y^4}\exd y^2},\qquad \theta^3=\exd y^3\\
\theta^4=\textstyle{-\frac{k_1(y^4)^3+2k_2y^4+2\Psi'_2}{2y^4}\exd y^2-\frac{k_1(y^4)^3+2k_2y^4+2\Psi'_2}{2}y^1\exd y^3+\frac{1}{y^4}\exd y^4} 
\end{gathered}
\end{equation}
where $k_1,k_2$ and $\Psi_2'$ are constants.
Similarly, the third branch is characterized by $J_2=0$ and $J_4$ non-vanishing, which is comprised of 3-parameter family of pKE metrics and can be expressed as in \eqref{eq:J2nonvan-J4van} after switching $\theta^1\leftrightarrow \theta^2$ and  $\theta^3\leftrightarrow \theta^4.$ 

Lastly, the fourth branch, characterized by $J_2=J_4=0,$ is comprised of  the only homogeneous pKE metrics of type $D.$  They form a  1-parameter family parametrized by $\Psi'_2,$ for which a choice of coframe is given by   
\begin{equation}
  \label{eq:hom-typeD}
\begin{aligned}
    \theta^1&=\textstyle{\frac{\der y^1}{1-\tfrac32\Psi_2' y^1y^3}},\qquad && \theta^2=\textstyle{\frac{\der y^2}{1-\tfrac32\Psi_2' y^2 y^4}},\qquad
  \theta^3&=\textstyle{\frac{\der y^3}{1-\tfrac32\Psi_2' y^1y^3}},\qquad &&  \theta^4=\textstyle{\frac{\der y^4}{1-\tfrac32\Psi_2' y^2 y^4}},
\end{aligned}
\end{equation}
  \end{theorem}
  \begin{proof}
Let us first work in a  neighborhood $U$ in which $J_2,J_4$ are nowhere vanishing. Using the action of $a_{11}$ and $a_{22},$ whose infinitesimal version is given in \eqref{eq:BianchiTypeD-1},  there is a unique coframe with respect to which $J_2=J_4=1$ and therefore
\begin{equation}
  \label{eq:Gamma1122-reduced}
  \Gamma^1_{~1}=J_{41} \theta^1+J_3 \theta^2+\theta^3+2 \theta^4,\qquad \Gamma^2_{~2}=-J_1 \theta^1+(-J_{41}+\Psi_2-\Psi'_2) \theta^2+2 \theta^3+\theta^4.
\end{equation}
To find our normal coordinate system we make use of the orbits of the Killing vector fields of these pKE metrics. Let $v=v^i\frac{\partial}{\partial\theta^i}$ be a Killing vector field i.e. $\sL_v\theta^i=0$ where $\sL$ denotes the Lie derivative. It is straightforward to use the structure equations \eqref{pk} and reductions \eqref{eq:Gamma1221-reduced} and \eqref{eq:Gamma1122-reduced} when $J_2=J_4=1$ to obtain $\exd v^i$ for $i=1,\dots,4$. Subsequently,  the identities $\exd^2 v^i=0$ imply that the   isometry group for such pKE metrics is two dimensional since
\[v=\textstyle{v^1(\frac{\partial}{\partial\theta^1}+J_1\frac{\partial}{\partial\theta^3})+v^2(\frac{\partial}{\partial\theta^2}-J_3\frac{\partial}{\partial\theta^4}),}\]
where 
\begin{equation}
  \label{eq:vis}
  \exd v^1=v^1(J_1\theta^1-\theta^3)-(2v^1+v^2)(J_3\theta^2+\theta^4),\quad \exd v^2=v^1(\theta^1+J_1\theta^3)+v^2(\theta^2-J_3\theta^4).
\end{equation}
Because of the relation $\textstyle{\frac{\exd(v^1+v^2)}{2(v^1+v^2)}=J_1\theta^1-J_3\theta^2-\theta^3-\theta^4},$ we restrict ourselves to an open set where $\exd(v^1+v^2)$ and $v^1+v^2$ are non-zero. As a result, one can assume $v^1+v^2>0$ and   define 
\[\textstyle{y^3=\frac{v^2}{{v^1+v^2}}},\qquad y^4=\sqrt{v^1+v^2}.\] 
From \eqref{eq:vis} it follows that 
\begin{equation}
  \label{eq:theta34}
\theta^3=\textstyle{J_1\theta^1-\exd y^3-\frac{y^3}{y^4}\exd y^4},\qquad \theta^4=\textstyle{-J_3\theta^2+\exd y^3+\frac{y^3-1}{y^4}\exd y^4}
\end{equation}
Using the reduced 1-forms  \eqref{eq:Gamma1122-reduced} and normalized values $J_2=J_4=1$ in the Bianchi identities \eqref{eq:BianchiTypeD-1} together with the expressions \eqref{eq:theta34}, it follows that $J_1,J_3,J_{41}$ and $\Psi_2$ are functions of $y^3$ and $y^4.$ For instance, one obtains $\exd\Psi_2=-\frac{3\Psi_2}{y^4}\exd y^4$ which implies $\Psi_2=\frac{k_4}{(y^4)^3}$ for a constant $k_4.$ Similarly, elementary calculations can be carried out to show
\begin{equation}
  \label{eq:invariants-generic}
  \begin{aligned}
    J_1&=-\Psi_2'(y^3)^3+\textstyle{\frac{k_1}{y^4}(y^3)^2-\frac{k_2}{(y^4)^2}y^3-\frac{2k_3+k_4}{2(y^4)^3}},\\ J_3&=\textstyle{-\Psi'_2(y^3)^3+(3\Psi'_2+\frac{k_1}{y^4})(y^3)^2-(3\Psi'_2+\frac{2k_1}{y^4}+\frac{k_2}{(y^4)^2})y^3 +\Psi'_2+\frac{k_1}{y^4}+\frac{k_2}{(y^4)^2} -\frac{k_3}{(y^4)^3}},\\
  J_{41}&=\textstyle{2\Psi'_2(y^3)^3-(3\Psi'_2+\frac{2k_1}{y^4})(y^3)^2+(\frac{2k_1}{y^4}+\frac{2k_2}{(y^4)^2})y^3-\frac{k_2}{(y^4)^2} +\frac{2k_3+k_4}{(y^4)^3}},\\
  \end{aligned}
\end{equation}
for constants $k_1,k_2,k_3$ and $k_4.$ 
To express $\theta^1$ and $\theta^2$ in a local normal form, one makes use of the structure equations \eqref{pk} and the reductions \eqref{eq:Gamma1221-reduced} and \eqref{eq:Gamma1122-reduced} when $J_2=J_4=1$ to obtain
\[
\begin{aligned}
  \exd(\theta^1+\theta^2)&=\textstyle{\frac{2}{y^4}\exd y^4\w(\theta^1+\theta^2)}.
 \end{aligned}
\]
By  Darboux' theorem, locally one obtains  $\theta^1=(y^4)^2\exd y^1-\theta^2$ for a local coordinate $y^1.$  Lastly, the relation
\[\exd\theta^2=\textstyle{\frac{1}{y^4}\exd y^4\w\theta^2+y^4\left(y^3\exd y^4+y^4\exd y^4\right)\w\exd y^1}\]
implies that   $\theta^2=(y^4)^2y^3\exd y^1+y^4\exd y^2$ for a local coordinate $y^2.$ This proves the local normal form presented in  \eqref{eq:J24-nonvan}. Furthermore, it is straightforward to verify that  $\frac{\partial}{\partial y^1}$ and $\frac{\partial}{\partial y^2}$ are the Killing vector fields for these pKE metrics.

For the second branch  we restrict ourselves to open sets $U\subseteq M$ in which $J_4=0$ and $J_2$ is nowhere vanishing. Over $U$ one obtains $J_{41}=J_1=0.$ Using the relations \eqref{eq:BianchiTypeD-1}, consider the set of coframes with respect to  which $J_2=1$ and 
\begin{equation}
  \label{eq:Gamma22-reduced}
  \Gamma^2_{~2}=(\Psi_2-\Psi'_2)\theta^2+\theta^4.
\end{equation}
As described before, it is straightforward to express the  Killing vector fields of such pKE metrics which are of the form  
\begin{equation}
  \label{eq:KillingJ4van}
  v=\textstyle{v^1\frac{\partial}{\partial\theta^1}+v^2\frac{\partial}{\partial\theta^2} +v^3\frac{\partial}{\partial\theta^1}+J_3v^2\frac{\partial}{\partial\theta^4} +v^5\frac{\partial}{\partial\Gamma^1_{~1}}}
\end{equation}
 and, as a result, have to satisfy the differential relations
\begin{equation}
  \label{eq:vis-J4van}
  \begin{aligned}
    \exd v^1&=-v^1\Gamma^1_{~1}+(J_3v^2+v^5)\theta^1-v^1J_3\theta^2,\\
    \exd v^2&=-J_3v^2\theta^2-v^1\theta^3-v^2\theta^4+v^3\theta^1,\\
    \exd v^3&=v^3\Gamma^1_{~1}-(J_3v^2+v^5)\theta^3-v^3\theta^4,\\
    \exd v^5&=(J_3-2\Psi'_2-\Psi_2)(v^3\theta^1-v^1\theta^3)+v^2(\Psi'_2-\Psi_2)(J_3\theta^2+\theta^4).
  \end{aligned}
\end{equation}
Using the structure equations \eqref{pk} and the reduced 1-forms \eqref{eq:Gamma1221-reduced} and \eqref{eq:Gamma22-reduced} when $J_4=0$ and $J_2=1,$ together with the Bianchi identities \eqref{eq:BianchiTypeD-1}, it follows that 
\begin{equation}
  \label{eq:theta4-J4van}
  \exd(J_3\theta^2+\theta^4)=0\Rightarrow \theta^4=\textstyle{\frac{1}{y^4}\exd y^4}-J_3\theta^2
\end{equation}
for a local coordinate $y^4.$ It follows from  \eqref{eq:KillingJ4van} that the orbits of the isometry group for such pKE metrics are level sets of $y^4.$  This implies that such pKE metrics have  cohomogeneity one. Using the Bianchi identities \eqref{eq:BianchiTypeD-1} and the reduction \eqref{eq:Gamma22-reduced}, one obtains that $J_3$ and $\Psi_2$ are given by 
\[\Psi_2=k_1(y^4)^3,\qquad J_3=\textstyle{\frac{1}{2}k_1(y^4)^3+k_2y^4+\Psi'_2.}\] 
It remains to express $\theta^1,\theta^2,\theta^3$ in a local normal form which will be done using  \eqref{eq:vis-J4van}. Restricting to open sets where $v^3\neq 0,$ the relations \eqref{eq:vis-J4van} imply that 
\begin{equation}
  \label{eq:theta1Gamma11-J4van}
  \begin{aligned}
  \theta^1 &= \textstyle{\frac{v^1}{v^3}\theta^3+\frac{v^2}{v^3y^4}\exd y^4+\frac{1}{v^3}\exd v^2}\\
 \Gamma^1_1 &= \textstyle{-(\Psi'_2+\frac{(y^4)^3}{2k_1}+k_2y^4)\theta^2+\frac{v^2(\frac 12k_1(y^4)^3+k_2y^+\Psi'_2)+v^5}{v^3}\theta^3+\frac{\exd y^4}{y^4}+\frac{\exd v^3}{v^3}}
\end{aligned}
\end{equation}
Using \eqref{eq:theta4-J4van}, \eqref{eq:theta1Gamma11-J4van} in \eqref{eq:vis-J4van}, it is a matter of elementary calculation to show 
\[v^1=\textstyle{\frac{k_2y^4(v^2)^2}{v^3}},\qquad v^5=\textstyle{-\frac{k_1(y^4)^3-2k_2y^4+2\Psi'_2}{2}v^2}.\]
Lastly, the reduced structure equations imply that 
\[\exd\theta^2=-\textstyle{\theta^3\w(\frac{d v^2}{v^3}-\frac{v^2\exd y^4}{v^3y^4})+\theta^2\w\frac{\exd y^4}{y^4}},\qquad \exd\theta^3=-\theta^3\w\textstyle{\frac{\exd v^3}{v^3}}.\]
Using Darboux' theorem, one can find local coordinate system with respect to which
\[\theta^2=\textstyle{\frac{\exd y^2}{y^4}}+v^2\exd y^3,\qquad \theta^3=v^3\exd y^3.\]
As a result, we have a local coordinate system $(y^1,\dots,y^5),$ where $y^1=v^2$ and $y^5=v^3$ with respect to which we have expressed $\theta^i$'s and $\Gamma^1_{~1}$. It is clear that $v^3$ acts by scaling on $\theta^1$ and $\theta^3$ and therefore, corresponds to the element $a_{11}$ in the  reduced structure group \eqref{eq:typeII-red-str-gp-2}. The coframe \eqref{eq:J2nonvan-J4van} is obtained by setting $v^3=1.$ Finally, note that in terms of the local coordinates $(y^1,\dots,y^4)$, the trajectories of the Killing vector fields are given by the level sets of $y^2$ and $y^3$ and $(y^1)^2-(y^3)^2.$   The third branch characterized by  $J_2=0$ and $J_4$  nowhere vanishing can be treated similarly. 
 
    Finally, when $J_2=J_4=0,$ straightforward computation shows that all $J_i$'s vanish and $\Psi_2=\Psi'_2.$ As a result, such metrics are homogeneous. The structure equations are given by
\begin{subequations}\label{pKE-D-hom}
  \begin{align}
    \exd\theta^1=-\Gamma^1_{~1}\dz\theta^1,\qquad \exd\theta^3=\Gamma^1_{~1}\dz\theta^3,\qquad \exd\Gamma^1_{~1}=-3\Psi_2'\theta^1\dz\theta^3,  \label{pKE-D-hom1}\\
  \exd\theta^2=-\Gamma^2_{~2}\dz\theta^2,\qquad \exd\theta^4=\Gamma^2_{~2}\dz\theta^4,\qquad \exd\Gamma^2_{~2}=-3\Psi_2'\theta^2\dz\theta^4.  \label{pKE-D-hom2}
  \end{align}
\end{subequations}
Hence, it is sufficient to find a normal form for \eqref{pKE-D-hom1}. Using Darboux's theorem, \eqref{pKE-D-hom1}  imply that there are coordinates $y^1,y^3,y^5$ with respect to which
\[\theta^1=e^{-y^5+F_1}\exd y^1,\qquad \theta^3=e^{y^5-F_1+F_2}\exd y^3,\qquad \Gamma^1_{~1}=\exd y^5-\textstyle{\frac{\partial(F_1-F_2)}{\partial y^1}\exd y^1-\frac{\partial F_1}{\partial y^3}\exd y^3}\]
for  arbitrary functions $F_1=F_1(y^1,y^3)$ and  $F_2=F_2(y^1,y^3)$ which satisfy
\begin{equation} 
  \label{eq:LiouvillesEq-1}
  \textstyle{\frac{\partial^2}{\partial y^1\partial y^3}F_2-3\Psi'_2e^{F_2}=0}.
\end{equation}
The equation above is known as  Liouville's equation \cite{Henrici} whose solutions can be expressed as
\begin{equation}
  \label{eq:LiouvillesSol}
F_2(y^1,y^3)=\ln\left(\frac{2p'q'}{-3\Psi'_2(p-q)^2}\right)  
\end{equation}
for two arbitrary functions $p=p(y^1)$ and $q=q(y^3).$  Setting  
\[\textstyle{p(y^1)=6\Psi_2'y^1,\qquad q(y^3)=\frac{4}{y^3},\qquad F_1(y^1,y^3)=\textstyle{\ln\left(\frac{1}{1-\frac 32\Psi'_2y^1y^3}\right)},\qquad y^5=0 }\]
   one obtains the expression \eqref{eq:hom-typeD} for $\theta^1$ and $\theta^3.$ The expressions of $\theta^2$ and $\theta^4$ are obtained similarly.

  \end{proof} 
  It is straightforward to use the coframe \eqref{eq:hom-typeD} in order to arrive at the potential function $V_2$ in \ref{sec:simple-examples}.
  Moreover, one can characterize the branchings in Theorem \ref{thm:petrov-types-D-LocNormForm} in terms of  the vanishing of $J_1$ and $J_3.$ 
\begin{remark} 
 Theorem \ref{thm:petrov-types-D-LocNormForm} is yet another instance of explicit local normal form for certain classes of (pseudo-)Riemannian metrics whose Weyl curvature has algebraic type $D,$ which includes  the  Pleba\'nski-Demia\'nski  metrics \cite{Kinnersley, Debever,PD} in the Lorentzian signature (see \cite{GP-survey,Kamran} for a survey of  all the results), and ambitoric metrics in Riemannian signature \cite{ACG}.      
\end{remark}  

\subsubsection{Petrov type $II$}
\label{sec:petrov-type-II}
Now we proceed  to pKE metrics whose anti-self-dual Weyl curvature has Petrov type $II$. 
\begin{theorem}\label{thm:typeII-exa-1}
  Given a pKE metric of Petrov type $II$, the  $\SOtt$-valued Cartan connection $\cB$ on the principal bundle $\bbS^1\times\bbS^1\to \cF^6\to M$, as defined in Proposition \ref{rei}, 
satisfies the Yang-Mills equations $D*K_\cB=0$, where $K_\cB=\der \cB+\cB\dz \cB$, if and only if  
\begin{equation}
  \label{eq:YM-typeII-1}
  \Psi_2=\Psi_2',\quad\mathrm{and}\quad J_1=J_2=J_3=J_4=0.
\end{equation}
in \eqref{eq:J16-typeIID}. Examples of such pKE structure are given by  $g=2\theta^1\theta^3+2\theta^2\theta^4$ where
\begin{equation}
\label{exII}\begin{aligned}
  \theta^1=&\der x-\tfrac32\Psi_2'\big(x^2+xf_1(a)+f_2(b)\big)\der a\\
    \theta^2=&\der b\\
    \theta^3=&\der a\\
    \theta^4=&\der y-\tfrac32\Psi_2'\big(y^2+yf_3(b)+f_4(a)\big)\der b
\end{aligned}
\end{equation}
for some arbitrary functions $f_1,\dots,f_4$ where $f'_2,f'_4$ are  nowhere vanishing.
\end{theorem}
\begin{proof}
Using the expression of $K_\cB$ in \eqref{eq:typeII-curv-YM-1}, one immediately obtains  \eqref{eq:YM-typeII-1} (see Remark \ref{rmk:YM-AllTypes}).   
  One can integrate the structure equations for   pKE metrics of Petrov type $II$ that are Yang-Mills assuming some simplifying conditions. The integration procedure can be carried out similarly to what was explained in Theorem \ref{thm:typeD-AllCases-NormalForm} and will not be explained here. It turns out  that pKE metrics arising from the coframe \eqref{exII} satisfy 
\[\begin{aligned}
    \Psi_2&=\Psi_2'=\mathrm{const},\qquad     \Psi_0=\Psi_1=\Psi_3=0,\\ \Psi_4&=-\tfrac34\Big(3\left(f_3f'_{2}+f_1f_{4}'+2xf_{4}'+2yf_{2}'\right)\Psi'_2-2f_{4}''-2f_{2}''\Big)\Psi_2'\\
J_5&=\textstyle{-\frac{3}{2}\Psi_2'f'_2},\qquad J_6=\textstyle{\frac{3}{2}\Psi_2'f'_4}  
\end{aligned}\]
Note that if   $f'_2=f'_4=0$ one obtains homogeneous pKE metrics of type $D$. 

 The $\SOtt$-valued Cartan connection $\cB$, as defined in (\ref{cacoB}), has curvature
\[K_\cB=\bma[c|c]
\begin{matrix}0&\tfrac32\sqrt{\tfrac32}|\Psi_2'|{}^{3/2}~f'_{2}\\
0&0\end{matrix}&0\\
\cmidrule(lr){1-2}
0&\begin{matrix}0&0\\
  -\tfrac32\sqrt{\tfrac32}|\Psi_2'|{}^{3/2}~f'_{4}&0\end{matrix}\ema\sigma^2_-,\]
  which is anti-self-dual, and therefore satisfies the Yang-Mills equations $D*K_\cB=0$. 
\end{proof} 
\begin{remark}\label{rmk:YM-AllTypes}
Note that if  the conditions \eqref{eq:YM-typeII-1} for $J_1,\dots,J_6$ in ~\eqref{eq:J16-typeIID}, which arise from the Yang-Mills equation, are replaced by $J_5=J_6=0,$ then the $Weyl^-$ of the pKE metric has type $D$  as discussed in Theorem \ref{thm:typeD-AllCases-NormalForm}. Furthermore, as shown in Proposition \ref{rei},  one can always associate a Cartan geometry of type $(\SOtt,\mathbf{T}^2)$, with a canonical Cartan connection,  to pKE metrics of any  Petrov type by appropriately reducing the structure group. However, except for type $II$, it can be shown that the set of Yang-Mills solutions among other types is empty.    
\end{remark}
\subsubsection{Petrov type $III$}\label{sec:reduction-type-iii} 
 Assume that the quartic $W(\lambda)$ in \eqref{eq:quartic2} has a repeated root of multiplicity three. As we did in \ref{sec:petrov-type-d}, by coframe adaptation one can translate the multiple root to zero, which is equivalent to finding an adapted coframe with respect to which 
\[\Psi_0=\Psi_1=\Psi_2=0,\quad\mathrm{and}\quad \Psi_3\neq 0.\]
In this case  Proposition \ref{typ2}  still remains valid and the following differential relations hold
\begin{equation}
  \label{eq:TypeIII-InfintAct}
  \begin{aligned}
  \der\Psi_3&=-\Gamma^1_{~1}\Psi_3+\Gamma^2_{~2}\Psi_3-2J_1\Psi_3\theta^1+(J_1\Psi_4+\Psi_{41})\theta^2+(J_2\Psi_4-\Psi_{44})\theta^3 +2J_2\Psi_3\theta^4 \\
  \der\Psi_4 &= -2\Gamma^1_{~1}\Psi_4+4\Gamma^1_{~2}\Psi_3+2\Gamma^2_{~2}\Psi_4+\Psi_{41}\theta^1 +\Psi_{42}\theta^2+\Psi_{43}\theta^3+\Psi_{44}\theta^4\\
  \der J_2 &=-J_1J_2\theta^1+J_2^2\theta^4-\Gamma^2_{~2}J_2+J_{22}\theta^2+J_{23}\theta^3
\end{aligned}
\end{equation}
for some functions $J_{22},J_{23}.$ 

Using the action of $a_{11}$ and $a_{22}$, one can find the set of adapted coframes with respect to which the quantities  $\Psi_3$ and $J_2$ are normalized  to constants. The set of such adapted coframes give rise to a  line bundle $\cF^5\to M$ with the group parameter  $a_{12}$ in \eqref{eq:typeII-red-str-gp}  as the fiber coordinate.

\begin{proposition}\label{typeiii}
Given a pKE metric whose anti-self-dual Weyl curvature has Petrov type III and for which $J_2\neq 0$, the bundle of adapted coframes preserving 
\[\Psi_0=\Psi_1=\Psi_2=0,\qquad \Psi_3=\mathrm{const}\neq 0,\qquad J_2=\mathrm{const}\neq 0\]
is a line bundle $\cF^5\to M,$ whose sections satisfy the structure equations \eqref{pk}  wherein
\[\Gamma^2_{~1}=J_1\theta^2+J_2\theta^3,\qquad \Gamma^1_{~1}=-3J_1\theta^1+J_3\theta^2+J_4\theta^3+3J_2\theta^4,\qquad \Gamma^2_{~2}=-J_1\theta^1+J_5\theta^3+J_6\theta^2+J_2\theta^4.\]
Examples of such pKE metrics  for which $J_1=0$, $J_6=-\Psi'_2,$ $J_2=\Psi_3=1$ are given by $g=2\theta^1\theta^3+2\theta^2\theta^4$ where
\[\begin{aligned}
  \theta^1&=-\scriptstyle{\frac{1}{4}\left(3\mathrm{e}^{\frac{1}{2}y^4}\Psi'_2y^1 +y^2\mathrm{e}^{\frac{1}{2}y^4}\Psi'_2 +6\Psi'^2_2y^1+2y^2\mathrm{e}^{y^4}-2\mathrm{e}^{\frac{1}{2}y^4}W_{y^3}  -\Psi'_2W_{y^3}+\mathrm{e}^{y^4}-U_{y^1}+U_{y^1,y^1}\right)\mathrm{e}^{-2y^4}\exd y^1}\\
  &\phantom{=}+\textstyle{ \, \mathrm{e}^{y^4} \exd y^2  +\frac{1}{2}\mathrm{e}^{-\frac{3}{2}y^4}\left(y^2\mathrm{e}^{\frac{1}{2}y^4}-2W_{y^3,y^4}\right)\exd y^4}\\
  & \phantom{=}\scriptstyle{ -  \frac{1}{4}\left(6\Psi'^2_2y^1y^2\mathrm{e}^{\frac{1}{2}y^4}-9y^2\mathrm{e}^{y^4}\Psi'_2y^1 +\mathrm{e}^{y^4}\Psi'_2(y^2)^2+9\Psi'_2y^1W_{y^3}\mathrm{e}^{\frac{1}{2}y^4} -2W_{y^3}\Psi'_2y^2\mathrm{e}^{\frac{1}{2}y^4} -6\Psi'^2_2y^1W_{y^3}\right.}
\\
& \phantom{=}\scriptstyle{ \left. +4y^2\mathrm{e}^{y^4}W_{y^3} -2W_{y^3}^2 \mathrm{e}^{\frac{1}{2}y^4}-U_{y^1}y^2\mathrm{e}^{\frac{1}{2}y^4}+U_{y^1,y^1}y^2\mathrm{e}^{\frac{1}{2}y^4} +W_{y^3}^2\Psi'_2-2(y^2)^2\mathrm{e}^{\frac{3}{2}y^4}-\mathrm{e}^{y^4}W_{y^3}\right.}
\\
& \phantom{=}\scriptstyle{ \left.-4U \mathrm{e}^{\frac{1}{2}y^4}+4W_{y^3,y^3}\mathrm{e}^{\frac{1}{2}y^4}+U_{y^1}W_{y^3}-W_{y^3}U_{y^1,y^1} +y^2\mathrm{e}^{\frac{3}{2}y^4}\right)\mathrm{e}^{-2y^4}\exd y^3}
\\
  \theta^2&= \textstyle{\mathrm{e}^{-\frac{1}{2}y^4}\exd y^1 +\left(y^2-W_{y^3}\right)\mathrm{e}^{-\frac{1}{2}y^4}\exd y^3}\\
  \theta^3&=\mathrm{e}^{y^4}\exd y^3\\
  \theta^4&=\textstyle{\left(-\Psi'_2\mathrm{e}^{-\frac{1}{2}y^4}-1\right)\exd y^1  +\frac{1}{2}\exd y^4  }\\
  &\textstyle{
   \phantom{=} +\left(\frac{3}{2}\mathrm{e}^{-\frac{1}{2}y^4}\Psi'^2_2y^1+\frac{3}{4}\mathrm{e}^{-\frac{1}{2}y^4}W_{y^3} \Psi'_2  -\frac{1}{4}\mathrm{e}^{-\frac{1}{2}y^4}U_{y^1}+\frac{1}{4}\mathrm{e}^{-\frac{1}{2}y^4}U_{y^1} -\frac{1}{2}y^2\mathrm{e}^{\frac{1}{2}y^4}\right.}\\
  &\phantom{=}\textstyle{\left. +\frac{3}{4}\Psi'_2y^1-\frac{3}{4}\Psi'_2y^2- \frac{3}{4}\mathrm{e}^{\frac{1}{2}y^4} +\frac{1}{2}W_{y^3}\right) \exd y^3} 
\end{aligned}
\]
where $U=U(y^1,y^3),W=W(y^3,y^4)$ are arbitrary functions. 
\end{proposition}
\begin{proof}
  We skip the proof due to its similarity to that of Theorems \ref{thm:typeD-AllCases-NormalForm} and \ref{thm:typeII-exa-1}. 
\end{proof}
\begin{remark}
We point out that by the action of the structure  group, infinitesimally  given in \eqref{eq:TypeIII-InfintAct}, one can reduce the  structure group to identity by  translating $\Psi_4$ to zero and obtain a unique choice of coframe at each point. However, to obtain examples above one does not need to carry out full reduction. Some solutions satisfying these conditions where obtained earlier by A. Chudecki in \cite{Chudecki}.    
\end{remark} 
 
\subsubsection{Petrov type $N$}\label{sec:reduction-type-n}
Assume that the quartic $W(\lambda)$ in \eqref{eq:quartic2} has a repeated root of multiplicity four. As we did in \ref{sec:petrov-type-d}, by coframe adaptation one can translate the multiple root to zero, which is equivalent to finding an adapted coframe with respect to which 
\[\Psi_0=\Psi_1=\Psi_2=\Psi_3=0,\quad \mathrm{and}\quad \Psi_4\neq 0.\]
In this case  Proposition \ref{typ2}  still remains valid and the following differential relations hold
\begin{equation}
  \label{eq:TypeN-InfintAct}
  \begin{aligned}
  \der\Psi_4&=-2\Gamma^1_{~1}\Psi_4+2\Gamma^2_{~2}\Psi_4-J_1\Psi_4\theta^1+\Psi_{42}\theta^2+\Psi_{43}\theta^3 +J_2\Psi_4\theta^4 \\
  \der J_2 &=-J_1J_2\theta^1+J_{22}\theta^2+J_{23}\theta^3+J_2^2\theta^4-\Gamma^2_{~2}J_2
\end{aligned}
\end{equation}
As a result, by normalizing $\Psi_4$ and $J_2$ to non-zero constants we can reduce the parameters $a_{11}$ and $a_{22}$ in the structure group \eqref{eq:typeII-red-str-gp}  which reduces the bundle of adapted coframes to a line bundle $\cF^5\to M$ with the element $a_{12}$ in \eqref{eq:typeII-red-str-gp} as the fiber coordinate. It follows that 
\begin{equation}
  \label{eq:TypeN-2ndRed}
  \begin{aligned}
\Gamma^1_{~1}=-\textstyle{\frac 32 J_1\theta^1+J_3\theta^2+J_4\theta^3+\frac 32 J_2\theta^4},\quad \Gamma^2_{~2}=-J_1\theta^1+J_5\theta^2+J_6\theta^3+J_2\theta^4.
\end{aligned}
\end{equation}
for some functions $J_3,\dots, J_6.$ 
It is straightforward to obtain 
\[\exd J_4\equiv -\textstyle{\frac 52 J_2\Gamma^1_2}\qquad \mathrm{mod}\qquad \{\theta^1,\theta^2,\theta^3,\theta^4\}.\]
Because the quantity $J_2$ is normalized to a non-zero constant, the above differential relation can be interpreted as the infinitesimal action of the 1-dimensional structure group on $J_4$ (see Remark \ref{rmk:infinit-action-psi0}). Hence, by choosing $a_{12}$ appropriately, one can translate $J_4$  to zero which would reduce the structure group to identity. In other words there is a unique coframe at each point with respect to which one has the relations equations  \eqref{eq:Gamma21-reduced},  \eqref{eq:TypeN-2ndRed} and
\begin{equation}
  \label{eq:Gamma12-TypeN-red}
\Gamma^1_{~2}=\textstyle{-\frac 15\left(2J_3+3J_5+7\frac{\Psi'_2}{J_2}\right)\theta^1+J_7\theta^2+J_8\theta^3+\frac 35 J_6\theta^4}.
\end{equation}
for some functions $J_7$ and $J_8$ on $M.$ As a result we obtain the following.
\begin{theorem}\label{typeiv}
Given a pKE metric whose anti-self-dual Weyl curvature has Petrov type $N$ and for which $J_2\neq 0$,  there is a unique adapted coframe that  preserves 
\[\Psi_0=\Psi_1=\Psi_2=\Psi_3=0,\qquad \Psi_4=\mathrm{const}\neq 0,\qquad J_2=\mathrm{const}\neq 0,\qquad J_4=0\]
with respect to which the relations \eqref{eq:Gamma21-reduced}, \eqref{eq:TypeN-2ndRed} and \eqref{eq:Gamma12-TypeN-red} hold. 
A class of examples for which $J_1=J_6=J_7=0$, $\Psi'_2=-8J_3=4J_5,$  $J_2=-4,$ and $\Psi_3=1$ is given by $g=2\theta^1\theta^3+2\theta^2\theta^4$ where
  \[
\begin{aligned}
  \theta^1&= 2\mathrm{e}^{-3y^4}\exd y^1+\left(-16(y^3)^2+F_1(y^2)+y^3F_2(y^2)\right)\exd y^2\\
  \theta^2&=8 \mathrm{e}^{-2y^4}\left(\exd y^3-y^1\exd y^2\right)\\
  \theta^3&= \mathrm{e}^{y^4}\exd y^2\\
  \theta^4 &= \textstyle{-\frac{1}{2}\exd y^4-\frac{1}{2}\Psi'_2\mathrm{e}^{-2y^4}\left(\exd y^3-y^1\exd y^2\right) }
\end{aligned}
\] 
where $F_1(y^2)$ and $F_2(y^2)$ are arbitrary functions. 
\end{theorem}
\begin{proof}
  We skip the proof due to its similarity to that of Theorems \ref{thm:typeD-AllCases-NormalForm} and \ref{thm:typeII-exa-1}. 
\end{proof}

\subsubsection{Petrov type $O$}
\label{sec:type-o}
The Petrov type O corresponds to pKE metrics for which $\Psi_0=\dots=\Psi_4=0.$ Since the only non-zero quantity in the structure equations \eqref{pk} is the constant $\Psi'_2,$ it follows that such metrics are homogeneous and therefore no reduction of the structure bundle is possible. Nevertheless, one can follow the procedure explained before and integrate the structure equations from which the following choice of coframe is obtained
\begin{equation}
  \label{eq:hom-TypeO}
\theta^1=\textstyle{\frac{\exd y^1}{\Psi'_2(y^2+y^1y^3-y^4)}},\ \  \theta^2=\textstyle{\frac{\exd y^2}{\Psi'_2(y^2+y^1y^3-y^4)}},\ \  \theta^3=\textstyle{\frac{(y^4-y^2)\exd y^3-y^3\exd y^4}{\Psi'_2(y^2+y^1y^3-y^4)}},\ \  \theta^4=\textstyle{\frac{y^1\exd y^3-\exd y^4}{\Psi'_2(y^2+y^1y^3-y^4)}}
\end{equation}
Using the coframe above one can recover the potential function  $V_1$ given in \ref{sec:simple-examples} for the so-called dancing metric.

\subsubsection{Homogeneous models and local generality of various Petrov types}
\label{sec:local-generality-pke}
The  structure equations of pKE metrics in dimension four and reduced structure equations obtained for  non-generic Petrov types enable one to use the Cartan-K\"ahler theory and obtain the  local generality of analytic pKE metrics of each Petrov type. We will not give the details of how  Cartan-K\"ahler theory is implemented and refer the reader to \cite{Bryant-Notes} for details. 
 
Assuming analyticity for pKE metrics, the following table gives the local generality of various Petrov types. 
\begin{table}[h]
 \begin{center}
\begin{tabular}{ |c|c| } 
 \hline
  Petrov type  & Local generality \\ \hline\hline
$G$ & 2 functions  of 3 variables \\ \hline
  $II$ & 4 functions  of 2 variables  \\\hline
  $III$ & 3 functions of 2 variables  \\\hline
  $N$ & 2 functions of 2 variables  \\ \hline
 $II$ and Yang-Mills & 2 functions of 2 variables  \\\hline
  $D$ & 5 constants  \\\hline
 $O$ & 1 constant  \\\hline
\end{tabular}
\end{center}
  \caption{Local generality of pKE metrics}
  \label{tab:CK}

\vspace{-.65cm}

\end{table}

Furthermore, the reduced structure equations for each Petrov types allows one to look for homogeneous models. Finding homogeneous models involves a straightforward inspection of structure equations considering all possible normalizations which can be carried out algorithmically. We will not present all the necessary computation here.  It turns out that the only homogeneous models of pKE metrics satisfying $\Psi_2'\neq 0$ are the 1-parameter families of pKE metrics of type $D$ and $O$ which correspond  to the coframes \eqref{eq:hom-typeD} and \eqref{eq:hom-TypeO}. In particular, there is no homogeneous pKE metric of type $G,$ $II,$ $III$ and $N$ for which $\Psi'_2\neq 0.$
  
\section{(2,3,5)-distributions arising from pKE metrics}\label{sec:2-3-5-dist}
This section contains the highlight of the article.  In \ref{sec:primer-2-3-5} we  give a brief review of the  geometry of  (2,3,5)-distributions. In \ref{sec:null-self-dual}  we show that  the naturally induced rank 2 twistor distribution on the space of self-dual null planes of any pKE metric is (2,3,5) in an open subset if $\Psi'_2\neq 0.$ Furthermore, the root type of the Cartan quartic of this twistor distribution   agrees with the root type of  the quartic representation of $Weyl^-.$  This remarkable and surprising  coincidence is contrasted with the case of twistor distribution naturally arising on the space of anti-self-dual null planes of pKE metrics satisfying $Weyl^-\neq 0,$ which is considered in \ref{sec:2-3-5-ASD}. In the latter case, the coefficients of the Cartan quartic depend on the fourth jet of the coefficients of $Weyl^-$ and there is no further simplification from the larger context  of  twistor distributions arising from indefinite conformal structures in dimension four satisfying $Weyl^-\neq 0.$ In other words, a priori, no relation between the type of the Cartan quartic and the Petrov type of $Weyl^-$ or $Weyl^+$ can be made (see Remark \ref{rmk:235-asd-weyl-bundle}). Moreover, our construction in \ref{sec:null-self-dual} gives rise to 5-dimensional para-Sasaki-Einstein structures  and conformal structures with $\mathbf{SL}_3(\RR)$ holonomy, as studied in \cite{SW-G2}. Consequently, our explicit examples of pKE metrics of special real Petrov  type in  \ref{sec:cart-reduct-symm}, provide examples of 5-dimensional  para-Sasaki-Einstein metrics.
\subsection{A primer on $(2,3,5)$ distributions}\label{sec:primer-2-3-5}

In this section we recall  the basic definitions and theorems about the local geometry of a generic 2-plane field on a 5-dimensional manifold, $Q,$ which involves a Cartan connection and a naturally induced  conformal structure of signature (3,2). 

Given a 5-dimensional manifold, $Q,$ with a rank 2 distribution $\cD\subset \rT Q$ let $\partial \cD$ denote its first derived system  defined as the distribution whose sections are given by $\Gamma(\cD)+[\Gamma(\cD),\Gamma(\cD)],$ where $\Gamma(\cD)$ denotes the sheaf of sections of the distribution $\cD.$ Moreover, define $\partial^2\cD=\partial(\partial \cD).$   
\begin{definition}
  A rank 2 distribution in dimension 5, $\cD\subset \rT Q,$ is called a (2,3,5)-distribution if 
 \[\mathrm{rank}(\partial \cD)=3,\qquad\mathrm{and}\qquad \mathrm{rank}(\partial^2\cD)=5.\]
\end{definition}
 Locally, a generic rank 2 distribution  is a (2,3,5)-distribution.  
Given a (2,3,5)-distribution, locally, one  can find a frame  $\{\mathbf{v}_1,\cdots,\mathbf{v}_5\}$ for $M$ such that
\[\cD=\mathrm{span}\{\bfv_4,\bfv_5\},\qquad \partial \cD=\mathrm{span}\{\bfv_3,\bfv_4,\bfv_5\},\qquad \partial^2\cD=\mathrm{span}\{\bfv_1,\cdots,\bfv_5\}\]
where 
\[\bfv_3=-[\bfv_4,\bfv_5],\qquad \bfv_2=-[\bfv_3,\bfv_4],\qquad \bfv_1=-[\bfv_3,\bfv_5].\] 
As a result, the corresponding coframe $\{\eta^1,\cdots,\eta^5\}$ satisfies
\begin{equation}
  \label{eq:coframe235}
  \begin{aligned}
    \exd\eta^1&\equiv \eta^3\w\eta^4\qquad \mathrm{mod}\qquad \{\eta^1,\eta^2\},\\
    \exd\eta^2&\equiv \eta^3\w\eta^5\qquad \mathrm{mod}\qquad \{\eta^1,\eta^2\},\\
    \exd\eta^3&\equiv \eta^4\w\eta^5\qquad \mathrm{mod}\qquad \{\eta^1,\eta^2,\eta^3\}.
  \end{aligned}
\end{equation}
 Cartan in his famous 'five-variables' paper \cite{CartanG2} solved the equivalence problem for (2,3,5)-distributions   and explicitly introduced the distribution $\cD_{o}$ whose algebra of infinitesimal symmetries is given by the split real form of the exceptional Lie algebra $\mathfrak{g}^*_2.$  Recall that the noncompact exceptional simple Lie group  of dimension 14, $\mathbf{G}^*_2\subset \mathbf{SO}_{4,3} $ acts transitively on the projective quadric $\mathsf{Q}_{3,2}\subset \PP^{6}$ defined by the (3,2)-signature diagonal matrix.  Let $\mathbf{P}_1$ be the parabolic subgroup of $\mathbf{G}^*_2$ that preserves a null line in $\mathsf{Q}_{3,2}$.   Using his method of equivalence, Cartan associated an $\{e\}$-structure on a 14-dimensional $\mathbf{P}_1$-principal bundle $\pi: \cG\to Q$ to any (2,3,5)-distribution.

 Using the appropriate transformation Cartan's original construction  results in the following theorem.
\begin{theorem}[\cite{CartanG2,NurowskiG2}]\label{thm:CartanNurowski-235}
  Any (2,3,5)-distribution, $\cD\subset \rT Q,$ defines a Cartan geometry $(\cG,Q,\omega_{\bG^*_2})$ of type $(\rG_2^*,\rP_1).$ Expressing the  $\fg_2^*$-valued Cartan connection  as 
\begin{equation}
   \label{eq:G2-CarNur-Conn}
    \def\arraystretch{1.5}
   \omega_{\mathbf{G}^*_2}=\textstyle{
\begin{pmatrix}
     -\zeta_1-\zeta_4 & -\zeta_8 & -\zeta_9  & -\frac{1}{\sqrt{3}}\zeta_7 & \frac{1}{3}\zeta_5 & \frac{1}{3}\zeta_6 & 0\\
      \eta^1 & \zeta_1 & \zeta_2 & \frac{1}{\sqrt{3}}\eta^4 & -\frac{1}{3}\eta^3 & 0 & \frac{1}{3}\zeta_6\\
      \eta^2 & \zeta_3 & \zeta_4 & \frac{1}{\sqrt{3}}\eta^5 & 0 & -\frac{1}{3}\eta^3 & -\frac{1}{3}\zeta_5\\
      \frac{2}{\sqrt{3}}\eta^3 & \frac{2}{\sqrt{3}}\zeta_5 & \frac{2}{\sqrt{3}}\zeta_6 & 0 & \frac{1}{\sqrt{3}}\eta^5 & -\frac{1}{\sqrt{3}}\eta^4 & -\frac{1}{\sqrt{3}}\zeta_7\\
      \eta^4 & \zeta_7 & 0 & \frac{2}{\sqrt{3}}\zeta_6 & -\zeta_4 & \zeta_2 & \zeta_9\\
      \eta^5 & 0 & \zeta_7 & -\frac{2}{\sqrt{3}}\zeta_5 & \zeta_3 & -\zeta_1 & -\zeta_8 \\
      0 & \eta^5 & -\eta^4 & \frac{2}{\sqrt{3}}\eta^3 & -\eta^2 & \eta^1 & \zeta_1+\zeta_4
   \end{pmatrix}}
 \end{equation}
the distribution $\cD$ is the projection  of $\mathrm{Ker}\{\eta^1,\eta^2,\eta^3\}$ from $\cG$ to $Q.$  
\end{theorem}
\begin{remark}\label{rmk:Curv-235-quartic}
Cartan realized that the curvature  $K_{\omega_{\bG^*_2}}=\exd\omega_{\rG_2^*}+\omega_{\rG_2^*}\w\omega_{\rG_2^*}$ can be interpreted as a ternary quartic form  $\mathbf{W}\in \mathrm{Sym}^4(\partial\cD)^*$ and expressed as
\begin{equation}
  \label{eq:Cartan-Full-quartic}
  \begin{aligned}
\mathbf{W}&=\sum_{i=0}^4\m{i\\4}a_i(\eta^4)^{4-i}(\eta^5)^i+\sum_{i=0}^3\m{i\\3}b_i(\eta^4)^{3-i}(\eta^5)^i\eta^3\\ &+\sum_{i=0}^2\m{i\\2}c_i(\eta^4)^{2-i}(\eta^5)^i(\eta^3)^2+\sum_{i=0}^1d_i(\eta^4)^{1-i}(\eta^5)^i(\eta^3)^3+e(\eta^3)^4
  \end{aligned}
\end{equation} 
where the coefficients $a_0,\cdots,e$ are components of the curvature $K_{\omega_{\bG^*_2}}$ (see \cite{NurowskiG2}). Moreover, the fundamental curvature tensor is a binary quartic form  $\mathbf{C}\in\mathrm{Sym}^4(\cD^*),$ referred to as the \emph{Cartan quartic}, given by the first 5 terms in \eqref{eq:Cartan-Full-quartic}. If  $\mathbf{C}$ is identically zero  it follows that  $K_{\omega_{\bG^*_2}}=0$ i.e.  the (2,3,5)-distribution is flat.
 It is convenient to  express the Cartan quartic in 1-variable $z$ as follows  
\begin{equation}\label{eq:CartanQuartic-C}
  \begin{aligned}
C(z)&:=\textstyle{\mathbf{C}(\frac{\partial}{\partial\eta^4}+z\frac{\partial}{\partial\eta^5},\frac{\partial}{\partial\eta^4}+z\frac{\partial}{\partial\eta^5},\frac{\partial}{\partial\eta^4}+z\frac{\partial}{\partial\eta^5},\frac{\partial}{\partial\eta^4}+z\frac{\partial}{\partial\eta^5})}\\ &=a_0+4a_1z+6a_2z^2+4a_3z^3+a_4z^4.
\end{aligned}
\end{equation}
\end{remark}
Using Cartan's result and the embedding  $\rG^*_2\hookrightarrow \mathbf{SO}_{4,3},$ the following non-trivial link between (2,3,5)-distributions and conformal structures of signature $(3,2)$ can be obtained.
\begin{theorem}[\cite{NurowskiG2}]\label{thm:Pawel-G2-confstr}
  Any (2,3,5)-distribution $\cD\subset \rT Q$ defines a conformal structure $[\tilde h]$ of signature (3,2) on $Q,$ which can be expressed as $\tilde h=s^*h$ for any section  $s\colon Q\to \cG,$ where 
  \begin{equation}
 \label{eq:Nurowski-Metric}
     h=\eta^1\eta^5-\eta^2\eta^4+{\textstyle \frac{2}{3}}\,\eta^3\eta^3\in \mathrm{Sym}^2(\rT^*Q).
  \end{equation}
The conformal holonomy of this conformal structure takes value  in $\rG_2^*$ and its Weyl curvature can be expressed in terms of  $K_{\omega_{\bG^*_2}}.$
\end{theorem}
\begin{remark}\label{rmk:235QuarticRedone}
  Using Theorem \ref{thm:Pawel-G2-confstr}, we give another interpretation of the Cartan quartic \eqref{eq:CartanQuartic-C} which will be important for analyzing non-integrable twistor distributions.
  At each point $q\in Q,$ consider the $\bbS^1$-family  of  planes
  \begin{equation}
    \label{eq:Znull}
    \cZ_q:=\Ker\{\eta^2-z\eta^1,\eta^5-z\eta^4,\eta^3\}\subset\rT_q Q,
  \end{equation}
  where $z\in\RR\cup\{\infty\}.$ Such planes are null with respect to the conformal structure $[h],$ defined in  \eqref{eq:Nurowski-Metric}, and intersect the distribution $\cD_q$ along the lines $\langle\frac{\partial}{\partial\eta^4}+z\frac{\partial}{\partial\eta^5}\rangle.$ The bundle
  \begin{equation}
    \label{eq:Zbundle}
    \gamma\colon\cZ\to Q,
  \end{equation}
where $\gamma^{-1}(q)=\cZ_q,$ is the circle bundle of such null planes.
  Denote the components of the Weyl curvature for the conformal structure $[h]$ by  $W^i_{~jkl}$   where $1\leq i,j,k,l\leq 5.$ Following our discussion in \ref{sec:cart-penr-petr}, let $W_{ijkl}=h_{im}W^m_{~jkl}$ and define the multilinear map
  \[\widetilde\bW:=W_{ijkl}(\eta^i\w\eta^j)\circ(\eta^k\w\eta^l)\in\mathrm{Sym}^2(\Lambda^2 \rT Q)\to C^\infty(Q).\]
  Restricting to $\cZ,$ one obtains the quartic polynomial
  \[
    \begin{aligned}
      C(z)&=\widetilde\bW(\textstyle{\frac{\partial}{\partial\eta^4}+z\frac{\partial}{\partial\eta^5}, \frac{\partial}{\partial\eta^1}+z\frac{\partial}{\partial\eta^2},
      \frac{\partial}{\partial\eta^4}+z\frac{\partial}{\partial\eta^5},
      \frac{\partial}{\partial\eta^1}+z\frac{\partial}{\partial\eta^2}})\\
    &=a_4z^4+4a_3z^3+6a_2z^2+4a_1z+a_0,
      \end{aligned}
\]
where
\[a_4=W_{5225},\qquad a_3=W_{4225},\qquad a_2=W_{4125},\qquad a_1=W_{4124},\qquad a_0=W_{4114}.\]
Let us point out that the circle bundle $\cZ$ is not preserved by the action of the full structure group for the geometry of  (2,3,5)-distributions. In order to remedy this issue and define $\cZ$ one can   make  a choice of splitting for $\partial\cD$ given by $\partial\cD=\cD\oplus\langle\ell\rangle.$ Such splitting will reduce the structure group and allows one to define $\cZ$ invariantly. 
We will see in the next section that twistor distributions arising from pKE metrics are naturally equipped with such splitting therefore enable one to define  $\cZ.$ 
\end{remark}

\begin{remark}\label{rmk:parallel-conformal-spinor}   
  Using Theorem \ref{thm:Pawel-G2-confstr}, the rank 2 distribution $\cD=\Ker\{\eta^1,\eta^2,\eta^3\}$ is null with respect to the conformal structure  $[\tilde h]$. 
In fact, it has been shown \cite{HS-spinor} that $\cD$ induces a  \textrm{parallel spin tractor.}   Conversely,  it has been shown that  conformal structures of signature (3,2)  which are equipped with a parallel spin-tractor  arise from the construction of Theorem \ref{thm:Pawel-G2-confstr}.  The existence of such  parallel objects implies that  the conformal holonomy of the conformal structure is  a subgroup of $\mathbf{G}_2^*$ (see  \cite{HK-G2, HS-spinor} for more details.) This is an instance of an extension of   the  holonomy principle  in pseudo-Riemannian geometry, as explained in Remark \ref{rmk:pk-G-str}, to the context of Cartan geometries (see \cite{CS-Parabolic}.) 
\end{remark}


As was mentioned in  Remark~\ref{rmk:2approaches}, in order to find the Cartan connection \eqref{eq:G2-CarNur-Conn}, one can either work with a lifted coframe defined on the bundle $\cG$ or start with a choice of coframe on the manifold $Q$ and impose the structure equations to find the Cartan connection in terms of the coframe, which, if needed, can consequently be equivariantly lifted to $\cG.$ In this article we will follow the latter approach, as we did for the pKE structures.

\subsection{Null self-dual planes and a remarkable coincidence}
\label{sec:null-self-dual}

In this section we show the main result of this article by finding the Cartan connection of the twistor distribution on the space of self-dual null planes and  showing that the root type of its Cartan quartic is the same as the root type of  $Weyl^-.$ Furthermore, as a by-product of our construction, one obtains para-Sasaki-Einstein metrics in dimension five and 5-dimensional conformal structures with $\mathbf{SL}_3(\RR)$-holonomy.

\subsubsection{Twistor distribution on $\cO_+$}
\label{sec:2-3-5-dist-asd}

It was observed in \cite{AN-G2} that for a 4-dimensional conformal structure of split signature the circle bundles of self-dual and anti-self-dual null planes, $\cO_+$ and $\cO_-,$ are each equipped with a   naturally defined rank 2 distribution  which is  referred to as the \emph{twistor distribution.} The twistor distribution on $\cN_+$ and $\cN_-$ is (2,3,5)   in an open set $U\subset M$ if the self-dual and anti-self-dual Weyl curvature of the conformal structure is nowhere vanishing in $U$. 

Given a pKE structure, in order to define the twistor distribution on $\cO_+$ we make use of  the parametrization \eqref{eq:sd-asd-parametrized}, where $\mu\in \mathbb{R}\cup \{\infty\}$. As a result, on $\cO_+$ one obtains the \emph{0-adapted} coframe
\[\eta^1_0=\theta^1+\mu\theta^4,\quad \eta^2_0=\theta^2-\mu\theta^3,\quad \eta_0^3=\der\mu,\quad \eta^4_0=\theta^4,\quad\eta^5_0=\theta^3,\]
where the subscript $0$ refers to the adaptation with respect to  which the  2-distribution will be defined.

A coframe $(\eta^1,\dots,\eta^5)$ defines a (2,3,5) distribution $\cD=\Ker\{\eta^1,\eta^2,\eta^3\}$ if 
\begin{subequations}
  \label{eq:235cof1}
  \begin{align}
    \der\eta^1&\equiv \eta^3\dz\eta^4,\quad\mathrm{mod}\quad \{\eta^1,\eta^2\},\label{eq:235cof1-1}\\
    \der\eta^2&\equiv \eta^3\dz\eta^5,\quad\mathrm{mod}\quad \{\eta^1,\eta^2\},\label{eq:235cof1-2}\\
    \der\eta^3&\equiv\eta^4\dz\eta^5,\quad\mathrm{mod}\quad \{\eta^1,\eta^2,\eta^3\}.\label{eq:235cof1-3}
  \end{align}
\end{subequations}
To define the twistor distribution on $\cN$ we further adapt  the coframe $\{\eta^1_0,\dots,\eta^5_0\}$ so that it satisfies \eqref{eq:235cof1}.  Using the structure equations \eqref{pk}, one obtains
\[
  \begin{aligned}
    \der\eta^1_0&\equiv (\eta^3_0+\mu\Gamma^2_{~2}+\mu\Gamma^1_{~1})\dz\eta^4_0,
    \quad \mathrm{and}\quad
    \der\eta^2_0&\equiv -(\eta^3_0+\mu\Gamma^1_{~1}+\mu\Gamma^2_{~2})\dz\eta^5_0,
  \end{aligned}
\]
modulo $\{\eta^1_0,\eta^2_0\}.$
In order to obtain the relations \eqref{eq:235cof1-1},\eqref{eq:235cof1-2}   define the \emph{1-adapted} coframe as
\begin{equation}
  \label{eq:G2adaptedcof-1adap}
\eta^1_1=\theta^1+\mu\theta^4,\quad \eta^2_1=\theta^2-\mu\theta^3,\quad \eta_1^3=\der\mu+\mu\Gamma^1_{~1}+\mu\Gamma^2_{~2},\quad \eta^4_1=\theta^4,\quad\eta^5_1=-\theta^3
  \end{equation}
Using \eqref{pk}, the 1-adapted coframe satisfies
\begin{equation}
  \label{eq:G2adaptedcof-1ada-relations}
   \begin{aligned}
    \der\eta^1_1&\equiv \eta^3_1\dz\eta^4_1,\quad&&\mathrm{mod}\quad\{\eta^1_1,\eta^2_1\},\\
    \der\eta^2_1&\equiv \eta^3_1\dz\eta^5_1,\quad&&\mathrm{mod}\quad\{\eta^1_1,\eta^2_1\},\\
    \der\eta^3_1&\equiv -6\mu^2\,\Psi_2'\,\eta^4_1\dz\eta^5_1,\quad&&\mathrm{mod}\quad \{\eta^1_1,\eta^2_1,\eta^3_1\}.
  \end{aligned}
\end{equation}
Using the fact that $\Psi'_2\neq 0$ and $\der\Psi'_2=0,$ the 2-adapted coframe defined by
\begin{equation}
  \label{eq:G2adaptedcof}
  \begin{gathered}
    \textstyle{\eta^1_2=\frac{-1}{6\mu^2\,\Psi_2'}\left(\theta^1+\mu\theta^4\right),\qquad \eta^2_2=\frac{-1}{6\mu^2\,\Psi_2'}\left(\theta^2-\mu\theta^3\right)},\\ \textstyle{\eta_2^3=\frac{-1}{6\mu\,\Psi_2'}\left(\frac{\der\mu}{\mu}+\Gamma^1_{~1}+\Gamma^2_{~2}\right),\qquad \eta^4_2=\theta^4,\qquad\eta^5_2=-\theta^3}
  \end{gathered}
\end{equation}
  satisfies \eqref{eq:235cof1}, therefore, defines a (2,3,5)-distribution, 
  \begin{equation}
    \label{eq:TwistorDistribution}
\cD:=\mathrm{Ker}\{\eta^1_2,\eta^2_2,\eta^3_2\},    
  \end{equation}
  on $\cN_+$ for $\mu\in \RR^*$. It is straightforward to show that $\cD$ is invariant under the induced action of the  structure group $\glg_2(\RR)$.  We have the following theorem.
  \begin{theorem}
The  $\mathbb{S}^1$-bundle of self-dual null planes   of a pKE metric, $\cO_+,$  is naturally equipped with a rank 2 distribution $\cD.$   If the scalar curvature of the pKE metric is non-zero,  the twistor distribution $\cD$ is (2,3,5) on the  open subset  $\cC =\cO_+\backslash\{\cH,\bar\cH\}$ where $\{\cH,\bar\cH\}\subset\cO_+$ are the  $\pm 1$-eigenspaces of the para-complex structure $K,$ as defined in \eqref{eq:H_barH_null_distr}. In terms of  the affine parameter $\mu$ in \eqref{eq:sd-asd-parametrized}, the open subset $\cC$ corresponds to  $\mu\in \RR^*,$  and the  conformal structure of signature $(3,2)$ associated to the twistor distribution, as in Theorem \ref{thm:Pawel-G2-confstr},  is given by $[h]$ where 
\begin{equation}
  \label{eq:NurowskiMetric-Q-1}
  h=\textstyle{\frac{1}{6\mu^2\Psi'_2}(\theta^1\theta^3+\theta^2\theta^4)+\frac{1}{54\mu^2(\Psi'_2)^2}(\frac{\exd\mu}{\mu}+ \Gamma^1_{~1}+\Gamma^2_{~2})^2}.
\end{equation}
  \end{theorem}
  \begin{proof}
    Using the coframe \eqref{eq:G2adaptedcof-1adap}, the relations \eqref{eq:G2adaptedcof-1ada-relations} imply that the 2-plane field $\Ker\{\eta^1_1,\eta^2_1,\eta^3_1\}$ is integrable on the hypersurfaces corresponding to  $\mu=0$ and $\mu=\infty,$ which by \eqref{eq:sd-asd-parametrized} are given by $\mathrm{Ker}\{\alpha^1,\alpha^2\}$ and $\mathrm{Ker}\{\bar\alpha^1,\bar\alpha^2\}$ i.e. 2-plane fields $\cH$ and $\bar\cH$. For $\mu>0$ and $\mu<0$ the twistor distribution is (2,3,5) by \eqref{eq:TwistorDistribution}. Consequently, using the adapted coframe \eqref{eq:G2adaptedcof},  the 
metric \eqref{eq:Nurowski-Metric}     gives  \eqref{eq:NurowskiMetric-Q-1}.
  \end{proof}
  \begin{remark}  
The simple expression \eqref{eq:NurowskiMetric-Q-1} for the metric $h$ is the key to what follows in Theorem \ref{cor:MainResult}. For general twistor distributions the expression for $h$ involves the second jet of the self-dual Weyl curvature of $g,$ whose components are denoted by $\Psi'_i$'s.   However, in pKE metrics the only non-zero component of $Weyl^+$ is the constant $\Psi'_2.$ This point is explained further in  \ref{sec:2-3-5-ASD}, in particular Remark \ref{rmk:235-asd-weyl-bundle}.    
  \end{remark}
Using the theorem above, one obtains the following theorem which is the main result of this section.
\begin{theorem}\label{cor:MainResult}
  Given a pKE metric with non-zero scalar curvature,  the  Cartan quartic $C(z)$ for the twistor distribution $\cD\subset\rT\cN_+$ on  $\cC=\cN_+\backslash\{\cH,\bar\cH\}$ is a non-zero multiple of the  quartic representation of the anti-self-dual Weyl curvature $W(z)$. More  explicitly, one obtains 
\[C(z)=-6\mu^2\Psi_2'\left(\Psi_0+4\Psi_1z+6\Psi_2z^2+4\Psi_3z^3+\Psi_4z^4\right)=-6\mu^2\Psi'_2W(z).\]
In particular, the root types of the Cartan quartic and the anti-self-dual Weyl curvature coincide.
\end{theorem}
\begin{remark}\label{rmk:natural-identification}
We point out that there is an underlying bundle map behind Theorem \ref{cor:MainResult} which allows one to express $C(z)$ as  a non-zero multiple of $W(z)$. As discussed in \ref{sec:cart-penr-petr}, $W(z)$ is defined on $\cN_-$ and, by Remark \ref{rmk:235QuarticRedone},  $C(z)$ is defined on $\cZ.$ However, using the twistorial nature of $\cD$ one can naturally identify $\cN_-$ and $\cZ$ in the following way. First note that given a pKE structure, the derived system  $\partial\cD:=\Ker\{\eta^1_2,\eta^2_2\}$ is equipped with a splitting $\cD\oplus\langle\frac{\partial}{\partial\eta^3_2}\rangle$  which is invariant under the induced action of $\glg_2(\RR).$  Therefore, by Remark \ref{rmk:235QuarticRedone}, the circle bundle $\cZ,$ as defined in \eqref{eq:Zbundle}, is well-defined on $\cC$. Consequently, via the bundle map $\nu_+\colon\cC\to M,$ it is elementary to check that  $\exd\nu_+(\cZ_q)=\cN_{p-},$ where $p=\nu_+(q),$ for all $q\in\cZ,$ using  definitions   \eqref{eq:sd-asd-parametrized} and  \eqref{eq:Znull}. Therefore, the property of $W(z)$ and $C(z)$ being proportional everywhere is well-defined.     
\end{remark}
\begin{proof}[Proof of Theorem \ref{cor:MainResult}]
Having the twistor distribution $\cD$ defined by \eqref{eq:TwistorDistribution} on $\cC$, it is straightforward to find the explicit Cartan connection \eqref{eq:G2-CarNur-Conn} for $\cD$ starting with the 2-adapted coframe \eqref{eq:G2adaptedcof}. Consequently, one obtains that the Cartan connection \eqref{eq:G2-CarNur-Conn} for $\cD$  is given by
\begin{equation}
  \label{eq:G2CartanConn-SD-coframe}
  \begin{aligned}
   \eta^1&= \textstyle{\frac{-1}{6\mu^2\,\Psi_2'}\left(\theta^1+\mu\theta^4\right)},\qquad &&\eta^2=\textstyle{\frac{-1}{6\mu^2\,\Psi_2'}\left(\theta^2-\mu\theta^3\right)},\\ \eta^3&=\textstyle{\frac{-1}{6\mu^2\,\Psi_2'}\left(\der\mu+\mu\Gamma^1_{~1}+\mu\Gamma^2_{~2}\right)},\qquad &&\eta^4=\theta^4,\qquad&&\eta^5=-\theta^3,\\
       \zeta_1 &= \textstyle{ \frac{1}{3}(2\Gamma^1_{~1}-\Gamma^2_{~2}+2\frac{\der\mu}{\mu})},\quad &&\zeta_2 =\Gamma^1_{~2},\qquad  &&\zeta_3=\Gamma^2_{~1},\\
      \zeta_4 &=  \textstyle{\frac{1}{3}(2\Gamma^2_{~2}-\Gamma^1_{~1}+2\frac{\der\mu}{\mu}),}
    &&\zeta_5= -3\mu\Psi'_2\theta^3,\qquad &&\zeta_6 = -3\mu\Psi'_2\theta^4,\\
    \zeta_7& = 0,\qquad&&\zeta_8= 6\mu^2(\Psi'_2)^2\theta^3,\qquad &&\zeta_9 = 6\mu^2(\Psi'_2)^2\theta^4.
  \end{aligned}
\end{equation} 
 Using the Cartan connection, one  computes the curvature \[K_{\bG^*_2}=\exd\omega_{\bG^*_2}+\omega_{\bG^*_2}\w\omega_{\bG^*_2}.\]
Therefore, the coefficients $a_0,\dots,a_4$ of the Cartan quartic \eqref{eq:CartanQuartic-C} are found to be 
\begin{equation}
  \label{eq:CartanQuartic-SD}
  a_i=-6\mu^2\Psi'_2 \Psi_i,\quad i=0,\dots,4.
\end{equation}
 By the formulas for $C(z)$ and $W(z),$ as given in \eqref{eq:CartanQuartic-SD} and \eqref{eq:quartic2}, it follows that 
\[C(z)=-6\mu^2\Psi'_2W(z).\]
\end{proof}
 \begin{remark}
As was mentioned in the introduction, it was not known whether (2,3,5) distributions that arise as a twistor distribution of split signature metrics can have any fixed (Petrov) root type. Theorem  \ref{cor:MainResult} shows that any root type can be achieved via pKE metrics  with non-zero scalar curvature for which $Weyl^-$ has the same root type. Moreover,   our examples in \ref{sec:cart-reduct-symm} provide explicit metrics whose twistor distributions have Cartan quartics of  real root type $II,III,N,D$ and $O.$
\end{remark}

\subsubsection{An invariant description}\label{sec:einst-metr-mathbfsl3}
To have an invariant understanding of $\cC=\cN_+\backslash\{\cH,\bar\cH\},$ we define another space equipped with a rank 2 distribution which will be shown to be isomorphic to the twistor distribution on $\cC.$  Consider the principal $\mathbf{GL}_2(\RR)$-bundle $\cF^8$ of adapted null coframes for pKE metrics, equipped with the Cartan connection $\cA,$ as in \eqref{eq:SL3-pke-conn-A}. Using the structure equations \eqref{pk}, one can define the 5-dimensional leaf space, $\cQ,$ of the Pfaffian system 
\[I=\{\theta^1,\theta^2,\theta^3,\theta^4,\Gamma^1_{~1}+\Gamma^2_{~2}\}.\]
Note that by the structure equations \eqref{pk} the Pfaffian system $I$ is integrable and its 5-dimensional leaf space $\cQ$ is the quotient of $\cF^8$ by the orbits of  $\mathbf{SL}_2(\RR)\subset \mathbf{GL}_2(\RR),$ which would give an $\RR^*$-bundle over $M$. As a result,     $\cQ$ is a cone.  Similar to our previous discussion, if $\Psi'_2\neq 0,$ after necessary adaptation, one obtains  the following   coframe on $\cQ$  
\begin{equation}
  \label{eq:G2adaptedcof-Q}
  \begin{gathered}
    \textstyle{\eta^1=\frac{-1}{6\,\Psi_2'}\left(\theta^1+\theta^4\right),\quad \eta^2=\frac{-1}{6\,\Psi_2'}\left(\theta^2-\theta^3\right)},\\ \textstyle{\eta^3=\frac{-1}{6\,\Psi_2'}\left(\Gamma^1_{~1}+\Gamma^2_{~2}\right),\quad
 \eta^4=\theta^4,\quad\eta^5=-\theta^3},
  \end{gathered} 
\end{equation}    
which defines a rank 2 distribution  $\tilde \cD\subset\rT\cQ,$ by  $\tilde \cD=\Ker\{\eta^1,\eta^2,\eta^3\}.$ The distribution $\tilde \cD$ is invariant under  the induced action of the structure group  $\mathbf{GL}_2(\RR).$ 

The Cartan connection for the $\tilde \cD$ is obtainable from  \eqref{eq:G2CartanConn-SD-coframe} by setting $\mu=1.$    In the language of parabolic geometries \cite{CS-Parabolic}, this is the explicit form of the extension functor from a pKE structure to the corresponding (2,3,5)-distribution. 

To relate our discussion to  $\cC=\cN_+\backslash\{\cH,\bar\cH\},$  note that the structure group of a pKE metric can be decomposed as 
\[\mathbf{GL}_2(\RR)=\mathbf{SL}_2(\RR)\times\RR^*.\]
The $\RR^*$-action  transforms the pKE coframe in the following way
\[\theta^1\rightarrow s\theta^1,\quad \theta^2\rightarrow s\theta^2,\quad\theta^3\rightarrow \textstyle{\frac{1}{s}\theta^3},\quad\textstyle{\theta^1\rightarrow \frac{1}{s}\theta^4}, \] 
by setting $a_{11}=a_{22}=s$ and $a_{12}=a_{21}=0$ in the structure group \eqref{ga2}. This action results in the following change of coframe \eqref{eq:G2adaptedcof} on  $\cC$
\begin{equation}
  \label{eq:G2adaptedcof-ScaleBundle}
  \begin{gathered}
    \textstyle{\eta^1_2=\frac{-1}{6\mu^2\,\Psi_2'}\left(s\theta^1+\frac{\mu}{s}\theta^4\right),\qquad \eta^2_2=\frac{-1}{6\mu^2\,\Psi_2'}\left(s\theta^2-\frac{\mu}{s}\theta^3\right)},\\ \textstyle{\eta_2^3=\frac{-1}{6\mu\,\Psi_2'}\left(\frac{\der\mu}{\mu}-\frac{\exd s}{s}+\Gamma^1_{~1}+\Gamma^2_{~2}\right),\qquad \eta^4_2=\frac{1}{s}\theta^4,\qquad\eta^5_2=-\frac{1}{s}\theta^3}. 
  \end{gathered}
\end{equation}
  Using the expressions \eqref{eq:G2adaptedcof-ScaleBundle}  
one  finds the bundle isomorphism $\tau\colon \cC\to\cQ$ in terms of the fiber coordinates  $s$ and $\mu,$ given by $\tau(\mu)=\frac{1}{s^2}$ for $s>0$ and $\tau(\mu)=-\frac{1}{s^2}$ for $s<0,$ via which $\tilde\cD=\tau_{*}\cD.$ This gives the equivalence of the (2,3,5) distributions induced on $\cC$ and $\cQ$
  
\subsubsection{para-Sasaki-Einstein structures and $\mathbf{SL}_3(\RR)$ holonomy}
\label{sec:para-sasaki-einstein}
To state the last result of this section we give the following definition of a \emph{para-Sasaki-Einstein structure}.
 \begin{definition}\label{def:para-sasaki-einstein-1}
A para-Sasaki-Einstein structure on $\cQ$ is $(\phi,\xi,\beta,h),$ where $\phi:\rT\cQ\to\rT\cQ$ is an endomorphism, $\xi$ is  a vector field, $\beta$ is a 1-form, and $h$ is a split signature metric. The quadruple $(\phi,\xi,\beta,h)$ satisfies
\begin{equation}
  \label{eq:Sasaki-Conditions1}
  \phi^2=\mathrm{Id}-\beta\otimes\xi,\quad \beta(\xi)=1,\quad \phi(\xi)=0,\quad \beta\circ \phi=0,\quad h(\xi,.)=\beta.
\end{equation}
Additionally,  the $\pm 1$-eigenspaces of $\phi$ define rank 2 integrable  sub-distributions of $\mathrm{Ker}(\beta),$ the metric $h$ is Einstein, and
 the compatibility conditions
\begin{equation}
  \label{eq:sasaki-Cond2}
  \begin{aligned}
  h(\phi X,\phi Y)&=-h(X,Y)+ \beta(X)\beta(Y),\\
  \exd \beta(X,Y)&=h(\phi X,Y)\qquad &\forall X,Y\in\rT\cQ
  \end{aligned}
\end{equation}
hold.   
 \end{definition}
We have the following proposition.
\begin{proposition}
  The cone $\cQ$ is equipped with a para-Sasaki-Einstein structure arising from the para-K\"ahler-Einstein structure on $M.$ 
\end{proposition}
\begin{proof}
Using the coframe \eqref{eq:G2adaptedcof-Q}, the claimed para-Sasakian structure on $\cQ$ is given by 
\begin{equation*}
  \begin{aligned}
\phi&=\textstyle{\eta^1\otimes\frac{\partial}{\partial\eta^1}+\eta^2\otimes\frac{\partial}{\partial\eta^2}-\eta^4\otimes\frac{\partial}{\partial\eta^4}-\eta^5\otimes\frac{\partial}{\partial\eta^5}}\\
    \xi&=\textstyle{\sqrt{\frac 3 2}\frac{\partial}{\partial\eta^3}}\\
    \beta&=\textstyle{\sqrt{\frac 2 3}\eta^3}\\
    h&=\eta^1\eta^5-\eta^2\eta^4+\textstyle{\frac 23}(\eta^3)^2
  \end{aligned}
\end{equation*}
  The Levi-Civita connection of $h$ with respect to the coframe \eqref{eq:G2adaptedcof-Q} is 
 \begin{equation}
   \label{eq:so5-LC}
    \def\arraystretch{1.3}
   \begin{pmatrix}
     -\Gamma^2_{~2}-4\Psi'_2\eta^3& \Gamma^1_{~2}& 2\Psi'_2\eta^1+\frac{2}{3}\eta^4& -\frac 13\eta^3&  0\\
\Gamma^2_{~1} & \Gamma^2_{~2}+2\Psi'_2\eta^3&  2\Psi'_2\eta^2+\frac 23\eta^5 & 0 & -\frac 13\eta^3\\
\frac 32\Psi'_2\eta^5 & -\frac 32\Psi'_2\eta^4 & 0& \frac 32\Psi'_2\eta^2+\frac 12\eta^5 & -\frac 32\Psi'_2\eta^1-\frac 12\eta^4\\
0 & 0 & -2\Psi_2'\eta^4 & -2\Psi'_2\eta^3-\Gamma^2_{~2} &\Gamma^1_{~2}\\
0& 0& -2\Psi'_2\eta^5 & \Gamma^2_{~1} & 4\Psi'_2\eta^3+\Gamma^2_{~2}
\end{pmatrix}
 \end{equation}
which is   $\mathfrak{so}_{3,2}$-valued  with respect to the matrix representation of the metric $h$
\[\begin{pmatrix}
  0 & 0 & 0 & 0 & 1\\
 0 & 0 & 0 & -1 & 0\\
 0 & 0 & \frac 43 & 0 & 0\\
0 & -1 & 0&0&0\\
1 &0&0&0&0
\end{pmatrix}.
\] 
It follows that the metric $h$ on $\cQ$ is Einstein whose   Einstein constant is   $-24(\Psi'_2)^2.$
 \end{proof}

  \begin{remark}
The proposition above is an analogue of the well-known construction of Sasakian structures from   K\"ahler metrics \cite{GN-sasaki}. Moreover, by the same analogy,  following \cite{GN-sasaki}, it can be  checked that  using the potential function of a para-K\"ahler-Einstein metric and the coordinate $s$ on the cone, as introduced in \ref{sec:einst-metr-mathbfsl3}, one obtains a  potential function for the resulting para-Sasaki-Einstein metric on $\cQ.$ 
\end{remark}

Finally one can directly verify the well-known relation between   the existence of an Einstein representative in  the conformal structure $[h]$ of the metric \eqref{eq:NurowskiMetric-Q-1} defined by  the (2,3,5) distribution, and  the holonomy reduction of the Cartan conformal connection.  This relation goes back to  works of many people, including \cite{Sasaki-holonomy} (see Section 5.2 in  \cite{CS-Parabolic}  for an overview.)  One can check that the Cartan connection \eqref{eq:G2-CarNur-Conn} given by \eqref{eq:G2CartanConn-SD-coframe} takes value in $\mathfrak{sl}_3(\RR)\subset \mathfrak{g}_2\subset \mathfrak{so}_{4,3}$ for any non-zero value of $\Psi'_2.$ This implies that the conformal holonomy of the conformal structure $[h]$ is reduced and the conformal class $[h]$ must contain at least one Einstein metric.  We refer the reader to \cite{SW-G2} for a more detailed study of conformally Einstein structures arising from (2,3,5) distributions and their corresponding conformal holonomy reductions.

\subsection{Null anti-self-dual planes}  
\label{sec:2-3-5-ASD}
The purpose of this short section is to justify the remarkable nature of Theorem \ref{cor:MainResult} by examining the twistor distribution induced on the $\mathbb{S}^1$-bundle of anti-self-dual null planes of a pKE metric, $\cO_-.$ 
 Following the description in \ref{sec:einst-metr-mathbfsl3}, one can identify $\cO_-$ as the leaf space of the Pfaffian system 
\[I_{asd}=\{\theta^1,\theta^2,\theta^3,\theta^4,\Gamma^1_{~2}\}.\]
Consequently, if $Weyl^-\neq 0,$ one can follow the discussion in \ref{sec:2-3-5-dist-asd} to find an adapted coframe. 
An adapted coframe for the twistor distribution satisfying \eqref{eq:235cof1} on the open set of $\cO_-$ where $\Psi_4\neq 0$ is given by 
\begin{equation*}
  \label{eq:G2-CartanConn-ASD-coframe}
      \begin{aligned}
    \eta^1&=\textstyle{\frac{1}{\Psi_4}(\alpha^1+\bar\alpha^2)\quad  
     \eta^2=\frac{1}{\Psi_4}\bar\alpha^2,\quad\eta^3=\frac{1}{\Psi_4}\Gamma^1_2},\\
   \eta^4&=\textstyle{\bar\alpha^1-\alpha^2-\frac{30\Psi_3\Psi_4^2+3\Psi_4\Psi_{422}-\Psi_4\Psi_{433}- 5\Psi_{42}^2+4\Psi_{43}^2}{30\Psi_4^2}\alpha^1} \\
  &\textstyle{- \frac{15\Psi_3\Psi_4^2-3\Psi_4\Psi_{432}-3\Psi_4\Psi_{433}- 5\Psi_{42}\Psi_{43}+5\Psi_{43}^2}{15\Psi_4^2}\bar\alpha^2+\frac{\Psi_{42}+\Psi_{43}}{3\Psi_4^2}\Gamma^1_2 }\\
 \eta^5&=\textstyle{\bar\alpha^1-\frac{3\Psi_4\Psi_{422}-5\Psi_{42}^2}{30\Psi_4^3}\alpha^1-\frac{\Psi_{42}}{3\Psi_{4}^2}\Gamma^1_2 }\\
 &\textstyle{-\frac{30\Psi_3\Psi_4^2-6\Psi_4\Psi_{432}-3\Psi_4\Psi_{433}+10\Psi_{43}\Psi_{42}+5\Psi_{43}^2}{30\Psi_4^3}\bar\alpha^2}.
\end{aligned}
\end{equation*}
Note that if $Weyl^-\neq 0,$ one can always find a coframe with respect to which $\Psi_4\neq 0.$  

Consequently, the expressions for the coefficients of the Cartan quartic,  $a_0,\dots,a_4$, are found to be very complicated and depend on the 4th jet of  $\Psi_i$'s. For instance, in the coframe introduced above, one obtains
\begin{equation}
  \label{eq:CartanQuartic-ASD}
  \begin{aligned}
  a_0&=\textstyle{ -\frac{10\Psi_4^3 \Psi_{42222}-70\Psi_4^2\Psi_{42}\Psi_{4222}- 49\Psi_4^2\Psi_{422}^2+ 280\Psi_4\Psi_{42}^2\Psi_{422}- 175\Psi_{42}^4} {100\Psi_4^4} }.
  \end{aligned}
\end{equation} 
The equation arising from the vanishing of $a_0$  is referred to as \emph{Noth's equation}  and its general solution can be presented by a certain family of rational sextics \cite{AN-Noth, DS-Noth}.  
 
We could not identify any relation between the root types of $Weyl^\pm$ and the root type of the Cartan quartic of the twistor distribution on $\cN_-,$ with the exception of the homogeneous pKE metrics of Petrov type $D^r$ where they coincide.  
\begin{remark}\label{rmk:235-asd-weyl-bundle}  
  Following our approach in this section, one can find the Cartan quartic for the twistor distribution  on  $\cN_-$ and $\cN_+$ for any conformal structure of split signature  provided that  $Weyl^-\neq 0$ and $Weyl^+\neq 0,$ respectively. The computations are extremely tedious. Restricting to $\cN_-$, if $Weyl^-\neq 0,$ one obtains that   the expression for the metric $h,$  as defined in \eqref{eq:Nurowski-Metric}, involves the second jets of $Weyl^-.$ Furthermore, in an appropriate coframe,   $a_i$'s can be expressed in terms of the fourth jets of $Weyl^-$ and zeroth jets of $Weyl^+$ at each point. For instance, consider conformal structures  $[g],$ where $g=2\theta^1\theta^3+2\theta^2\theta^4,$  with $Weyl^-\neq 0$ and denote the components of $Weyl^-$ and $Weyl^+$ by $\Psi_i$'s and $\Psi'_i$'s respectively, as we did in \ref{sec:structure-equations}.
 Then on the open subset of $\cN_-$ where $\Psi_4\neq 0,$ there is an adapted coframe for the twistor distribution with respect to which  
 \begin{equation*} 
 a_0=\textstyle{ -\Psi'_0-\frac{10\Psi_4^3 \Psi_{42222}-70\Psi_4^2\Psi_{42}\Psi_{4222}- 49\Psi_4^2\Psi_{422}^2+ 280\Psi_4\Psi_{42}^2\Psi_{422}- 175\Psi_{42}^4} {100\Psi_4^4} }. 
\end{equation*}
Comparing the expression above to  \ref{eq:CartanQuartic-ASD}, it is clear that the case of twistor distribution induced on $\cN_-$ for pKE metrics is nearly as complicated as in the case of  general conformal structures. As a result, one cannot expect any relation between the root types of $Weyl^\pm$ for a conformal structure and the Cartan quartic of the twistor distribution on $\cN_\pm.$   
\end{remark}

 \appendix
\setcounter{equation}{0}
\setcounter{subsection}{0}
 \setcounter{theorem}{0}

\section*{Appendix}
\renewcommand{\theequation}{A.\arabic{equation}}
\renewcommand{\thesection}{A}

The connection 1-form referred  to in \ref{sec:pk-struct-coord} are given by 
\begin{equation}
  \label{eq:Conn1Form-LocCoord}
  \begin{aligned}
    \Gamma^a{}_a=&\frac{V_{aay}V_{bx}-V_{aax}V_{by}}{V_{ay}V_{bx}-V_{ax}V_{by}}\der a+\frac{V_{aby}V_{bx}-V_{abx}V_{by}}{V_{ay}V_{bx}-V_{ax}V_{by}}\der b\\
    \Gamma^a{}_b=&\frac{V_{aby}V_{bx}-V_{abx}V_{by}}{V_{ay}V_{bx}-V_{ax}V_{by}}\der a+\frac{V_{bby}V_{bx}-V_{bbx}V_{by}}{V_{ay}V_{bx}-V_{ax}V_{by}}\der b\\
    \Gamma^b{}_a=&\frac{V_{aay}V_{ax}-V_{aax}V_{ay}}{-V_{ay}V_{bx}+V_{ax}V_{by}}\der a+\frac{V_{aby}V_{ax}-V_{abx}V_{ay}}{-V_{ay}V_{bx}+V_{ax}V_{by}}\der b\\
     \Gamma^b{}_b=&\frac{V_{aby}V_{ax}-V_{abx}V_{ay}}{-V_{ay}V_{bx}+V_{ax}V_{by}}\der a+\frac{V_{bby}V_{ax}-V_{bbx}V_{ay}}{-V_{ay}V_{bx}+V_{ax}V_{by}}\der b\\
     \Gamma^x{}_x=&\frac{V_{bxx}V_{ay}-V_{axx}V_{by}}{V_{ay}V_{bx}-V_{ax}V_{by}}\der x+\frac{V_{bxy}V_{ay}-V_{axy}V_{by}}{V_{ay}V_{bx}-V_{ax}V_{by}}\der y\\
     \Gamma^x{}_y=&\frac{V_{bxy}V_{ay}-V_{axy}V_{by}}{V_{ay}V_{bx}-V_{ax}V_{by}}\der x+\frac{V_{byy}V_{ay}-V_{axy}V_{by}}{V_{ay}V_{bx}-V_{ax}V_{by}}\der y\\
     \Gamma^y{}_x=&\frac{V_{bxx}V_{ax}-V_{axx}V_{bx}}{-V_{ay}V_{bx}+V_{ax}V_{by}}\der x+\frac{V_{bxy}V_{ax}-V_{axy}V_{bx}}{-V_{ay}V_{bx}+V_{ax}V_{by}}\der y\\
     \Gamma^y{}_y=&\frac{V_{bxy}V_{ax}-V_{axy}V_{bx}}{-V_{ay}V_{bx}+V_{ax}V_{by}}\der x+\frac{V_{byy}V_{ax}-V_{ayy}V_{bx}}{-V_{ay}V_{bx}+V_{ax}V_{by}}\der y.
    \end{aligned}
\end{equation}
If $(M,g,K)$ is a para-K\"ahler-Einstein  structure written in a null adapted frame satisfying \eqref{pk} then the derivatives of the non-vanishing curvature coefficients are given by
\begin{subequations}\label{bi}
  \begin{align}
&\der\Psi_2'=0,\label{eq:bi1}\\
&\der\Psi_0=2\Psi_0(\Gamma^1{}_1-\Gamma^2{}_2)+4\Psi_1\Gamma^2{}_1+\Psi_{01}\al^1+\Psi_{11}\al^2-\Psi_{14}\bal^1+\Psi_{04}\bal^2, \label{eq:bi2}\\
&\der\Psi_1=\Psi_1(\Gamma^1{}_1-\Gamma^2{}_2)+\Psi_0\Gamma^1{}_2+3\Psi_2\Gamma^2{}_1+\Psi_{11}\al^1+\Psi_{21}\al^2-\Psi_{24}\bal^1+\Psi_{14}\bal^2,\label{eq:bi3}\\
&\der\Psi_2=2\Psi_1\Gamma^1{}_2+2\Psi_3\Gamma^2{}_1+\Psi_{21}\al^1+\Psi_{31}\al^2-\Psi_{34}\bal^1+\Psi_{24}\bal^2,\label{eq:bi3}\\
&\der\Psi_3=-\Psi_3(\Gamma^1{}_1-\Gamma^2{}_2)+3\Psi_2\Gamma^1{}_2+\Psi_4\Gamma^2{}_1+\Psi_{31}\al^1+\Psi_{41}\al^2-\Psi_{44}\bal^1+\Psi_{34}\bal^2,\label{eq:bi4}\\
&\der\Psi_4=-2\Psi_4(\Gamma^1{}_1-\Gamma^2{}_2)+4\Psi_3\Gamma^2{}_1+\Psi_{41}\al^1+\Psi_{42}\al^2+\Psi_{43}\bal^1+\Psi_{44}\bal^2, \label{eq:bi5}
\end{align}
\end{subequations}
for some functions $\Psi_{ia}$ on $M$ which represent the coframe derivatives of $\Psi_i$'s.
 
The differential relations among the quantities appearing in proof of the  Proposition \ref{rei} for Petrov type $II$ is as follows.
\begin{equation}
  \label{eq:Bianchies-TypeIID-1}
\begin{aligned}
  \der J_1=&J_1\Gamma^1_{~1}-J_1^2\al^1+J_{12}\al^2+J_{13}\bal^1+J_1J_2\bal^2,\\
  \der J_2=&-J_2\Gamma^2_{~2}-J_1J_2\al^1+(J_{13}+\Psi_2-\Psi_2')\al^2+J_{23}\bal^1+J_2^2\bal^2,\\
  \der J_3=& J_3\Gamma^2_{~2}+J_{31}\al^1+J_{32}\al^2+J_{33}\bal^1+J_{34}\bal^2,\\
  \der J_4=&-J_4\Gamma^1_{~1}-(J_{34}-\Psi_2+\Psi_2')\al^1+J_{42}\al^2+(J_4^2-J_2J_5+J_{54})\bal^1+J_{44}\bal^2,\\
  \der J_5=&- 2J_5\Gamma^1_{~1}+J_5\Gamma^2_{~2}-(J_{33}-J_3J_4+J_2J_6)\al^1+J_{52}\al^2+J_{53}\bal^1+J_{54}\bal^2,\\
  \der J_6=& -J_6\Gamma^1_{~1}+2J_6\Gamma^2_{~2}-(J_{32}-J_3^2+J_1J_6)\al^1+J_{62}\al^2+\\&(J_{52}-J_3J_5+J_4J_6+\Psi_4)\bal^1+(J_{42}+J_1J_5-J_3J_4)\bal^2,\\
    \der\Psi_2'=&0,\\
    \der\Psi_2=&-3J_1\Psi_2\al^1+3J_3\Psi_2\al^2+3J_4\Psi_2\bal^1+3J_2\Psi_2\bal^2,\\
    \der\Psi_4=&-2\Psi_4\Gamma^1_{~1}+2\Psi_4 \Gamma^2_{~2}-(3J_6\Psi_2+J_1\Psi_4)\al^1+\Psi_{42}\al^2+\Psi_{43}\bal^1+(3J_5\psi_2+J_2\Psi_4)\bal^2.
  \end{aligned}
\end{equation} 
 The curvature 2-form for the Cartan connection $\cB$ in \eqref{cacoB} is given by
\begin{equation}
  \label{eq:typeII-curv-YM-1}
  \tiny{\begin{aligned}
  K_B&=\\&\bma[c|c]
\begin{matrix}\tfrac12(J_2J_6-J_1J_5)&-\sqrt{\tfrac32|\Psi_2'|}~J_5\\
0&\tfrac12(J_1J_5-J_2J_6)\end{matrix}&0\\
\cmidrule(lr){1-2}
0&\begin{matrix}\tfrac12(J_1J_5-J_2J_6)&0\\
  -\sqrt{\tfrac32|\Psi_2'|}~J_6&\tfrac12(J_2J_6-J_1J_5)\end{matrix}\ema\sigma^2_-+\\
&\bma[c|c]
\begin{matrix}\tfrac14(J_2J_3-J_1J_4-2\Psi_2+2\Psi_2')&-\sqrt{\tfrac38|\Psi_2'|}~J_4\\
-\sqrt{\tfrac38|\Psi_2'|}~J_1&\tfrac14(J_1J_4-J_2J_3+2\Psi_2-2\Psi_2')\end{matrix}&0\\
\cmidrule(lr){1-2}
0&\begin{matrix}\tfrac14(J_1J_4-J_2J_3+2\Psi_2-2\Psi_2')&\sqrt{\tfrac38|\Psi_2'|}~J_2\\
  -\sqrt{\tfrac38|\Psi_2'|}~J_3&\tfrac14(J_2J_3-J_1J_4-2\Psi_2+2\Psi_2')\end{matrix}\ema\sigma^3_-+\\
&\bma[c|c]
\begin{matrix}\tfrac12 J_1J_3&\sqrt{\tfrac32|\Psi_2'|}~J_3\\
0&-\tfrac12J_1J_3\end{matrix}&0\\
\cmidrule(lr){1-2}
0&\begin{matrix}-\tfrac12J_1J_3&\sqrt{\tfrac32|\Psi_2'|}~J_1\\
  0&\tfrac12J_1J_2\end{matrix}\ema\sigma^1_++
\bma[c|c]
\begin{matrix}\tfrac12J_2J_4&0\\
\sqrt{\tfrac32|\Psi_2'|}~J_2&-\tfrac12J_2J_4\end{matrix}&0\\
\cmidrule(lr){1-2}
0&\begin{matrix}-\tfrac12J_2J_4&0\\
  -\sqrt{\tfrac32|\Psi_2'|}~J_4&\tfrac12J_2J_4\end{matrix}\ema\sigma^2_++\\
&\bma[c|c]
\begin{matrix}\tfrac14(J_2J_3+J_1J_4)&\sqrt{\tfrac38|\Psi_2'|}~J_4\\
\sqrt{\tfrac38|\Psi_2'|}~J_1&-\tfrac14(J_2J_3+J_1J_4)\end{matrix}&0\\
\cmidrule(lr){1-2}
0&\begin{matrix}-\tfrac14(J_2J_3+J_1J_4)&\sqrt{\tfrac38|\Psi_2'|}~J_2\\
-\sqrt{\tfrac38|\Psi_2'|}~J_3&\tfrac14(J_2J_3+J_1J_4)\end{matrix}\ema\sigma^3_+.
\end{aligned}}
\end{equation}

\bibliographystyle{alpha}   
\bibliography{pKE2}

\end{document}